\documentclass[11pt]{article}

\usepackage[utf8]{inputenc}
\usepackage[british]{babel}
\usepackage[a4paper, margin=2.5cm]{geometry}
\usepackage{amsmath, amssymb, amsthm, amsfonts}
\usepackage{mathabx}
\usepackage{mathtools}
\usepackage{thmtools}
\usepackage[shortlabels]{enumitem}
\usepackage[font={small, it}]{caption}
\usepackage{subcaption}
\usepackage{microtype}
\usepackage{xcolor}
\usepackage{multicol}
\usepackage{tikz}
\usepackage[backref]{hyperref}

\usetikzlibrary{calc,decorations.pathreplacing,decorations.pathmorphing,decorations.markings,shapes}

\newcounter{sarrow}
\newcommand\xrsquigarrow[1]{%
  \stepcounter{sarrow}
  \mathrel{
    \begin{tikzpicture}[decoration=snake]
      \node[minimum size=1cm] (\thesarrow) {\strut#1};
      \draw[->,decorate] (\thesarrow.south west) -- (\thesarrow.south east);
    \end{tikzpicture}%
  }
}

\newcommand{\cE}{\ensuremath{\mathcal E}}
\newcommand{\cF}{\ensuremath{\mathcal F}}
\newcommand{\calG}{\ensuremath{\mathcal G}}
\newcommand{\cH}{\ensuremath{\mathcal H}}
\newcommand{\calI}{\ensuremath{\mathcal I}}

\newcommand{\calL}{\ensuremath{\mathcal L}}

\newcommand{\calP}{\ensuremath{\mathcal P}}

\newcommand{\calT}{\ensuremath{\mathcal T}}

\newcommand{\eps}{\varepsilon}
\renewcommand{\phi}{\varphi}

\DeclareMathOperator*{\E}{\mathbb{E}}
\DeclareMathOperator*{\N}{\mathbb{N}}
\DeclareMathOperator*{\Z}{\mathbb{Z}}
\DeclareMathOperator*{\R}{\mathbb{R}}

\DeclarePairedDelimiter\ceil{\lceil}{\rceil}
\DeclarePairedDelimiter\floor{\lfloor}{\rfloor}

\let\setminus=\smallsetminus

\newcommand{\osref}[2]{%
  \setlength\abovedisplayskip{5pt plus 2pt minus 2pt}
  \setlength\abovedisplayshortskip{5pt plus 2pt minus 2pt}
  \ensuremath{\overset{\text{#1}}{#2}}
}

\newcommand{\Gnp}{G_{n, p}}
\newcommand{\Bin}{\ensuremath{\mathrm{Bin}}}
\newcommand{\bw}{\ensuremath{\mathbf{w}}}
\newcommand{\obw}{\ensuremath{\overline{\mathbf{w}}}}
\newcommand{\bx}{\ensuremath{\mathbf{x}}}

\newcommand{\by}{\ensuremath{\mathbf{y}}}
\newcommand{\oby}{\ensuremath{\overline{\mathbf{y}}}}

\newcommand{\binoms}[2]{{\textstyle\binom{#1}{#2}}} 

\colorlet{royalred}{red!70!black}
\definecolor{royalazure}{rgb}{0.0, 0.22, 0.66}

\declaretheorem[parent=section]{theorem}
\declaretheorem[sibling=theorem]{lemma}
\declaretheorem[sibling=theorem]{proposition}
\declaretheorem[sibling=theorem]{claim}
\declaretheorem[sibling=theorem]{corollary}

\declaretheorem[sibling=theorem,style=definition]{definition}

\setlist{itemsep=0.2em, topsep=0.2em, parsep=0.1em, partopsep=0.1em}

\hypersetup{
  colorlinks,
  linkcolor={red!60!black},
  citecolor={green!50!black},
  urlcolor={blue!80!black}
}

\newlength{\bibitemsep}
\newlength{\bibparskip}
\let\oldthebibliography\thebibliography
\renewcommand\thebibliography[1]{%
  \oldthebibliography{#1}%
  \setlength{\parskip}{\bibitemsep}%
  \setlength{\itemsep}{\bibparskip}%
}

\title{Triangle resilience of the square of a Hamilton cycle \\ in random graphs}
\author{
  Manuela Fischer%
  \thanks{Institute of Theoretical Computer Science, ETH
  Z\"{u}rich, 8092 Z\"{u}rich, Switzerland
  \newline
  Email:
  \{\texttt{manuela.fischer}\textbar \texttt{nskoric}\textbar
  \texttt{steger}\textbar \texttt{mtrujic}\}\texttt{@inf.ethz.ch}}
  \and
  Nemanja \v{S}kori\'{c}\footnotemark[1]
  \and
  Angelika Steger\footnotemark[1]
  \and
  Milo\v{s} Truji\'{c}\footnotemark[1] \textsuperscript{,}\thanks{author was
  supported by grant no.\ 200021 169242 of the Swiss National Science
  Foundation.}
}
\date{}

\begin{document}
\maketitle

\begin{abstract}
  Since first introduced by Sudakov and Vu in 2008, the study of resilience
  problems in random graphs received a lot of attention in probabilistic
  combinatorics. Of particular interest are resilience problems of spanning
  structures. It is known that for spanning structures which contain many
  triangles, local resilience cannot prevent an adversary from destroying all
  copies of the structure by removing a negligible amount of edges incident to
  every vertex. In this paper we generalise the notion of local resilience to
  $H$-resilience and demonstrate its usefulness on the containment problem of
  the square of a Hamilton cycle. In particular, we show that there exists a
  constant $C > 0$ such that if $p \geq C\log^3 n/\sqrt{n}$ then w.h.p.\ in
  every subgraph $G$ of a random graph $\Gnp$ there exists the square of a
  Hamilton cycle, provided that every vertex of $G$ remains on at least a $(4/9
  + o(1))$-fraction of its triangles from $\Gnp$. The constant $4/9$ is optimal
  and the value of $p$ slightly improves on the best-known appearance threshold
  of such a structure and is optimal up to the logarithmic factor.
\end{abstract}

\section{Introduction}

One of the central questions of extremal graph theory concerns determining
sufficient conditions for the containment of (spanning) structures. Some of the
most influential examples, dating back to the middle of the previous century,
include Tur\'{a}n's theorem~\cite{turan1941extremalaufgabe} and Dirac's
theorem~\cite{dirac1952some}. The former states that having more than
$\floor{n^2/4}$ edges in a graph with $n$ vertices is sufficient in order for a
triangle to exist, while the latter states that a graph with minimum degree
$\ceil{n/2}$ is Hamiltonian. Several years later, first
P\'{o}sa~\cite{erdos1964problem}, and then Seymour~\cite{seymour1973problem},
conjectured that for any integer $k \geq 2$, a graph $G$ with $n$ vertices and
minimum degree $\delta(G) \geq kn/(k + 1)$ contains the $k$-{\em th power of a
Hamilton cycle}. For a cycle $C$ and an integer $k \in \N$, the $k$-{\em th
power of a cycle} ($k$-{\em cycle} for short) is obtained by including an edge
between all pairs of vertices with distance on $C$ of at most $k$. The second
power of a cycle is also called the {\em square of a cycle}. It required the
development of powerful tools, most notably Szemer\'{e}di's regularity lemma and
the blow-up lemma, before this conjecture was finally proven by Koml\'{o}s,
S\'{a}rk\"{o}zy, and Szemer\'{e}di~\cite{komlos1998proof}, at least for all
sufficiently large values of $n$.

\begin{theorem}[\cite{komlos1998proof}]\label{thm:KSS}
  For any $k \in \N$, there exists an $n_0 \in \N$ such that if $G$ has order
  $n$ with $n \geq n_0$ and $\delta(G) \geq kn/(k + 1)$, then $G$ contains the
  $k$-th power of a Hamilton cycle.
\end{theorem}

For more history on the problem and similar embedding questions we refer the
reader to the literature, cf.\ e.g.~\cite{bottcher2009proof, chau2011posa,
hajnal1970proof, komlos2001proof, komlos1998posa, levitt2010avoid} and the
survey~\cite{kuhn2009embedding}.

Generalising the type of problems considered in the above theorem, we arrive at
the following Dirac-type question: given a graph property $\calP$, what is the
minimum number $\alpha$ such that every graph $G$ on $n$ vertices and minimum
degree at least $\alpha n$ satisfies $G \in \calP$? This leads to the notion of
{\em local resilience} that we now introduce formally.

\begin{definition}[Local resilience]\label{def:local-resilience}
  Let $G = (V, E)$ be a graph and $\cal P$ a monotone increasing graph property.
  The {\em local resilience} of $G$ with respect to $\calP$ is defined as:
  \begin{align*}
    r(G, \calP) := \min\{ r : \exists \tilde G \subseteq G
    & \text{ such that each } v \in V \text{ satisfies } \\
    & \deg_{\tilde G}(v) \leq r \cdot \deg_{G}(v) \text{ and } G - \tilde G
    \text{ does not have } \calP \}.
  \end{align*}
\end{definition}

Looking back at the aforementioned results, Dirac's theorem implies that the
local resilience of the complete graph $K_n$ with respect to Hamiltonicity is at
least $n/2$ and the theorem of Koml\'{o}s, S\'{a}rk\"{o}zy, and Szemer\'{e}di
implies that the local resilience of `containment of the $k$-th power of a
Hamilton cycle' is at least $n/(k + 1)$. Moreover, it is not too difficult to
construct examples that show that both of these results are optimal (consider,
for example, $k = 2$ and a complete $3$-partite graph with two parts of size $(n
- 1)/3$ and one of size $(n + 2)/3$).

In this paper we study how Theorem~\ref{thm:KSS} can be transferred to the
setting of random graphs. Such transference results recently received
considerable attention including several breakthrough results by Balogh, Morris,
and Samotij~\cite{balogh2015independent}, Conlon and
Gowers~\cite{conlon2016combinatorial}, Conlon, Gowers, Samotij, and
Schacht~\cite{conlon2014klr}, Saxton and Thomason~\cite{saxton2015hypergraph},
and Schacht~\cite{schacht2016extremal}.

We denote by $\Gnp$ the probability space of all graphs with vertex set $[n] :=
\{1, \dotsc, n\}$ where each edge appears randomly with probability $p := p(n)
\in (0, 1)$, independently of all other edges. A systematic study of local
resilience in random graphs was initiated by Sudakov and
Vu~\cite{sudakov2008local} and already led to many beautiful and deep results,
see e.g.~\cite{balogh2011local, allen2020bandwidth, bottcher2013almost,
krivelevich2010resilient, lee2012dirac, montgomery2019hamiltonicity,
nenadov2019resilience} and the recent surveys~\cite{bottcher2017large,
sudakov2017robustness}. Inspired by other transference results (such as the ones
mentioned above as well as many more) from dense graphs to the random setting,
one may be tempted to guess that {\em with high probability}\footnote{We say
that an event holds with high probability (w.h.p.\ for short), if the
probability that it holds tends to $1$ as $n$ tends to infinity.} a random graph
is such that every subgraph with minimum degree roughly $(2/3 + o(1))np$
contains the square of a Hamilton cycle.

On second thoughts, however, one easily sees that this cannot hold for $p \gg
\log n/\sqrt{n}$. An adversary can remove all the edges with both endpoints
lying in the neighbourhood of an arbitrary vertex $v$, thus preventing $v$ from
being in a triangle (which implies in particular that $v$ cannot be contained in
any square of a cycle); note that the deletion of these edges changes the degree
of every other vertex only by $o(np)$. In fact, Huang, Lee, and
Sudakov~\cite{huang2012bandwidth} and Balogh, Lee, and
Samotij~\cite{balogh2012corradi} showed that an adversary can always prevent as
many as $\Omega(p^{-2})$ vertices from being in triangles by deleting $o(np)$
edges touching each vertex, as long as $1/\sqrt{n} \ll p \ll 1$. The former
result shows the claim even when $p$ is a fixed constant, independent of $n$.

In this paper we overcome the obstacles that the notion of local resilience
encounters with respect to containment of spanning structures (that contain
triangles). For this we generalise the notion of local resilience. More
precisely, we restrict the adversary to only remove a fraction of certain
substructures touching each vertex. In the classic definition of local
resilience these substructures correspond to edges. For obtaining the square of
a Hamilton cycle it turns out that one should replace edges by triangles. This
then motivates the following question:
\begin{center}
  {\em How many triangles at a vertex does an adversary have to destroy in order
  to obtain a graph without the square of a Hamilton cycle?}
\end{center}

We capture this question under the notion of {\em $K_3$-resilience}, or more
generally {\em $H$-resilience} as given in the following definition.
\begin{definition}
  Let $H$ be a fixed graph and let $\calP$ be a monotone increasing graph
  property. For a graph $G$, the $H$-{\em resilience} of $G$ with respect to
  $\calP$ is defined as
  \begin{align*}
    r_H(G, \calP) := \min\{ r : \exists \tilde G \subseteq G
    & \text{ such that the removal of $\tilde G$ destroys at most an
    $r$-fraction} \\
    & \text{of copies of $H$ in $G$ at every vertex} \text{ and $G - \tilde G$
    does not have } \calP \}.
  \end{align*}
\end{definition}
In the main result of this paper we show that the above definition can be used
in order to determine the resilience of $\Gnp$ with respect to the containment
of the square of a Hamilton cycle.
\begin{theorem}\label{thm:main-theorem}
  The $K_3$-resilience of $\Gnp$ w.r.t.\ the containment of the square of a
  Hamilton cycle is w.h.p.\ $5/9 \pm o(1)$, provided that $p \gg n^{-1/2}
  \log^{3} n$.
\end{theorem}

In other words, the above theorem shows that w.h.p.\ the adversary needs to
delete more than a $(5/9)$-fraction of the triangles lying on each vertex in
order to destroy all copies of the square of a Hamilton cycle in $\Gnp$. The
density value $p$ is optimal up to the logarithmic factor, as a simple
application of the first moment method shows that for $p \ll n^{-1/2}$ a random
graph $\Gnp$ w.h.p.\ does not contain the square of a Hamilton cycle.
Additionally, this result marginally improves upon the current {\em appearance
threshold} for the square of a Hamilton cycle in $\Gnp$ by Nenadov and the
second author~\cite{nenadov2019powers}, by a $\log n$ factor in the density $p$.

The second result of the paper rephrases the above theorem in slightly different
terms. From Theorem~\ref{thm:KSS} we know that in the dense case it is
sufficient to require that the minimum degree is at least $2n/3$. Although the
analogous statement cannot be true in the case of random graphs, we prove that
w.h.p.\ every spanning subgraph which satisfies the correct minimum degree
condition and the additional property that each edge is contained in $\alpha
np^2$ triangles, contains the square of a Hamilton cycle. Before we can state
this result precisely, we need a definition.

\begin{definition}
  Let $\Gamma$ be a graph on $n$ vertices. We denote by $\calG(\Gamma, n,
  \alpha, p)$ the family of all spanning subgraphs $G \subseteq \Gamma$ that
  satisfy the following properties:
  \begin{enumerate}
    \item for every $v \in V(G)$: $\deg_G(v) \geq (2/3 + \alpha) np$, and
    \item for every $\{u, v\} \in E(G)$: $|N_G(u) \cap N_G(v)| \geq \alpha
      np^2$.
  \end{enumerate}
\end{definition}

With this at hand we can state the second result of the paper.

\begin{theorem}\label{thm:square-res-codegree}
  For every $\alpha > 0$ there exists a positive constant $C := C(\alpha)$, such
  that a random graph $\Gamma \sim \Gnp$ w.h.p.\ has the following property,
  provided that $p \geq C n^{-1/2} \log^{3} n$. Each member of $\calG(\Gamma, n,
  \alpha, p)$ contains the square of a Hamilton cycle.
\end{theorem}

As in the first result of the paper, the value of $p$ is almost optimal.
Furthermore, the constant $2/3$ in the definition of the class $\calG(\Gamma, n,
\alpha, p)$ cannot be improved, as the same counterexample as in the dense case
works in this scenario as well. Even though the type of conditions in two
theorems above look quite different at first sight, we prove
Theorem~\ref{thm:main-theorem} by a reduction to
Theorem~\ref{thm:square-res-codegree}.

The proof of Theorem~\ref{thm:square-res-codegree} uses the so-called {\em
absorbing method}. In particular, we make use of a strategy paved by Nenadov and
the second author~\cite{nenadov2019powers}. This method is discussed in
Section~\ref{sec:absorbers}. In Section~\ref{sec:preliminaries} we introduce
some notation and probabilistic tools, and state several useful lemmas about
properties of (random) graphs, culminating in Lemma~\ref{lem:everything} about
edge expansion properties. In Section~\ref{sec:K3-res} we give the proof of
Theorem~\ref{thm:main-theorem} by a reduction to
Theorem~\ref{thm:square-res-codegree}. We also show in this section that the
constant $5/9$ in Theorem~\ref{thm:main-theorem} is best possible. In
Section~\ref{sec:graph-definitions} we introduce several classes and definitions
of graphs which we rely on throughout the paper. In Section~\ref{sec:theproof}
we give the proof of Theorem~\ref{thm:square-res-codegree} modulo several
lemmas. Each of the subsequent
Sections~\ref{sec:absorbers}--\ref{sec:block-expansion} are dedicated to the
proof of one of the technical lemmas and/or claims. Finally, we conclude by
discussing some related open problems in Section~\ref{sec:conclusion}.

\section{Tools and preliminaries}\label{sec:preliminaries}

Our graph theoretic notation is standard (see, e.g.~\cite{bondy2008graph}). In
particular, for a graph $G = (V, E)$ we denote by $N_G(v)$ the neighbourhood of
a vertex $v \in V$ and by $\deg_{G}(v)$ its size, i.e.\ $\deg_{G}(v) =
|N_G(v)|$. Similarly, for $X \subseteq V$ we write $N_G(X)$ for the union of
neighbourhoods of vertices in $X$, that is $N_G(X) := \{ u : \{v, u\} \in E
\text{ and } v \in X\}$. Furthermore, for $X, Y \subseteq V$, we let $N_G(X, Y)$
denote $N_G(X) \cap Y$ and if $X$ consists of a single vertex we abbreviate
$N_G(\{x\}, Y)$ to $N_G(x, Y)$. If $X, Y \subseteq V$ are disjoint subsets of
vertices we write $e_G(X, Y)$ for the number of edges with one endpoint in $X$
and the other in $Y$. We use a set of edges $I \subseteq E$ interchangeably as a
set of edges and a (sub)graph. In particular, we write $\deg_{I}(v)$ to denote
the number of edges from $I$ that are incident to a vertex $v$ and $e_{I}(X, Y)$
for the number of edges in $I$ that have an endpoint in each of the subsets $X$
and $Y$. We omit the subscript $G$ (resp.\ $I$) whenever it is clear from the
context to which graph $G$ we refer to. For $k, \ell \in \N$ and a cycle
$C_\ell$ with $\ell$ vertices, we let $C^{k}_{\ell}$ denote the $k$-th power of
$C_\ell$, that is a graph obtained by adding an edge between any two vertices of
$C_\ell$ which are at distance at most $k$. Given two graphs $H$ and $G$, and a
function $f \colon V(H) \to V(G)$, we say that $f$ is an {\em embedding} of $H$
into $G$ if it is an injection and for all $\{v, u\} \in E(H)$ we have $\{f(u),
f(v)\} \in E(G)$.

For an integer $k \geq 2$ and a set $V$ we write $\binom{V}{k}$ for the family
of all subsets of $V$ with cardinality exactly $k$. We write $V^k$ to denote the
family of all ordered $k$-tuples of $V$ whose entries are pairwise different,
that is $V^k := \{ (v_1, \dotsc, v_k) : v_i \in V \text{ for all } i \in [k]
\text{ and } v_i \neq v_j \text{ for all } 1 \leq i < j \leq k \}$. An element
of $V^k$ is usually denoted by a lower case bold letter. Given $\bw = (v_1,
\dotsc, v_k) \in V^k$, we write $\obw$ to denote the tuple obtained by reversing
the order in $\bw$, i.e.\ $\obw = (v_k, \dotsc, v_1)$. Moreover, for two ordered
tuples $\bw_1$ and $\bw_2$, the tuple $\bw_1 \bw_2$ is an ordered tuple obtained
by concatenation of $\bw_1$ and $\bw_2$. For a function $g$ applicable to the
elements of a tuple $(x_1, \dotsc, x_n)$ we for convenience shorten $(g(x_1),
\dotsc, g(x_n))$ to $g(x_1, \dotsc, x_n)$.

For an integer $n \in \N$ we write $[n] := \{1, \dotsc, n\}$. Given $a, b, c, x
\in \R$ we write $x \in (a \pm b)c$ to denote $(a - b)c \leq x \leq (a + b)c$.
We make use of the standard asymptotic notation, $o$, $O$, $\omega$, $\Omega$,
and $\Theta$. For two functions $a$ and $b$, we write $a \ll b$ to denote $a =
o(b)$ and similarly $a \gg b$ for $a = \omega(b)$. All logarithms are with
respect to base~$e$. We omit floors and ceilings whenever they are not of
importance. Lastly, we write $C_{5.1}$ to indicate that the constant $C_{5.1}$
is given by Theorem/Lemma/Claim~{5.1}.

The following statement about $r$-star-matchings is an easy corollary of Hall's
matching theorem~\cite{hall1935representatives}. A star of size $r$ ($r$-{\em
star}, for short) is a complete bipartite graph $K_{1, r}$ with the vertex
adjacent to all others being the {\em centre}.

\begin{lemma}[$r$-star-matching]\label{lem:star-hall}
  Let $r \geq 1$ be an integer and let $G = (A \cup B, E)$ be a bipartite graph.
  If for every subset $A' \subseteq A$ it holds that $|N(A')| \ge r|A'|$, then
  $G$ contains a collection of pairwise disjoint $r$-stars, such that the
  centres of these stars cover all vertices in $A$.
\end{lemma}
\begin{proof}
  Consider the `blow-up' of $G$ in which each vertex in $A$ is replaced by $r$
  copies that are connected to the same vertices as the original vertex. Then
  this new graph satisfies Hall's condition and thus contains a matching that
  saturates all copies of the vertices in $A$. The corollary follows by
  contracting the copies of each vertex.
\end{proof}

We also make use of a generalised version of Hall's theorem due to Haxell which
has recently seen a surge of applications in embedding spanning structures into
random graphs, especially in the resilience setting.

\begin{theorem}[\cite{haxell1995condition}]\label{thm:haxells-theorem}
  Let $\cH = (A \cup B, E)$ be an $r$-uniform hypergraph such that $|A \cap e| =
  1$ and $|B \cap e|= r - 1$ for every edge $e \in E$. If for every $A'
  \subseteq A$ and $B' \subseteq B$ such that $|B'| \leq (2r - 3)|A'|$ there is
  an edge $e \in E$ intersecting $A'$ but not $B'$, then $\cH$ contains an
  $A$-saturating matching (that is, a collection of vertex-disjoint hyperedges
  whose union contains A).
\end{theorem}

We repeatedly make use of the following two standard tail estimates used in
random graph theory, cf.\ e.g.~\cite{alon2016probabilistic,
frieze2016introduction}.

\begin{lemma}[Chernoff's inequality]\label{lemma:chernoff}
  Let $X \sim \Bin(n, p)$ and let $\mu := \E[X]$. Then for all $0 < a < 1$:
  \begin{itemize}
    \item $\Pr \left[ X < (1 - a)\mu \right] < e^{-a^2\mu/2}$, and
    \item $\Pr \left[ X > (1 + a)\mu \right] < e^{-a^2\mu/3}$.
  \end{itemize}
  Moreover, the inequalities above also hold if $X$ has the hypergeometric
  distribution with the same mean.
\end{lemma}

\begin{theorem}[Janson's inequality]\label{thm:janson}
  Let $p \in (0, 1)$ and consider a family $\{ H_i \}_{i \in \calI}$ of
  subgraphs of the complete graph on the vertex set $[n]$. Let $\Gamma \sim
  \Gnp$. For each $i \in \calI$, let $X_i$ denote the indicator random variable
  for the event $\{H_i \subseteq \Gamma\}$ and, for each ordered pair $(i, j)
  \in \calI \times \calI$ with $i \neq j$, write $H_i \sim H_j$ if $E(H_i) \cap
  E(H_j)\neq \varnothing$. Let
  \begin{align*}
    X &:= \sum_{i \in \calI} X_i, \\
    \mu &:= \mathbb{E}[X] = \sum_{i \in \calI} p^{e(H_i)}, \\
    \Delta &:= \sum_{\substack{(i, j) \in \calI \times \calI \\ H_i \sim H_j}}
    \mathbb{E}[X_i X_j] = \sum_{\substack{(i, j) \in \calI \times \calI \\ H_i
    \sim H_j}} p^{e(H_i) + e(H_j) - e(H_i \cap H_j)}.
  \end{align*}
  Then for all $0 < \gamma < 1$ we have
  \[
    \Pr[X < (1 - \gamma)\mu] \le e^{- \frac{\gamma^2 \mu^2}{2(\mu + \Delta)}}.
  \]
\end{theorem}

Next, we collect several facts about random graphs mostly concerning the number
of edges and triangles between certain subsets, as well as a simple edge
expansion property.

\begin{lemma}\label{lemma:triangles}
  For every $\eps \in (0, 1/100)$ and $p := p(n) \in (0, 1)$, the random graph
  $\Gamma \sim \Gnp$ w.h.p.\ satisfies the following.

  Let $s \geq \eps^{-10}\log n/p^2$ be an integer. Then there are at least
  $(1-n^{-6})\binom{n}{s}$ subsets $W \subseteq V(\Gamma)$ of size $s$ such that
  the following holds. For every $W' \subseteq W$ of size $|W'| \geq \eps|W|$
  and every family of pairs $\calP \subseteq \binom{V(\Gamma) \setminus W}{2}$
  of size $|\calP| \geq \eps^{-10}\log n/p^2$ and such that no vertex of
  $\Gamma$ appears in more than $1 + 1/p$ pairs from $\calP$, we have
  \[
    \sum_{\{u, v\} \in \calP} |N(u, W') \cap N(v, W')| \leq (1 + \eps) |\calP|
    |W'| p^2.
  \]
\end{lemma}
\begin{proof}
  An easy application of Chernoff's inequality (Lemma~\ref{lemma:chernoff}) and
  the union bound shows that w.h.p.\ every pair of vertices $u, v \in V(\Gamma)$
  have a common neighbourhood which satisfies $|N(u) \cap N(v)| \leq
  (1+\eps^6)np^2$. For a set $W'' \subseteq V(\Gamma)$ and a triple $(u, v, w)
  \in (V(\Gamma) \setminus W'') \times (V(\Gamma) \setminus W'') \times W''$, we
  define a random variable
  \[
    X_{u, v, w} =
      \begin{cases}
        1, & \text{if } \{u, v\} \in \calP \text{ and } \{u, v\} \subseteq N(w),
        \\
        0, & \text{otherwise}.
      \end{cases}
  \]
  Suppose that w.h.p.\ for all $\calP$ as in the statement of the lemma and all
  $W'' \subseteq V(\Gamma)$, $|W''| \geq \eps^{-7}\log n/p^2$, the following
  holds:
  \begin{equation}\label{eq:triangles-eq1}
    \sum_{\{u, v\} \in \calP} \sum_{w \in W''} X_{u, v, w} \ge (1 - \eps^3)
    |\calP| |W''| p^2.
  \end{equation}
  Condition on the fact that $\Gamma \sim \Gnp$ satisfies these two properties,
  which happens with high probability. We claim that this is sufficient in order
  to show the lemma.

  Let $W \subseteq V(\Gamma)$ be a subset of size $s \geq \eps^{-10}\log n/p^2$
  chosen uniformly at random among all such subsets. For a fixed pair $u, v \in
  V(\Gamma)$, from Chernoff's inequality we have
  \[
    \Pr\big[|N(u, W) \cap N(v, W)| > (1+\eps^3)|W|p^2\big] \leq
    e^{-(\eps^3/3)|W|p^2} \leq e^{-\eps^{-6}\log n} < 1/n^9,
  \]
  and so by the union bound we conclude that with probability at least $1 -
  o(n^{-6})$ for every two vertices $u, v \in V(\Gamma)$ it holds that
  \begin{equation}\label{eq:triangles-eq2}
    |N(u, W) \cap N(v, W)| \le (1 + \eps^3) |W| p^2.
  \end{equation}
  Let $W' \subseteq W$ be of size $|W'| \geq \eps|W|$, let $X = \sum_{\{u, v\}
  \in \calP} \sum_{w \in W'} X_{u, v, w}$, and note that $X$ counts exactly the
  quantity we are interested in. Take $W'' := W \setminus W'$. Then, trivially,
  \[
    X = \sum_{\{u, v\} \in \calP} \sum_{w \in W'} X_{u, v, w} = \sum_{\{u, v\}
    \in \calP} \sum_{w \in W} X_{u, v, w} - \sum_{\{u, v\} \in \calP} \sum_{w
    \in W''} X_{u, v, w}.
  \]
  The first term on the right hand side of the previous equation is by
  \eqref{eq:triangles-eq2} bounded from above by $(1 + \eps^3)|\calP||W|p^2$.
  If $|W''| \geq \eps^3|W| \geq \eps^{-7}\log n/p^2$, then we can use the lower
  bound from \eqref{eq:triangles-eq1} to obtain
  \[
    X \leq (1 + \eps^3)|\calP||W|p^2 - (1 - \eps^3)|\calP||W''|p^2 \leq (1 +
    \eps^3)|\calP||W'|p^2 + 2\eps^3|\calP||W|p^2 \leq (1 + \eps)|\calP||W'|p^2,
  \]
  where the last inequality holds because $|W'| \ge \eps|W|$ and $\eps < 1/100$.
  In case $|W''| < \eps^3|W|$ we have $|W'| \ge (1 - \eps^3)|W|$ and thus from
  \eqref{eq:triangles-eq2} we have
  \[
    X \leq (1 + \eps^3)|\calP||W|p^2 \leq \frac{1 + \eps^3}{1 - \eps^3}
    |\calP||W'|p^2 \leq (1 + \eps)|\calP| |W'|p^2,
  \]
  since $\eps < 1/100$. In conclusion, $W$ satisfies the assertion of the lemma
  with probability at least $1 - o(n^{-6})$ which implies the desired statement.
  It remains to show that \eqref{eq:triangles-eq1} is indeed true.

  Denote the left hand side of \eqref{eq:triangles-eq1} by $X''$. By linearity
  of expectation we have $\mu := \E[X''] = |\calP| |W''| p^2$. For two triples
  $(u_1, v_1, w_1)$ and $(u_2, v_2, w_2)$ we write $(u_1, v_1, w_1) \sim
  (u_2, v_2, w_2)$ if their corresponding random variables are dependent. Note
  that $(u_1, v_1, w_1)$ and $(u_2, v_2, w_2)$ can only be dependent if $|\{u_1,
  v_1\} \cap \{u_2, v_2\}| = 1$ and $w_1 = w_2$. Thus
  \[
    \Delta := \sum_{(u_1, v_1, w_1) \sim (u_2, v_2, w_2)} p^3 \le |\calP| |W''|
    \frac{2}{p} p^3 = 2\mu.
  \]
  By Janson's inequality (Theorem~\ref{thm:janson}) it follows that
  \[
    \Pr[ X'' < (1 - \eps^3)|\calP||W''|p^2 ] \le e^{- \eps^6\mu^2/(6 \mu)} =
    e^{-\eps^6|\calP||W''|p^2/6}.
  \]
  Let us denote $\eps^{-10} \log n/p^2$ by $t$. By the union bound over all
  choices of $\calP$ and $W''$ and by using standard bounds on binomial
  coefficients, we get
  \begin{align*}
    \sum_{x = t}^{n^2} \binom{n^2}{x} \sum_{y = \eps^3t}^{n} \binom{n}{y}
    e^{-\frac{\eps^6}{6} x y p^2}
    & \leq \sum_{x = t}^{n^2} \sum_{y = \eps^3t}^{n} e^{2x\log n} \cdot e^{y\log
    n} \cdot e^{-\frac{\eps^9}{6}xyp^2} \\
    &\leq \sum_{x = t}^{n^2} \sum_{y = \eps^3t}^{n} e^{(2x+y)\log n -
    \max\{x,y\} \cdot \frac{1}{6\eps}\log n} \\
    &\leq \sum_{x = t}^{n^2} \sum_{y = \eps^3t}^{n} e^{-10\max\{x,y\}\log n}
    \leq n^3 \cdot e^{-10\eps^{-10}\log^2 n/p^2} < 1/n^8.
  \end{align*}
  This implies that with probability at least $1 - o(n^{-5})$ we have
  \[
    \sum_{\{u, v\} \in \calP} \sum_{w \in W''} X_{u, v, w} \ge
    (1-\eps^3)|\calP||W''|p^2,
  \]
  for all permissible $\calP$ and $W''$, as required.
\end{proof}

\begin{lemma}\label{lemma:edge-expansion}
  For every $\eps \in (0, 1/300)$ and $p := p(n) \in (0, 1)$, the random graph
  $\Gamma \sim \Gnp$ w.h.p.\ satisfies the following.

  Let $s \geq \eps^{-3}\log n/p^2$ be an integer. Then there are at least
  $(1-n^{-6})\binom{n}{s}$ subsets $W \subseteq V(\Gamma)$ of size $s$ such that
  the following holds. For every family of pairs $\calP \subseteq
  \binom{V(\Gamma) \setminus W}{2}$ such that no vertex of $\Gamma$ appears in
  more than $1+\eps/p$ pairs from $\calP$ and $|\calP| \leq 1+\eps/p^2$, we have
  \[
    \Big| \bigcup_{\{u, v\} \in \calP} \big(N(u, W) \cap N(v, W)\big) \Big| =
    (1\pm5\eps)|\calP||W|p^2.
  \]
\end{lemma}
\begin{proof}
  Fix a $\calP$ as in the statement and let $U$ denote the set of vertices of
  $\Gamma$ not appearing in $\calP$; note that $|U| = (1-o(1))n$. Let $q$ denote
  the probability for a vertex $w \in U$ to be in the common neighbourhood of
  $u$ and $v$, for some $\{u, v\} \in \calP$, and let $\cE_{u, v}$ denote this
  event for fixed $u$ and $v$. By the union bound we have that $q \le
  |\calP|p^2$. As for the lower bound on $q$, we use the inclusion-exclusion
  principle (Bonferroni's inequality) to get
  \begin{align*}
    q & \geq \sum_{\{u, v\} \in \calP} \Pr[\cE_{u, v}] - \sum_{\{u, v\} \cap
    \{u', v'\}| < 2} \Pr[\cE_{u, v} \land \cE_{u', v'}] \\
    & = \sum_{\{u, v\} \in \calP} \Pr[\cE_{u, v}] -
    \sum_{|\{u, v\} \cap \{u', v'\}| = 0} \Pr[\cE_{u, v} \land \cE_{u', v'}] -
    \sum_{|\{u, v\} \cap \{u', v'\}| = 1} \Pr[\cE_{u, v} \land \cE_{u', v'}].
  \end{align*}
  Using the fact that no vertex of $\Gamma$ appears in more than $1 + \eps/p$
  pairs from $\calP$,
  \[
    q \geq |\calP|p^2 - |\calP|(|\calP| - 1) p^4 - |\calP| \frac{\eps}{p} p^3
    \geq (1 - 2\eps)|\calP|p^2.
  \]
  Let us denote $\bigcup_{\{u, v\} \in \calP} \big(N(u, U) \cap N(v, U)\big)$ by
  $Z$. Observe that the expected size of $Z$ is $|U|q$ and thus by Chernoff's
  inequality and our estimates for $q$ we get
  \begin{equation}\label{eq:edge-expansion-eq1}
    \Pr[|Z| \leq (1-4\eps)|\calP|np^2] \leq \Pr[|Z| \leq (1-\eps)|U|q] \leq
    e^{-\frac{\eps^2}{2}(1-2\eps)|\calP||U|p^2} \leq
    e^{-\frac{\eps^2}{8}|\calP|np^2}.
  \end{equation}
  Similarly, we have
  \begin{equation}\label{eq:edge-expansion-eq2}
    \Pr[|Z| \geq (1+4\eps)|\calP|np^2] \leq \Pr[|Z| \geq (1+\eps)|U|q] \leq
    e^{-\frac{\eps^2}{3}(1-2\eps)|\calP||U|p^2} \leq
    e^{-\frac{\eps^2}{12}|\calP|np^2}.
  \end{equation}
  Since $np^2 \geq \eps^{-3}\log n$, by combining \eqref{eq:edge-expansion-eq1} and
  \eqref{eq:edge-expansion-eq2} together with the union bound over all choices
  for $\calP$ we get that with high probability for all such sets $\calP$ and
  $u, v \in \calP$
  \begin{equation}\label{eq:edge-expansion-eq3}
    \Big|\bigcup_{\{u,v\} \in \calP} \big(N(u, U) \cap N(v, U)\big)\Big| = (1
    \pm 4\eps)|\calP|np^2.
  \end{equation}
  Condition on $\Gamma \sim \Gnp$ satisfying this. Let $W \subseteq V(\Gamma)$
  be a set of size $s \geq \eps^{-3}\log n/p^2$ chosen uniformly at random among
  all such sets. Let $\calP$ be as in the statement of the lemma, and note that
  \eqref{eq:edge-expansion-eq3} is fulfilled for $\calP$ and the corresponding
  set $U$. By Chernoff's inequality, we have
  \[
    \Big|\bigcup_{\{u,v\} \in \calP} \big(N(u, W) \cap N(v, W)\big)\Big| = (1
    \pm 5\eps)|\calP||W|p^2,
  \]
  with probability at least
  \[
    1 - e^{-\frac{2\eps^2}{3}|\calP||W|p^2} \geq 1 - e^{-\frac{2}{3\eps}\log n}
    \geq 1 - n^{-6},
  \]
  as required.
\end{proof}

The following definition is used at various places throughout the paper. It
captures essential properties of a random graph $\Gamma \sim \Gnp$ and its
subgraphs $G \in \calG(\Gamma, \alpha, \eps, p)$ which are used in order to
prove the main result. Some of the properties follow easily from others; we list
them all separately for ease of reference later on.

\begin{definition}
  Let $\eps, \alpha > 0$, $p \in (0, 1)$, and $n \in \N$. Let $\Gamma$ be a
  graph on $n$ vertices and $G \in \calG(\Gamma, n, 2\alpha, p)$. We say that a
  subset $W \subseteq V(G)$ is $(\alpha,\eps,p)$-{\em good} with respect to
  $\Gamma$ and $G$ if the following properties hold:
  \begin{enumerate}[leftmargin=3em,label=(G\arabic*)]
    \item\label{p:unif-density} For every two disjoint subsets $X \subseteq W$,
      $Y \subseteq V(\Gamma)$ of sizes $|X|, |Y| \geq \eps^{-3}\log n/p$ we have
      \[
        e_\Gamma(X, Y) = (1 \pm \eps)|X||Y|p.
      \]
    \item\label{p:unif-deg} For every $v \in V(\Gamma)$ we have $\deg_\Gamma(v,
      W) = (1 \pm \eps)|W|p$.
    \item\label{p:good-deg-con} For every $v \in V(G)$ we have $\deg_G(v, W) =
      (1 \pm \eps) \deg_G(v) \frac{|W|}{n}$.
    \item\label{p:good-deg} For every $v \in V(G)$ we have $\deg_G(v, W) \ge
      (2/3 + \alpha)|W|p$.
    \item\label{p:good-codeg} For every $\{u, v\} \in E(G)$ we have $|N_G(u, W)
      \cap N_G(v, W)| \ge \alpha|W|p^2$.
    \item\label{p:unif-triangles} For all subsets $W' \subseteq W$ of size $|W'|
      \geq \eps|W|$ and all $\calP \subseteq \binom{V(G) \setminus W}{2}$ of
      size $|\calP| \ge \eps^{-10} \log n/p^2$ such that no vertex of $G$
      appears in more than $1 + 1/p$ pairs from $\calP$, we have
      \[
        \sum_{\{u, v\} \in \calP} |N_G(u, W') \cap N_G(v, W')| \leq (1 +
        \eps)|\calP||W'|p^2.
      \]
    \item\label{p:good-edge-expansion} For every set of edges $\calP \subseteq
      E(G)$ avoiding $W$ such that no vertex of $G$ appears in more than $1 +
      \eps/p$ edges from $\calP$ and $|\calP| \le 1 + \eps/p^2$, we have
      \[
        \Big| \bigcup_{\{u, v\}\in \calP} \big( N_G(u, W) \cap N_G(v, W) \big)
        \Big| \geq \alpha|\calP||W|p^2.
      \]
  \end{enumerate}
\end{definition}

In order to be precise, whenever speaking about $(\alpha,\eps,p)$-good sets, one
would always need to specify graphs $\Gamma$ and $G$ as well. However, for a
cleaner exposition, we omit it as these two graphs are always clear from the
context and almost always are the random graph $\Gamma \sim \Gnp$ and its
subgraph $G \in \calG(\Gamma, \alpha, \eps, p)$. The following proposition shows
that w.h.p.\ an overwhelming majority of sets $W$ are $(\alpha/2,\eps,p)$-good
with respect to $\Gamma \sim \Gnp$ and its subgraph $G$ which belongs to the
class $\calG(\Gamma, \alpha, \eps, p)$, for a certain choice of parameters.
Namely, a random choice of $W$ is typically good.

\begin{proposition}\label{prop:random-set-is-good}
  For every $\alpha > 0$ there exists a positive constant $\eps := \eps(\alpha)$
  such that the random graph $\Gamma \sim \Gnp$ w.h.p.\ satisfies the following
  for every $G \in \calG(\Gamma, n, \alpha, p)$.

  Let $s \geq \eps^{-10}\log n/p^2$ be an integer. Then there are at least
  $(1-n^{-5})\binom{n}{s}$ subsets $W \subseteq V(G)$ such that $W$ is
  $(\alpha/2,\eps,p)$-good.
\end{proposition}
\begin{proof}
  Choose $\eps > 0$ sufficiently small so that the arguments below follow
  through. One can easily show with Chernoff's inequality and the union bound
  that w.h.p.\ $\Gamma \sim \Gnp$ is such that
  \begin{itemize}
    \item {\em (density)} $e_\Gamma(X, Y) = (1 \pm \eps)|X||Y|p$, for every $X,
      Y \subseteq V(\Gamma)$ with $|X|, |Y| \geq \eps^{-3}\log n/p$;
    \item {\em (degree)} $\deg_\Gamma(v) = (1 \pm \eps/2)np$, for every $v \in
      V(\Gamma)$; and
    \item {\em (codegree)} $|N_\Gamma(u) \cap N_\Gamma(v)| \leq (1+\eps/2)np^2$,
      for every $u, v \in V(\Gamma)$.
  \end{itemize}
  Furthermore, with high probability, $\Gamma$ satisfies both the conclusion of
  Lemma~\ref{lemma:triangles} and that of Lemma~\ref{lemma:edge-expansion}.
  Condition on these five events from now on.

  Let $G \subseteq \Gamma$ be a member of $\calG(\Gamma, n, \alpha, p)$ and let
  $H := \Gamma - G$. We choose a set $W$ uniformly at random among all subsets
  of size $s$ and aim to show that with probability at least $1 - o(n^{-5})$
  such a set satisfies all \ref{p:unif-density}--\ref{p:good-edge-expansion},
  which would imply the desired statement. Property~\ref{p:unif-density} follows
  directly from the density event we conditioned on above.
  Properties~\ref{p:unif-deg}--\ref{p:good-codeg} follow from the degree event
  we conditioned on above, the definition of $\calG(\Gamma, n, \alpha, p)$, and
  then the fact that $W$ is chosen u.a.r., Chernoff's inequality, and a simple
  union bound. Moreover, \ref{p:unif-triangles} holds for $W$ with probability
  at least $1 - n^{-6}$ by the conclusion of Lemma~\ref{lemma:triangles}, since
  $G \subseteq \Gamma$.

  Finally, let us look at \ref{p:good-edge-expansion}. By the conclusion of
  Lemma~\ref{lemma:edge-expansion}, with probability at least $1 - n^{-6}$, $W$
  is one of the subsets for which the following holds: for any subset of edges
  $\calP \subseteq E(G)$ avoiding $W$ such that $|\calP| \le 1 + \eps/p^2$ and
  no vertex of $G$ appearing in more than $1 + \eps/p$ edges from $\calP$, we
  have
  \begin{equation}\label{eq:random-set-is-good-eq1}
    \Big| \bigcup_{\{u, v\} \in \calP} \big( N_\Gamma(u, W) \cap N_\Gamma(v, W)
    \big) \Big| \ge (1 - 5\eps)|\calP||W|p^2.
  \end{equation}
  From the codegree event we conditioned on above and the definition of the
  class $\calG(\Gamma, n, \alpha, p)$ we further get that
  \[
    |N_H(u) \cap N_H(v)| \leq (1+\eps/2-\alpha)np^2 \leq (1-3\alpha/4)np^2,
  \]
  for every edge $\{u, v\} \in E(G)$. Moreover, as $W$ is chosen uniformly at
  random, by Chernoff's inequality and the union bound we also have that with
  probability at least $1 - o(n^{-5})$, every $\{u, v\} \in E(G)$ satisfies
  \begin{equation}\label{eq:random-set-is-good-eq2}
    |N_H(u, W) \cap N_H(v, W)| \le (1-3\alpha/4+\eps)|W|p^2.
  \end{equation}
  By combining \eqref{eq:random-set-is-good-eq1} and
  \eqref{eq:random-set-is-good-eq2} we get
  {\small
    \begin{align*}
      \Big| \bigcup_{\{u, v\} \in \calP} \big( N_G(u, W) \cap N_G(v, W) \big)
      \Big| & \geq \Big| \bigcup_{\{u, v\} \in \calP} \big( N_\Gamma(u, W) \cap
      N_\Gamma(v, W) \big) \Big| - \sum_{\{u, v\} \in \calP} \big| N_H(u, W)
      \cap N_H(v, W) \big| \\
      & \geq (1-5\eps)|\calP||W|p^2 - |\calP|(1-3\alpha/4+\eps)|W|p^2 \\
      & \geq (\alpha/2)|\calP||W|p^2,
    \end{align*}
  }
  where the last inequality follows for small enough $\eps > 0$. This concludes
  the proof.
\end{proof}

\subsection{Edge expansion and triangles}

In this section we provide some expansion tools that we make use of in our
proof. Before we state them, we give some motivational background. A standard
approach for showing that two vertices $a$ and $b$ are connected by a path of
length, say, $\ell$, is to inductively prove a lower bound on the number of
vertices that can be reached by paths of length $i$, starting from each $a$ and
$b$. To prove such a bound one usually relies on expansion properties of
vertices in certain sets. In our case, we do not want to find paths, but {\em
square paths} (see~Section~\ref{sec:graph-definitions} for details), where each
new vertex is given by a triangle lying on a previous edge. In particular, we
build such paths by starting from an edge and determining how many edges
(instead of vertices) we can reach starting from this edge. Correspondingly, we
need expansion properties of edges instead of vertices. The goal of this section
is to provide such expansion properties.

More precisely, we are trying to understand the following setup. Suppose we are
given three disjoint subsets of vertices $W_1, W_2, W_3$ of a graph $G$. Assume
further that $F_{12} \subseteq E_G(W_1, W_2)$ is some set of edges from $G$
between $W_1$ and $W_2$. What we are interested in is the set of edges
\[
  F_{23} := \big\{ \{w_2, w_3\} \in E_G(W_2, W_3) : \exists w_1 \in W_1
  \text{ s.t.\ } \{w_1, w_2\} \in F_{12} \text{ and } \{w_1,w_3\} \in E(G)
  \big\}
\]
that `extend' an edge from $F_{12}$ to the set $W_3$ via a triangle. Namely, we
are aiming at providing some bound on $|F_{23}|$ in terms of $|F_{12}|$.

\begin{lemma}\label{lem:everything}
  For every $\alpha > 0$ there exists a positive constant $\eps := \eps(\alpha)$
  such that for sufficiently large $n$ and all $p := p(n) \in (0, 1)$, every
  graph $\Gamma$ on $n$ vertices and $G \in \calG(\Gamma, n, 2\alpha, p)$
  satisfy the following.

  Let $\tilde n \geq \eps^{-22}\log^2 n/p^2$ be an integer. Let $W_1, W_2, W_3
  \subseteq V(G)$ be three disjoint $(\alpha, \eps, p)$-good sets of size
  $\tilde n$ each, let $X \subseteq W_3$ be of size $|X| \leq \eps^4 \tilde n$,
  and $F_{12} \subseteq E_G(W_1, W_2)$ be a subset of edges. Then the following
  statements hold, where $U \subseteq W_2$ is the set of vertices incident to
  the edges in $F_{12}$, and $F_{23}$ is as defined above.
  \begin{enumerate}[(1)]
    \item\label{lemma:star-matching} If $|U| \geq |X|/ \log n$ and
      $\deg_{F_{12}}(v) \leq \eps/p$ for all $v \in W_1 \cup W_2$, then there
      exists a subset $U' \subseteq U$ of size $|U'| = (1 - \eps) \min\{|U|,
      \eps/p^2\}$ and an $(\alpha\tilde np^2/2)$-star-matching in $F_{23}$
      saturating $U'$ and avoiding $X$ in $W_3$.

    \item\label{lemma:small-expand} If $|F_{12}| \geq \eps^{-17}\tilde n\log^2
      n$ and $\deg_{F_{12}}(v, W_1) \leq \frac{\eps^{-4}\log n}{p}$ for all $v
      \in U$, then $e_{F_{23}}(U, W_3 \setminus X) \geq \eps^{-4}|F_{12}|$.

    \item\label{lemma:medium-dont-shrink} If $|F_{12}| \geq \eps^{-5}\tilde
      n\log n$ and $\deg_{F_{12}}(v, W_1) \geq \frac{\eps^{-4}\log n}{p}$ for
      all $v \in U$, then $e_{F_{23}}(U, W_3 \setminus X) \geq \frac{\alpha
      \tilde np}{4}|U|$.

    \item\label{lemma:medium-expand} If $|F_{12}| \geq \frac{\eps^{-5}\log n}{p}
      \tilde n$ and $\frac{\eps^{-4}\log n}{p} \leq \deg_{F_{12}}(v, W_1) \le
      \frac{\tilde np}{3}$ for all $v \in U$, then $e_{F_{23}}(U, W_3 \setminus
      X) \geq (1 + \alpha/4)|F_{12}|$.

    \item\label{lemma:large-to-all} If $|F_{12}| \geq \frac{\eps^{-10}\log n}{p}
      \tilde n$ and $\deg_{F_{12}}(v, W_1) \geq \frac{\tilde np}{3}$ for all $v
      \in U$, then $e_{F_{23}}(U, W_3 \setminus X) \ge (1 - \eps^2) e_G(U, W_3)$
      and there exists a subset $L \subseteq U$ of size $|L| \geq (1 -
      3\eps^3)|U|$, such that for every $u \in L$ we have $\deg_{F_{23}}(u, W_3
      \setminus X) \geq (1 - 2\eps^3) \deg_G(u, W_3)$.

    \item\label{lemma:large-stay-large} If $|U| \geq \frac{2}{3} \tilde n$ and
      $\deg_{F_{12}}(v, W_1) \geq \frac{\tilde np}{3}$ for all $v \in U$, then
      there exists $L \subseteq W_3 \setminus X$ of size $|L| \ge (1 -
      \eps)\tilde n$, such that for every $u \in L$ we have $\deg_{F_{23}}(u,
      W_2) \geq (1/3 + \alpha/2)\tilde np$.

    \item\label{lemma:sparse-expand-or} If $|F_{12}| \geq \eps^{-18}\tilde
      n\log^2 n$ and $e_{F_{23}}(W_2, W_3 \setminus X) < \eps^{-3}|F_{12}|$,
      then there exists a subset $L \subseteq W_2$ such that for every $v \in L$
      we have $\deg_{F_{12}}(v, W_1) \geq \eps^{-4}\log n/p$ and
      $e_{F_{12}}(W_1, L) \geq (1 - \eps)|F_{12}|$.

    \item\label{lemma:dense-expand-or} If $|F_{12}| \geq \eps \tilde n^2 p$ and
      $e_{F_{23}}(W_2, W_3 \setminus X) < (1 + \mu\alpha/8)|F_{12}|$, then there
      exists a subset $L \subseteq W_2$ such that for every $v \in L$ we have
      $\deg_{F_{12}}(v, W_1) > \tilde np/3$ and $e_{F_{12}}(W_1, L) \geq (1 -
      \mu)|F_{12}|$, for every $\mu \in (32\eps/\alpha, 1)$.

    \item\label{lemma:dense-doesnt-drop} If $|F_{12}| \geq \eps \tilde n^2 p$,
      then $e_{F_{23}}(W_2, W_3 \setminus X) \geq (1 - \sqrt\eps)|F_{12}|$.
  \end{enumerate}
\end{lemma}

Let $\eps := \eps(\alpha) > 0$ be a small enough constant such that all the
arguments follow through. We prove the statements one by one.

\begin{proof}[Proof of~\ref{lemma:star-matching} in Lemma~\ref{lem:everything}]
  Fix an arbitrary set $\tilde U \subseteq U$ of size $|\tilde U| = \min\{|U|,
  \eps/p^2\}$. Every subset $A \subseteq \tilde U$, by using property
  \ref{p:good-edge-expansion} applied for a single edge from $F_{12}$ incident
  to each $a \in A$ (as $\calP$), satisfies
  \[
    |N_{F_{23}}(A)| \ge \alpha |A| \tilde np^2.
  \]
  Lemma~\ref{lem:star-hall} thus implies there exists an $(\alpha \tilde
  np^2)$-star-matching $M$ in $F_{23}$ saturating $\tilde U$. Let $\tilde U_X$
  be the largest subset of $\tilde U$ such that for each $v \in \tilde U_X$ at
  least half of its edges in $M$ are incident to vertices in $X$. It must be
  that $|\tilde U_X| \le \eps |\tilde U|$ as otherwise we have
  \begin{align*}
    |X| & \geq |\tilde U_X| \cdot (\alpha/2)\tilde np^2 \geq \eps \min\{ |U|,
    \eps/p^2 \} \cdot (\alpha/2) \tilde np^2 \\
    &\geq \min\{ \eps^{-20} |U| \log^2 n, \eps^3 \tilde n \} \ge \min\{ |X|
    \log n, \eps^3 \tilde n \},
  \end{align*}
  which contradicts our assumptions on $|X|$. Thus, by setting $U' := \tilde U
  \setminus \tilde U_X$, we have for each vertex from $U'$ that half of its
  edges from the matching $M$ avoid $X$ and that suffices.
\end{proof}

\begin{proof}[Proof of~\ref{lemma:small-expand} in Lemma~\ref{lem:everything}]
  Let $F_{12}' \subseteq F_{12}$ be a largest subset of the edges from $F_{12}$
  obtained by keeping at most $\eps/p$ edges incident to every vertex in $U$.
  Note that
  \begin{equation}\label{eq:small-expand-eq1}
    |F_{12}'| \geq |F_{12}| \cdot \frac{\eps/p}{\eps^{-4}\log n/p} \geq
    \frac{\eps^5 |F_{12}|}{\log n} \geq \eps^{-12} \tilde n \log n.
  \end{equation}
  We define a sequence $J_1, \dotsc, J_t$ of disjoint subsets of $F_{12}'$ as
  follows. Let $J_1$ be a largest subset of $F_{12}'$ such that no vertex of
  $W_1$ is incident to more than $\eps/p$ edges from $J_1$. Assume we have
  defined $J_1, \dotsc, J_i$ for some $i \ge 1$. We then define $J_{i + 1}$ as a
  largest subset of $F_{12}' \setminus (J_1 \cup \dotsb \cup J_i)$ such that no
  vertex from $W_1$ is incident to more than $\eps/p$ edges from $J_{i + 1}$. We
  set $t$ to be the smallest integer such that $|J_1| + \dotsb + |J_t| \geq
  |F_{12}'|/2$. Note that from \ref{p:unif-deg} it follows that for every $i \in
  [t]$
  \[
    |J_i| \geq \frac{|F_{12}'|}{2} \cdot \frac{\eps/p}{(1+\eps)\tilde np} \geq
    \frac{\eps|F_{12}'|}{4\tilde n p^2} \osref{\eqref{eq:small-expand-eq1}}\geq
    \frac{\eps^{-10}\log n}{p^2}.
  \]
  By using property \ref{p:unif-triangles} for $J_i$ (as $\calP$), $X$ (as
  $W'$), and the previous inequality, we get that for every $i \in [t]$ it holds
  that $e_{F_{23}(J_i)}(U, X) \leq 2\eps^4|J_i| \tilde np^2$, where
  $F_{23}(J_i)$ denotes the edges in $F_{23}$ obtained by extending {\em only}
  the edges belonging to $J_i$. This further shows that
  \begin{equation}\label{eq:small-expand-eq2}
    e_{F_{23}(J^\star)}(U, X) \leq 2\eps^4 |J^\star| \tilde np^2,
  \end{equation}
  where $J^\star := J_1 \cup \dotsb \cup J_t$.

  Next, consider an arbitrary $v \in U$. As every such $v$ is incident to at
  most $\eps/p$ edges in $J^\star$ by definition, we have by
  \ref{p:good-edge-expansion} applied for edges of $J^\star$ incident to $v$ (as
  $\calP$) that $\deg_{F_{23}}(v, W_3) \ge \alpha \deg_{J^\star}(v, W_1) \tilde
  np^2$. Thus, it follows that
  \begin{equation}\label{eq:small-expand-eq3}
    e_{F_{23}(J^\star)}(U, W_3) \ge \alpha |J^\star| \tilde np^2.
  \end{equation}
  From the previous inequality and the fact that $|J^\star| \geq |F_{12}'|/2$,
  we obtain
  \begin{align*}
    e_{F_{23}}(U, W_3 \setminus X) &\geq e_{F_{23}(J^\star)}(U, W_3) -
    e_{F_{23}(J^\star)}(U, X)
    \osref{\eqref{eq:small-expand-eq2},\eqref{eq:small-expand-eq3}}\geq \alpha
    |J^\star|\tilde np^2 - 2\eps^4 |J^\star| \tilde np^2 \\
    & \ge (\alpha/2) |J^\star| \tilde np^2 \geq (\alpha/4)|F_{12}'| \tilde np^2
    \osref{\eqref{eq:small-expand-eq1}}\geq \frac{(\eps^5 \alpha/4) \tilde
    np^2}{\log n} |F_{12}| \ge \eps^{-4}|F_{12}|,
  \end{align*}
  where the last inequality follows from the assumption on $\tilde n$.
\end{proof}

\begin{proof}[Proof of~\ref{lemma:medium-dont-shrink} in Lemma~\ref{lem:everything}]
  From \ref{p:unif-deg} we have that $|U| \geq |F_{12}|/((1 + \eps) \tilde np)
  \geq \eps^{-4} \log n/p$. Thus, by \ref{p:unif-density} and the fact that $|X|
  \le \eps^4 \tilde n$, we get
  \begin{equation}\label{eq:medium-dont-shrink-eq1}
    e_G(U, X) \le 2\eps^4 |U| \tilde np.
  \end{equation}
  Consider an arbitrary $v \in U$ and let $N_1 := N_{F_{12}}(v, W_1)$ and $N_3
  := N_{F_{23}}(v, W_3)$. From the fact that $|N_1| \geq \eps^{-3}\log n/p$,
  \ref{p:good-codeg}, and \ref{p:unif-density} applied for $N_1$ (as $X$)
  and $N_3$ (as $Y$), we know
  \[
    |N_1| \cdot \alpha \tilde np^2 \le e_G(N_1, N_3) \le
    (1+\eps)|N_1|\max\{\eps^{-3}\log n/p, |N_3|\}p.
  \]
  By the assumption on $\tilde n$ we cannot have $|N_1|\alpha\tilde np^2 \le
  (1+\eps)|N_1|\eps^{-3}\log n$, hence it must be $|N_3| \ge (\alpha/2)\tilde
  np$. Together with \eqref{eq:medium-dont-shrink-eq1} we get
  \[
    e_{F_{23}}(U, W_3 \setminus X) = e_{F_{23}}(U, W_3) - e_{F_{23}}(U, X) \ge
    (\alpha/2)|U| \tilde np - 2\eps^{4} |U| \tilde np \ge (\alpha/4)|U| \tilde
    np,
  \]
  completing the proof.
\end{proof}

\begin{proof}[Proof of~\ref{lemma:medium-expand} in Lemma~\ref{lem:everything}]
  Let $U' \subseteq U$ be a subset of vertices in $U$ defined as
  \[
    U' := \{ v \in U : \deg_{F_{23}}(v, W_3) < (1/3 + \alpha/2)\tilde np \}.
  \]
  Assume for the moment that we can show $e_{F_{12}}(U', W_1) \le \eps
  |F_{12}|$. Then we would have
  \begin{equation}\label{eq:large-dont-shrink-eq1}
    (1 - \eps)|F_{12}| \le e_{F_{12}}(U \setminus U', W_1) \le |U \setminus U'|
    \cdot \frac{\tilde np}{3},
  \end{equation}
  where the upper bound follows from the assumption on $\deg_{F_{12}}(v, W_1)$.
  From this we get that $|U\setminus U'| \ge \eps^{-4}\log n / p$ (with room to
  spare). Therefore, by the definition of $U'$ and \ref{p:unif-density} applied
  for $X$ and $U \setminus U'$ (as $Y$) we obtain
  \begin{align*}
    e_{F_{23}}(U, W_3 \setminus X)
    &\ge e_{F_{23}}(U \setminus U', W_3) - e_{F_{23}}(U \setminus U', X) \\
    &\geq |U \setminus U'|\big((1/3+\alpha/2)\tilde np - (1+\eps)|X|p\big) \geq
    |U \setminus U'|(1/3 + \alpha/2 - 2\eps^4)\tilde np,
  \end{align*}
  where the last step follows from $|X| \leq \eps^4\tilde n$. This, together
  with \eqref{eq:large-dont-shrink-eq1} shows
  \[
    e_{F_{23}}(U, W_3 \setminus X) \geq |U \setminus U'| \cdot (1/3 +
    \alpha/4)\tilde np \geq 3(1 - \eps)|F_{12}| (1/3 + \alpha/4) \geq (1 +
    \alpha/4)|F_{12}|,
  \]
  as desired. It remains to prove $e_{F_{12}}(U', W_1) \le \eps |F_{12}|$.
  Towards a contradiction assume this is not the case.

  For every $v \in U'$ let $L_v := N_{F_{12}}(v, W_1)$ and $R_v := N_{G -
  F_{23}}(v, W_3)$. Observe that by \ref{p:good-deg} and the definition of $U'$
  we have
  \begin{equation}\label{eq:large-dont-shrink-eq1-2}
    |R_v| \ge (1/3 + \alpha/2)\tilde np.
  \end{equation}
  Moreover, the definition of $R_v$ gives $e_G(L_v, R_v) = 0$. On the other
  hand, property \ref{p:unif-density}, as $|L_v|, |R_v| \geq \eps^{-4}\log n/p$,
  states $e_{\Gamma}(L_v, R_v) = (1 \pm \eps)|L_v||R_v|p$. Let
  \begin{equation}\label{eq:large-dont-shrink-eq1-3}
    L_v' := \{ w \in L_v : \deg_{\Gamma}(w, R_v) \ge (1 - \eps)|R_v|p \}.
  \end{equation}
  It immediately follows from \ref{p:unif-density} that $|L_v'| \ge (1-\eps)
  |L_v|$. Hence, the assumption $e_{F_{12}}(U', W_1) \ge \eps |F_{12}|$ implies
  \[
    \sum_{v \in U'} |L_v'| \ge \sum_{v \in U'} (1 - \eps)|L_v| = (1-\eps)
    e_{F_{12}}(U', W_1) \geq (1 - \eps) \cdot \eps|F_{12}| \geq \frac{\eps^{-3}
    \log n}{p} \cdot \tilde n.
  \]
  Consequently, by averaging over vertices in $W_1$, there is a vertex $w \in
  W_1$ and a set $U_w \subseteq U'$ such that $|U_w| \geq \eps^{-3}\log n/p$ and
  $w \in L_v'$ for all $v \in U_w$.

  Set $T := \bigcup_{v \in U_w} N_\Gamma(w, R_v)$. Note that all vertices in $T$
  are connected to $w$ in $\Gamma$, however none is connected to $w$ in $G$ due
  to the way $R_v$'s are defined. For all $v \in U_w$ we have
  \[
    \deg_G(v, T) \ge |R_v \cap T| \ge |N_\Gamma(w, R_v)|
    \osref{\eqref{eq:large-dont-shrink-eq1-3}} \ge (1-\eps)|R_v| p
    \osref{\eqref{eq:large-dont-shrink-eq1-2}} \ge (1/3 + \alpha/4)\tilde np^2.
  \]
  Therefore,
  \[
    e_\Gamma(U_w, T) \ge e_G(U_w, T) \ge |U_w| \cdot (1/3 + \alpha/4)\tilde
    np^2.
  \]
  Since $|U_w| \geq \eps^{-3}\log n/p$, we get from \ref{p:unif-density} applied
  for $U_w$ (as $X$) and $T$ (as $Y$) that
  \[
    |T| \ge \frac{(1/3+\alpha/4)}{1+\eps}\tilde np \ge (1/3+\alpha/8)\tilde np.
  \]
  Lastly, as $T \subseteq N_{\Gamma-G}(w, W_3)$, \ref{p:unif-deg} and
  \ref{p:good-deg} provide the desired contradiction.
\end{proof}

\begin{proof}[Proof of~\ref{lemma:large-to-all} in Lemma~\ref{lem:everything}]
  From \ref{p:unif-deg} we get that $|U| \geq |F_{12}|/((1 + \eps)\tilde np)
  \geq \eps^{-9}\log n/p^2$. Let $U' \subseteq U$ be defined as
  \begin{equation}\label{eq:def_u_prim}
    U' := \{ v \in U : \deg_{F_{23}}(v, W_3) < (1 - \eps^3) \deg_G(v, W_3) \}.
  \end{equation}
  Assume for the moment that we can show $e_G(U', W_3) \le \eps^3 \cdot e_G(U,
  W_3)$. Then
  \[
    |U'| \cdot (2/3 + \alpha/2)\tilde np \osref{\ref{p:good-deg}}\le e_G(U',W_3)
    \le \eps^3 \cdot e_G(U, W_3) \osref{\ref{p:unif-density}}\le \eps^3 \cdot
    (1+\eps)|U|\tilde np.
  \]
  This implies $|U'| \leq 2\eps^3 |U|$. Since $|X| \leq \eps^4 \tilde n$, we
  deduce from \ref{p:unif-density} that there can be at most $\eps^{-3}\log n/p
  \leq \eps^3 |U|$ vertices in $U$ with degree into $X$ larger than $2\eps^4
  \tilde np$. Let $L$ be the vertices in $U \setminus U'$ that do not have this
  property. Therefore, $|L| \geq |U| - |U'| - \eps^3 |U| \geq (1 - 3\eps^3)|U|$
  and for all $v \in L$
  \[
    \deg_{F_{23}}(v, W_3 \setminus X) \osref{\eqref{eq:def_u_prim}} \geq (1 -
    \eps^3) \deg_G(v, W_3) - 2\eps^4 \tilde np \geq (1 - 2\eps^3) \deg_G(v,
    W_3).
  \]
  Moreover, by definition of $U'$ and \ref{p:unif-density} applied for $X$ and
  $U \setminus U'$ (as $Y$), since $|X| \leq \eps^4 \tilde n$, we obtain
  \begin{align*}
    e_{F_{23}}(U, W_3 \setminus X) &\geq e_{F_{23}}(U \setminus U', W_3) -
    e_{F_{23}}(U \setminus U', X) \\
    &\geq \Big( \sum_{v \in U \setminus U'} (1 - \eps^3) \deg_{G}(v, W_3) \Big)
    - 2\eps^4 |U \setminus U'| \tilde np \osref{\ref{p:good-deg}}\geq \sum_{v
    \in U \setminus U'} (1 - 2\eps^3) \deg_{G}(v, W_3).
  \end{align*}
  This together with the assumption $e_G(U', W_3) \leq \eps^3 \cdot e_G(U, W_3)$
  further shows
  \begin{align*}
    e_{F_{23}}(U, W_3 \setminus X) &\geq \sum_{v \in U \setminus U'}
    (1-2\eps^3)\deg_G(v, W_3) \geq (1-2\eps^3)e_G(U \setminus U', W_3) \\
    &\geq (1-2\eps^3)(1-\eps^3)e_G(U, W_3) \geq (1-\eps^2)e_G(U, W_3),
  \end{align*}
  as desired. It remains to prove $e_G(U', W_3) \le \eps^3 \cdot e_G(U, W_3)$.
  Towards a contradiction assume this is not the case.

  For every $v \in U'$ let $L_v := N_{F_{12}}(v, W_1)$ and $R_v := N_{G -
  F_{23}}(v, W_3)$. Observe that by \ref{p:good-deg} and the definition of $U'$
  we have $|R_v| \ge \eps^3 \cdot \deg_G(v, W_3) \geq (2\eps^3/3)\tilde np$.
  Moreover, the definition of $R_v$ gives $e_G(L_v, R_v) = 0$. On the other
  hand, property \ref{p:unif-density}, as $|L_v|, |R_v| \geq \eps^{-4}\log n/p$,
  states $e_{\Gamma}(L_v, R_v) = (1 \pm \eps)|L_v||R_v|p$. Let
  \begin{equation}\label{eq:rv_def}
    R_v' := \{ w \in R_v : \deg_{\Gamma}(w, L_v) \ge (1-\eps)|L_v|p \}.
  \end{equation}
  It immediately follows from \ref{p:unif-density} that $|R_v'| \ge (1 -
  \eps)|R_v|$. Hence, the assumption $e_G(U', W_3) > \eps^3 \cdot e_G(U, W_3)$
  implies
  \[
    \sum_{v \in U'} |R_v'| \ge \sum_{v \in U'} (1 - \eps)|R_v| \geq \sum_{v \in
    U'} (1 - \eps) \cdot \eps^3 \cdot \deg_G(v, W_3) \geq \frac{\eps^3}{2}
    e_G(U', W_3) > \frac{\eps^6}{2} e_G(U, W_3).
  \]
  Next, using \ref{p:good-deg-con}, we know that $(1 + \eps) e_G(U, W_3) \geq (1
  - \eps) e_G(U, W_1)$ and thus
  \[
    \sum_{v \in U'} |R_v'| \geq \frac{\eps^6}{4} e_G(U, W_1) \geq
    \frac{\eps^{6}}{4} |F_{12}| \geq \frac{\eps^{-3}\log n}{p} \cdot \tilde n.
  \]
  Consequently, by averaging over vertices in $W_3$, there is a vertex $w \in
  W_3$ and a set $U_w \subseteq U'$ such that $|U_w| \geq \eps^{-3} \log n/p$
  and $w \in R_v'$ for all $v \in U_w$.

  Set $T := \bigcup_{v \in U_w} N_\Gamma(w, L_v)$. Note that all vertices in $T$
  are connected to $w$ in $\Gamma$, however none is connected to $w$ in $G$. For
  all $v \in U_w'$ we now have
  \[
    \deg_G(v, T) \geq |L_v \cap T| \ge |N_\Gamma(w, L_v)|
    \osref{\eqref{eq:rv_def}} \geq (1-\eps)|L_v|p \geq (1/3-\eps/3)\tilde np^2,
  \]
  where the last inequality follows from the assumption $\deg_{F_{12}}(v, W_1)
  \ge \tilde np/3$. Therefore,
  \[
    e_\Gamma(U_w, T) \geq e_G(U_w, T) \ge |U_w| \cdot (1/3 - \eps/3)\tilde np^2.
  \]
  Since $|U_w| \geq \eps^{-3}\log n/p$ we get from \ref{p:unif-density} applied
  for $U_w$ (as $X$) and $T$ (as $Y$) that
  \[
    |T| \ge \frac{1/3 - \eps/3}{1 + \eps} \tilde np \ge (1/3 - 2\eps)\tilde np.
  \]
  Lastly, as $T \subseteq N_{\Gamma - G}(w, W_1)$, \ref{p:unif-deg} and
  \ref{p:good-deg} provide the desired contradiction.
\end{proof}

\begin{proof}[Proof of~\ref{lemma:large-stay-large} in Lemma~\ref{lem:everything}]
  Note that $|F_{12}| \geq |U| \cdot \tilde np/3 \geq \eps^{-10} \tilde n\log
  n/p$. By part $\ref{lemma:large-to-all}$ of this lemma, we get
  \begin{equation}\label{eq:full-graph-eq1}
    e_{F_{23}}(U, W_3 \setminus X) \geq (1-\eps^2)e_G(U, W_3).
  \end{equation}
  Let
  \[
    L := \{ u \in W_3 \setminus X : \deg_{F_{23}}(u, W_2) \ge
    (1/3+\alpha/2)\tilde np \}.
  \]
  We aim to show that $|L| \geq (1 - \eps)\tilde n$. Assume towards a
  contradiction that this is not the case. Together with \ref{p:good-deg} and
  the assumption $|X|\le \eps^4 \tilde n$ this implies that there exists a set
  $Q \subseteq W_3 \setminus (X \cup L)$
  of size
  $|Q| \ge (\eps/2)\tilde n$
  such that for each $v \in Q$ we have
  \begin{equation}\label{eq:full-graph-eq2}
    \deg_{G - F_{23}}(v, W_2) \geq (1/3 + \alpha/2)\tilde np.
  \end{equation}
  Let $Q' := \{ v\in Q : \deg_{G - F_{23}}(v, U) \ge (\alpha/4)\tilde np\}$. If
  $|Q'| \ge (\eps/4)\tilde n$ then $e_{G - F_{23}}(U, Q') \ge
  (\eps\alpha/16)\tilde n^2p,$ which together with \ref{p:unif-density} applied
  for $U$ (as $X$) and $Q'$ (as $Y$) contradicts \eqref{eq:full-graph-eq1}.
  Therefore, $|Q'| \le (\eps/2)\tilde n$. From \ref{p:unif-density} we then get
  \begin{equation}\label{eq:full-graph-eq3}
    e_{G - F_{23}}(W_2 \setminus U, Q \setminus Q') \le (1 +
    \eps)\max\{\eps^{-3} \log n/p, |W_2 \setminus U|\} \cdot |Q\setminus Q'| p.
  \end{equation}
  On the other hand, from \eqref{eq:full-graph-eq2} and the definition of $Q$
  and $Q'$ we have
  \[
    e_{G - F_{23}}(W_2 \setminus U, Q\setminus Q') \ge (1/3 + \alpha/4)
    |Q\setminus Q'| \tilde np.
  \]
  Together with \eqref{eq:full-graph-eq3} we deduce
  \[
    \max\{\eps^{-3} \log n/p, |W_2 \setminus U|\} \ge
    \frac{1/3+\alpha/4}{1+\eps}\tilde n \ge (1/3+\alpha/8)\tilde n,
  \]
  which contradicts the assumption $|U| \ge (2/3)\tilde n$.
\end{proof}

\begin{proof}[Proof of~\ref{lemma:sparse-expand-or} in Lemma~\ref{lem:everything}]
  Let $S := \{ v \in U : \deg_{F_{12}}(v, W_1) < \eps^{-4} \log n/p \}$ and let
  $F_S \subseteq F_{12}$ be the subset of edges that are incident to $S$.
  Observe that if $|F_S| < \eps|F_{12}|$ we are done. So assume otherwise. By
  using part~$\ref{lemma:small-expand}$ of this lemma with $F_S$ (as $F_{12}$)
  we get
  \[
    e_{F_{23}}(W_2, W_3 \setminus X) \geq e_{F_{23}}(S, W_3 \setminus X) \ge
    \eps^{-4}|F_S| \ge \eps^{-3}|F_{12}|,
  \]
  contradicting the starting assumption $e_{F_{23}}(W_2, W_3 \setminus X) <
  \eps^{-3}|F_{12}|$.
\end{proof}

\begin{proof}[Proof of~\ref{lemma:dense-expand-or} in Lemma~\ref{lem:everything}]
  Define sets $S, M, L \subseteq W_2$ (which may be thought of as `small',
  `medium', and `large') as
  \begin{align*}
    S &= \{ v \in W_2 : 1 \le \deg_{F_{12}}(v, W_1) \le \eps^{-4}\log n/p \}, \\
    M &= \{ v \in W_2 : \eps^{-4}\log n/p \le \deg_{F_{12}}(v, W_1) \le \tilde
    np/3 \}, \\
    L &= \{ v \in W_2 : \tilde np/3 < \deg_{F_{12}}(v, W_1) \},
  \end{align*}
  and denote by $F_S$, $F_M$, and $F_L$ the subsets of edges in $F_{12}$
  incident to $S$, $M$, and $L$, respectively. Note that $F_S$, $F_M$, and $F_L$
  partition the set $F_{12}$. We claim that
  \begin{enumerate}[(i), font=\itshape]
    \item\label{8-small} if $|F_S| \ge \eps^4 |F_{12}|$, then $e_{F_{23}}(S, W_3
      \setminus X) \ge (1 + \alpha/4)|F_S|$,
    \item\label{8-medium} if $|F_M| \ge \eps^4 |F_{12}|$, then $e_{F_{23}}(M,
      W_3 \setminus X) \ge (1 + \alpha/4)|F_M|$, and
    \item\label{8-large} if $|F_L| \ge \eps^4 |F_{12}|$, then $e_{F_{23}}(L, W_3
      \setminus X) \ge (1 - 3\eps)|F_L|$.
  \end{enumerate}
  To see this assume first that $|F_S| \ge \eps^4|F_{12}|$. Then $|F_S| \ge
  \eps^5 \tilde n^2p \ge \eps^{-17}\tilde n \log^2 n$ and we can thus apply
  part~$\ref{lemma:small-expand}$ of this lemma to $S$ (as $U$) and $F_S$ (as
  $F_{12}$) to get
  \[
    e_{F_{23}}(S, W_3 \setminus X) \ge \eps^{-4}|F_S| \ge (1 + \alpha/4)|F_S|.
  \]
  Next, assume $|F_M| \ge \eps^4|F_{12}|$. Then again $|F_M| \ge \eps^{-5}\tilde
  n\log n/p$ and we can apply part~$\ref{lemma:medium-expand}$ of this lemma to
  $M$ (as $U$) and $F_M$ (as $F_{12}$) to get
  \[
    e_{F_{23}}(M, W_3 \setminus X) \ge (1 + \alpha/4)|F_M|.
  \]
  Lastly, assume $|F_L| \ge \eps^4|F_{12}|$. Then $|F_L| \ge \eps^{-10} \tilde n
  \log n/p$ and we can apply part~$\ref{lemma:large-to-all}$ of this lemma to
  $L$ (as $U$) and $F_L$ (as $F_{12}$) to get
  \begin{equation}\label{eq:dense-expand-or-eq1}
    e_{F_{23}}(L, W_3 \setminus X) \ge (1 - \eps^2) e_G(L, W_3).
  \end{equation}
  Moreover, from \ref{p:good-deg-con} applied for all $v \in L$ and $W_1, W_3$
  (as $W$), we have $(1 + \eps)e_G(L, W_3) \geq (1 - \eps)e_G(L, W_1)$.
  Therefore, together with \eqref{eq:dense-expand-or-eq1}:
  \[
    e_{F_{23}}(L, W_3 \setminus X) \ge (1 - \eps^2)e_G(L, W_3) \geq (1 -
    \eps^2)\frac{1 - \eps}{1 + \eps} e_G(L, W_1),
  \]
  from which the third property follows as, trivially, $e_G(L, W_1) \geq |F_L|$.

  Having these three properties at hand, we are ready to prove the lemma. If
  $|F_L| \ge (1 - \mu)|F_{12}|$ we are done, so assume the contrary. Observe
  that this implies that at least one of $|F_S|, |F_M|$ has size at least
  $\eps^4|F_{12}|$. If $|F_L|$ is strictly smaller than $\eps^4|F_{12}|$, then
  either at least one of $|F_S|$ and $|F_M|$ has size at least
  $(1-2\eps^4)|F_{12}|$ or both have size at least $\eps^4|F_{12}|$. Thus by
  $\ref{8-small}$ and $\ref{8-medium}$ we get
  \[
    e_{F_{23}}(W_2, W_3 \setminus X) \geq \max\{ (1 + \alpha/4)(1 -
    2\eps^4)|F_{12}|, (1 + \alpha/4)(1 - \eps^4)|F_{12}|\} > (1 + \alpha/8)
    |F_{12}|,
  \]
  which is a contradiction to the assumption of the lemma. Therefore, $|F_L|
  \geq \eps^4|F_{12}|$ and at least one of $|F_S|$ and $|F_M|$ has size at least
  $|F_{12}|-|F_L|-\eps^4|F_{12}|$ or both have size at least $\eps^4|F_{12}|$.
  Thus, again by $\ref{8-small}$ and $\ref{8-medium}$,
  \begin{align*}
    e_{F_{23}}(W_2, W_3 \setminus X) &\geq \max\big\{(1 + \alpha/4)(|F_{12}| -
    \eps^4|F_{12}| - |F_L|), (1 + \alpha/4)(|F_{12}| - |F_L|)\big\} + (1 -
    3\eps)|F_L| \\
    &\ge (1 + \alpha/4)(1 - \eps^4)|F_{12}| - (3\eps + \alpha/4)|F_L|.
  \end{align*}
  Using our assumption $|F_L| < (1 - \mu)|F_{12}|$ and $\mu \geq 32\eps/\alpha$
  this implies $e_{F_{23}}(W_2, W_3 \setminus X) \ge (1 + \mu\alpha/8)|F_{12}|$,
  again contradicting the assumption of the lemma.
\end{proof}

\begin{proof}[Proof of~\ref{lemma:dense-doesnt-drop} in Lemma~\ref{lem:everything}]
  Clearly, if $e_{F_{23}}(W_2, W_3 \setminus X) \ge (1 + 4\eps)|F_{12}|$ we are
  done. Let us assume the opposite. By applying
  part~$\ref{lemma:dense-expand-or}$ of this lemma with $32\eps/\alpha$ (as
  $\mu$) we obtain a set $L \subseteq W_2$ such that for each $v \in L$ it holds
  that $\deg_{F_{12}}(v, W_1) > \tilde np/3$ and
  \begin{equation}\label{eq:dense-doesnt-drop-eq1}
    e_{F_{12}}(W_1, L) \ge (1-32\eps/\alpha)|F_{12}| \geq
    (1-\eps^{2/3})|F_{12}|.
  \end{equation}
  Next, we use part~$\ref{lemma:large-to-all}$ of this lemma with $L$ (as
  $U$) and $E_{F_{12}}(W_1, L)$ (as $F_{12}$) to conclude
  \begin{equation}\label{eq:dense-doesnt-drop-eq2}
    e_{F_{23}}(L, W_3 \setminus X) \ge (1 - \eps^2) e_G(L, W_3).
  \end{equation}
  By taking together \eqref{eq:dense-doesnt-drop-eq1} and
  \eqref{eq:dense-doesnt-drop-eq2}, we finally get
  \begin{align*}
    e_{F_{23}}(L, W_3 \setminus X) &\ge (1 - \eps^2) e_G(L, W_3)
    \osref{\ref{p:good-deg-con}}\ge (1 - \eps^2)(1 - 2\eps) e_G(W_1, L) \\
    &\ge (1 - 3\eps) e_{F_{12}}(W_1, L)
    \osref{\eqref{eq:dense-doesnt-drop-eq1}}\ge (1 - 3\eps)(1 -
    \eps^{2/3})|F_{12}| \ge (1 - \sqrt\eps)|F_{12}|,
  \end{align*}
  and the assertion follows.
\end{proof}

\section{Proof of Theorem \ref{thm:main-theorem}}\label{sec:K3-res}

The proof of Theorem~\ref{thm:main-theorem} is split into two natural parts. In
Theorem~\ref{thm:K3-res-ub} we show that the $K_3$-resilience of $\Gnp$ w.r.t.\
the containment of $C^2_n$ is w.h.p.\ at most $5/9 + \alpha$, for any $\alpha >
0$. Next, in Theorem~\ref{thm:K3-res-lb} we show that the $K_3$-resilience is
w.h.p.\ at least $5/9 - \alpha$, for any $\alpha > 0$. Both of the theorems rely
on the following fact: for every $\eps > 0$ a random graph $\Gamma \sim \Gnp$
w.h.p.\ has the property that each vertex is contained in $(1 \pm
\eps)\binom{n}{2}p^3$ triangles, provided that $p \gg n^{-2/3} \log^{1/3} n$
(see, e.g.~\cite[Theorem 8.5.4]{alon2016probabilistic}).

The proof of the upper bound of $K_3$-resilience stems from a simple
construction and an application of Janson's inequality. We actually show that
w.h.p.\ there exists a subgraph $G$ of $\Gnp$ such that each vertex is contained
in $(4/9 - \gamma)\binom{n}{2}p^3$ triangles and $G$ does not contain a family
of more than $(1 - \gamma)n/3$ vertex-disjoint triangles. This is sufficient
since $C^2_n$ contains $\floor{n/3}$ vertex-disjoint triangles.

\begin{theorem}\label{thm:K3-res-ub}
  For every $\gamma > 0$, there exists a positive constant $C := C(\gamma)$ such
  that for all $p \ge Cn^{-2/3}\log^{1/3}n$ a random graph $\Gamma \sim \Gnp$
  w.h.p.\ contains a spanning subgraph $G \subseteq \Gamma$ in which each vertex
  is contained in at least $(4/9 - \gamma) \binom{n}{2}p^3$ triangles and such
  that $G$ does not contain a family of more than $(1 - \gamma)n/3$
  vertex-disjoint triangles.
\end{theorem}
\begin{proof}
  Let $V(G) = V_1 \cup V_2$ be a partition of the vertex set of $G$ such that
  \[
    |V_1| = \ceil*{\Big( \frac{1}{3} + \frac{2\gamma}{3} \Big)n} \qquad
    \text{and} \qquad |V_2| = n - |V_1|.
  \]
  Observe that $|V_2| = (2/3 - 2\gamma/3 - o(1))n$. Furthermore, let $G$ be the
  graph obtained from $\Gamma$ by removing all edges with both endpoints in
  $V_1$. For a vertex $v \in V(G)$ let $\calT_v$ denote the family of all
  triangles in $K_n$ which contain $v$ and do not have more than one vertex in
  $V_1$. Set
  \[
    X = \sum_{T \in \calT_v} X_T, \qquad \mu = \E[X], \qquad \text{and} \qquad
    \Delta = \sum_{\substack{T_1, T_2 \in \calT_v \\ T_1 \sim T_2}} \E[X_{T_1}
    X_{T_2}],
  \]
  where $X_T$ is an indicator random variable for the event $\{T \subseteq
  \Gamma\}$. Note that $X$ is a random variable counting the number of triangles
  in $G$ that contain $v$. We aim to show that
  \begin{equation} \label{eq:triangles}
    \Pr \bigg[ X < (4/9 - \gamma)\binom{n}{2}p^3 \bigg] \le e^{-2\log n},
  \end{equation}
  for every $v \in V(G)$.

  Let us estimate $\mu$ and $\Delta$ for an arbitrary vertex $v \in V_1$. We
  have
  \begin{equation}\label{eq:triangles-mu}
    \mu = \sum_{T \in \calT_v} \E[X_T] = \binom{|V_2|}{2} p^3 \ge \big( 2/3 -
    2\gamma/3 - o(1) \big)^2 \frac{n^2p^3}{2} \ge (4/9 - 8\gamma /9)
    \binom{n}{2}p^3.
  \end{equation}
  Note that if two triangles $T_1, T_2 \in \calT_v$ do not share an edge, they
  are independent and thus $T_1 \not \sim T_2$. Therefore, we can bound $\Delta$
  as follows:
  \begin{equation}\label{eq:triangles-delta}
    \Delta = \sum_{\substack{T_1, T_2 \in \calT_v \\ T_1 \sim T_2}} \E[X_{T_1}
    X_{T_2}] \le |V_2|^3 p^5 \le \frac{5\mu^2}{|V_2|p}.
  \end{equation}
  Let us choose $\eps$ such that $(1 - \eps)(4/9 - 8\gamma/9) \ge 4/9 - \gamma$
  and apply Janson's inequality (Theorem~\ref{thm:janson}) with $\eps$ (as
  $\gamma$) to obtain
  \begin{align*}
    \Pr\bigg[ X < (4/9 - \gamma)\binom{n}{2}p^3 \bigg] &\le \Pr[ X < (1 -
    \eps)\mu] \le e^{-\frac{\eps^2 \mu^2}{2(\mu + \Delta)}} \le
    e^{-\frac{\eps^2 \mu^2}{4 \max\{\mu, \Delta\}}} \\
    &\le \max\{ e^{-\eps^2 \mu/4}, e^{-\eps^2 |V_2|p/20} \},
  \end{align*}
  where the last step follows from \eqref{eq:triangles-mu} and
  \eqref{eq:triangles-delta}. Since $|V_2|p \gg \log n$ and $\mu \geq (4/9 -
  8\gamma/9 - o(1))C^3 \log n$, by choosing $C$ large enough with respect to
  $\eps$ and $\gamma$ we obtain~\eqref{eq:triangles} for every vertex in $V_1$.
  The proof of \eqref{eq:triangles} for the case when $v \in V_2$ follows
  analogously and is omitted. By \eqref{eq:triangles} and the union bound over
  all vertices we get that w.h.p.\ each vertex $v\in V(G)$ is contained in at
  least $(4/9 - \gamma) \binom{n}{2}p^3$ triangles in $G$.

  Let $\cF$ be the largest family of vertex-disjoint triangles in $G$. Since $G$
  does not contain an edge with both endpoints in $V_1$, there is no triangle in
  $\cF$ with more than one vertex in $V_1$. This implies that $|V(\cF) \cap V_2
  | \ge 2 |V(\cF) \cap V_1|$, which further shows
  \[
    |V(\cF)| \le 3/2 \cdot |V_2| \le 3/2 \cdot (2/3 - 2\gamma/3) n \le (1 -
    \gamma )n,
  \]
  completing the proof.
\end{proof}

The following lower bound on the $K_3$-resilience is proven by a reduction to
Theorem~\ref{thm:square-res-codegree}.

\begin{theorem}\label{thm:K3-res-lb}
  For every $\gamma > 0$, there exists a positive constant $C := C(\gamma)$ such
  that a random graph $\Gamma \sim \Gnp$ w.h.p.\ satisfies the following,
  provided that $p \ge Cn^{-1/2}\log^3n$. Every spanning subgraph $G \subseteq
  \Gamma$ in which each vertex is contained in at least
  $(4/9+\gamma)\binom{n}{2}p^3$ triangles contains the square of a Hamilton
  cycle.
\end{theorem}
\begin{proof}
  Choose $\eps > 0$ such that $(1 + \eps)(4/9 + 5\gamma/12) < 4/9 + \gamma/2$
  and $\eps \leq \gamma/4$, and set $C = 10/\eps$. Let $G \subseteq \Gamma$ be
  an arbitrary spanning subgraph of $\Gamma$ such that each vertex of $G$ is
  contained in at least $(4/9 + \gamma)\binom{n}{2}p^3$ triangles. Let $G'
  \subseteq G$ be a subgraph obtained by removing each edge of $G$ which is
  contained in fewer than $\eps np^2$ triangles. We aim to show that $G'$ has
  minimum degree at least $(2/3 + \gamma/4)np$. If this is the case then by
  Theorem~\ref{thm:square-res-codegree} we are done.

  First, we show that by removing the edges which are contained in only a few
  triangles, we did not significantly change the overall number of triangles
  each vertex is in. As $np = \omega(\log n)$, w.h.p.\ we have that every $v \in
  V(\Gamma)$ satisfies $\deg_{\Gamma}(v) \leq (1 + \eps)np$. Moreover, each edge
  in $E(G) \setminus E(G')$ is contained in at most $\eps np^2$ triangles, which
  implies that we did not remove more than $\eps np^2 \cdot (1 + \eps) np \le
  2\eps n^2p^3$ triangles from $G$ touching a single vertex. Since $G$ has the
  property that each vertex is contained in at least $(4/9 + \gamma)
  \binom{n}{2}p^3$ triangles, the previous observation and the choice of $\eps$
  show that in $G'$ each vertex is in at least $(4/9 + \gamma/2)\binom{n}{2}p^3$
  triangles.

  In order to finish the argument we use the following claim, whose proof
  follows below.

  \begin{claim}\label{claim:edges-neigh}
    The following holds w.h.p. For every vertex $v \in V(\Gamma)$ and every
    subset $S \subseteq N_\Gamma(v)$ of size $|S| \ge (2/3)np$, we have $e(S)
    \leq (1 + \eps)\binom{|S|}{2}p$.
  \end{claim}

  With this claim we can easily complete the proof of the theorem. Suppose for
  contradiction that $v \in V(G')$ is a vertex with degree smaller than $(2/3 +
  \gamma/4)np$. By the claim above and the choice of $\eps$ we have that $v$ is
  contained in at most
  \[
    (1 + \eps) (2/3 + \gamma/4)^2 \frac{n^2 p^3}{2} \le (1 + \eps)(4/9 +
    5\gamma/12) \frac{n^2 p^3}{2} < (4/9 + \gamma/2) \binom{n}{2}p^3,
  \]
  triangles---a contradiction.
\end{proof}

\begin{proof}[Proof of Claim~\ref{claim:edges-neigh}]
  It suffices to show that the claim holds for a fixed vertex $v$ with
  probability at least $1 - e^{-\omega(\log n)}$, as the claim then follows by
  the union bound over all vertices. Let $v$ be a vertex from $\Gamma$ and let
  $S \subseteq N_\Gamma(v)$. Recall, we have that $\deg_\Gamma(v) \le (1 +
  \eps)np$ with probability at least $1 - e^{-\omega(\log n)}$. Similarly we
  have
  \begin{equation}\label{eq:s-edges}
    \Pr \bigg[ e(\Gamma[S]) > (1 + \eps) \binom{|S|}{2}p \bigg] \le e^{- \eps^2
    \binom{|S|}{2} p/3}.
  \end{equation}
  By using \eqref{eq:s-edges} and the union bound, the probability that the
  assertion of the claim fails is at most
  \begin{align*}
    \sum_{s = (2/3)np}^{(1 + \eps) np} \binom{(1 + \eps) np}{s} e^{-\eps^2 s(s -
    1)p/6}
    &\le \sum_{s = (2/3)np}^{(1 + \eps) np} \Big( \frac{2 e np}{s} \Big)^s
    e^{-\eps^2 s(s - 1)p/6} \\
    &\le \sum_{s = (2/3)np}^{(1 + \eps)np} (3e)^{np} e^{-\eps^2 s^2 p/12}
    \le \sum_{s = (2/3)np}^{(1 + \eps)np} (e^{3\log 3 - \eps^2 sp/12})^s \\
    & \le \sum_{s = (2/3)np}^{(1 + \eps) np} e^{-2s}
    \leq n \cdot e^{-\Omega(np)},
  \end{align*}
  where in the second to last inequality we used the fact that $\eps^2 s p/12
  \ge \eps^2 np^2/18 \ge 3\log 3$. Finally, since $np = \omega(\log n)$, the
  claim follows.
\end{proof}

\section{Definitions of some graphs}\label{sec:graph-definitions}

The following graphs are used often throughout the paper and we thus give their
definitions here, for easier reference later on. We note that most of these come
from or were inspired by similar definitions in~\cite{nenadov2019powers}.

An {\bf $\ell$-square-path}, denoted by $P^2_\ell$, is a graph defined on a
vertex set $\{v_1, \dotsc, v_\ell\}$ such that $v_i$ and $v_j$ are connected by
an edge if $1 \leq i < j \le i + 2$ (see Figure~\ref{fig:square-path}).

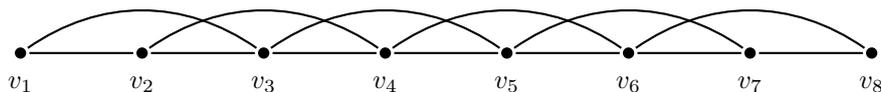
\begin{figure}[!htbp]
  \centering
  \begin{tikzpicture}[scale=0.8, >=stealth, shorten >=1pt, shorten <=1pt]
  \tikzstyle{blob} = [fill=black, circle, inner sep=1.5pt, minimum size=0.5pt]

  \foreach \i in {1,...,8} {
    \node[blob, label={[label distance=0.1cm]below:{\small $v_{\i}$}}] (\i) at (2*\i, 0) {};
  }

  \foreach \i in {1,...,7} {
    \draw[thick] let \n1={int(mod(\i+1, 10))} in (\i) to (\n1);
  }

  \foreach \i in {1,...,6} {
    \draw[thick, bend left=35] let \n1={int(mod(\i+2, 10))} in (\i) to (\n1);
  }
\end{tikzpicture}
  \caption{The $8$-square-path $P^{2}_{8}$.}
  \label{fig:square-path}
\end{figure}

Given a graph $G = (V, E)$ and $\mathbf{a}, \mathbf{b} \in V^2$, we say that $G$
contains a square-path {\em connecting} $\mathbf{a}$ to $\mathbf{b}$, if there
exists an $\ell \in \N$ and an embedding $g \colon V(P^2_\ell) \to V(G)$ such
that $\mathbf{a} = (g(v_1), g(v_2))$ and $\mathbf{b} = (g(v_{\ell - 1}),
g(v_\ell))$. Note that due to the fact that $\mathbf{a}$ and $\mathbf{b}$ are
(ordered) pairs of vertices, a path connecting $\mathbf{a}$ to $\mathbf{b}$ {\em
is not} the same as a path connecting $\mathbf{b}$ to $\mathbf{a}$. However, a
path connecting $\mathbf{a}$ to $\mathbf{b}$ is also a path connecting
$\overline{\mathbf{b}}$ to $\overline{\mathbf{a}}$. It is easy to see that one
can connect two square-paths in order to get a longer square-path.

\begin{proposition}\label{prop:embedding-square}
  Let $G$ be a graph and $\mathbf{a}, \mathbf{b}, \mathbf{c} \in V(G)^2$
  disjoint pairs of vertices. Suppose that in $G$ there exists a square-path
  connecting $\mathbf{a}$ to $\mathbf{b}$ and a square-path connecting
  $\mathbf{b}$ to $\mathbf{c}$ (and thus also $\overline{\mathbf{c}}$ to
  $\overline{\mathbf{b}}$) such that these paths are internally vertex-disjoint.
  Then the union of these two paths is a square-path that connects $\mathbf{a}$
  to $\mathbf{c}$. \qed
\end{proposition}

A {\bf $(b, \ell)$-pseudo-path} $S^b_\ell$, where $b \in \{1, 2\}$, is a graph
defined on the vertex set $\{u_1, \dotsc, u_\ell\}$ with the edge set
\[
  E(S^{b}_{\ell}) := \bigcup_{i = 2}^{\ell} \{u_{i - 1}, u_{i}\} \cup
  \bigcup_{\substack{i \in \{3, \dotsc, \ell\} \\ \text{$i$ is odd}}} \{u_{i -
  2}, u_{i}\} \cup
  \bigcup_{\substack{i \in \{4, \dotsc, \ell\} \\ \text{$i$ is even}}} \{u_{i -
  1 - b}, u_{i}\}.
\]
Observe that a $(1, \ell)$-pseudo-path is isomorphic to an $\ell$-square-path; a
$(2, \ell)$-pseudo-path is depicted in Figure~\ref{fig:alter-path}.

\begin{figure}[!htbp]
  \centering
  \begin{tikzpicture}[scale=0.7, >=stealth, shorten >=1pt, shorten <=1pt]
  \tikzstyle{blob} = [fill=black, circle, inner sep=1.5pt, minimum size=0.5pt]

  \foreach \i in {1,...,8} {
    \node[blob, label={[label distance=0.1cm]below:{\small $u_{\i}$}}] (\i) at (2*\i, 0) {};
  }

  \foreach \i in {1,...,7} {
    \draw[thick] let \n1={int(mod(\i+1, 10))} in (\i) to (\n1);
  }

  \foreach \i in {1,3,5} {
    \draw[thick, bend left=35] let \n1={int(mod(\i+2, 10))} in (\i) to (\n1);
  }

  \foreach \i in {1,3,5} {
    \draw[thick, bend left=35] let \n1={int(mod(\i+3, 10))} in (\i) to (\n1);
  }
\end{tikzpicture}
  \caption{The $(2, 8)$-pseudo-path $S^{2}_{8}$}
  \label{fig:alter-path}
\end{figure}
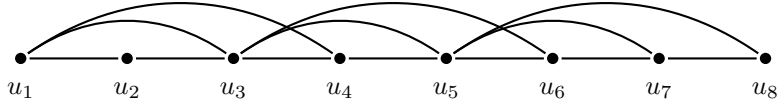

The notion of a $(b, \ell)$-pseudo-path connecting $\mathbf{a}$ to $\mathbf{b}$
is defined in a natural way, similarly as above.

An {\bf $\ell$-backbone-path} $B_{\ell}$, is a graph defined on the vertex set
\[
  W_{\ell} = \bigcup_{i \in [\ell]} \{w_{i, 1}, w_{i, 2}, w_{i, 3} ,w_{i, 4}\}.
\]
We set $\mathbf{w}_i^a = (w_{i, 1}, w_{i, 2})$ and $\mathbf{w}_i^b = (w_{i, 3},
w_{i, 4})$, for every $i \in [\ell]$. The edge set of $B_\ell$ is given by the
union of following graphs (see Figure~\ref{fig:backbone-path}):
\begin{itemize}
  \item edges $\{w_{1, 1}, w_{1, 2}\}$ and $\{w_{1, 3}, w_{1, 4}\}$;
  \item the $4$-square-path $(\mathbf{w}_i^a, \mathbf{w}_i^b)$ for every $2 \le
    i \le \ell$;
  \item the $4$-square-path $(\mathbf{w}_1^a, \overline{\mathbf{w}}_2^a)$;
  \item the $4$-square-path $(\overline{\mathbf{w}}_{i}^b,
    \overline{\mathbf{w}}_{i + 2}^a)$ for every $1 \le i \le \ell - 2$;
  \item the $4$-square-path $(\overline{\mathbf{w}}_{\ell - 1}^b,
    \mathbf{w}_{\ell}^b)$.
\end{itemize}

\begin{figure}[!htbp]
  \centering
  \begin{tikzpicture}[scale=0.5]
	\tikzstyle{blob} = [fill=black, circle, inner sep=1.5pt, minimum size=0.5pt]
	
	\foreach \i in {1,...,4} {
		\node[blob, label={[label distance=0.1cm]left:{\footnotesize $w_{2, \i}$}}] (2\i) at ({((2 * 3 + \i - 1)*360/30 - 288}:6) {};
	}

	\draw[thick]  (21) to (22);
	\draw[bend right=40] (21) to (23);
	\draw[thick]  (22) to (23);
	\draw[bend right=40] (22) to (24);
	\draw[thick]  (23) to (24);

	\foreach \i in {1,...,4} {
		\node[blob, label={below left:{\footnotesize $w_{4, \i}$}}] (4\i) at ({((4 * 3 + \i - 1)*360/30 - 288}:6) {};
	}

	\draw[thick]  (41) to (42);
	\draw[bend right=40] (41) to (43);
	\draw[thick]  (42) to (43);
	\draw[bend right=40] (42) to (44);
	\draw[thick]  (43) to (44);

	\foreach \i in {1,...,4} {
		\node[blob, label={below right:{\footnotesize $w_{5, \i}$}}] (5\i) at ({(4 * 6  - \i - 2)*360/30 - 288}:6) {};
	}
	\draw[thick]  (51) to (52);
	\draw[bend left=40] (51) to (53);
	\draw[thick]  (52) to (53);
	\draw[bend left=40] (52) to (54);
	\draw[thick]  (53) to (54);

	\foreach \i in {1,...,4} {
		\node[blob, label={[label distance=0.1cm]right:{\footnotesize $w_{3, \i}$}}] (3\i) at ({5 * 6  - \i - 2)*360/30 - 288}:6) {};
	}
	\draw[thick]  (31) to (32);
	\draw[bend left=40] (31) to (33);
	\draw[thick]  (32) to (33);
	\draw[bend left=40] (32) to (34);
	\draw[thick]  (33) to (34);

	\node[blob, label={[label distance=-0.1cm]above right:{\footnotesize $w_{1, 4}$}}] (04) at ({0 - 295}:6) {};
	\node[blob, label={[label distance=-0.05cm]above:{\footnotesize $w_{1, 3}$}}] (03) at ({1 * 360/30 - 295}:6) {};
	\node[blob, label={[label distance=-0.05cm]above:{\footnotesize $w_{1, 2}$}}] (02) at ({2 * 360/30 - 281}:6) {};
	\node[blob, label={[label distance=-0.1cm]above left:{\footnotesize $w_{1, 1}$}}] (01) at ({3 * 360/30 - 281}:6) {}; 

	\draw[thick]  (01) to (02);
	\draw[thick]  (03) to (04);

	\draw[bend left]  (01) to (22);
	\draw[bend left]  (02) to (22);
	\draw[bend left]  (02) to (21);

	\draw[bend left]  (23) to (41);
	\draw[bend left]  (23) to (42);
	\draw[bend left]  (24) to (42);

	\draw[bend left]  (43) to (54);
	\draw[bend left]  (43) to (53);
	\draw[bend left]  (44) to (53);

	\draw[bend left]  (51) to (33);
	\draw[bend left]  (52) to (33);
	\draw[bend left]  (52) to (34);

	\draw[bend left]  (31) to (03);
	\draw[bend left]  (32) to (03);
	\draw[bend left]  (32) to (04);
\end{tikzpicture}
  \caption{The graph $B_5$}
  \label{fig:backbone-path}
\end{figure}
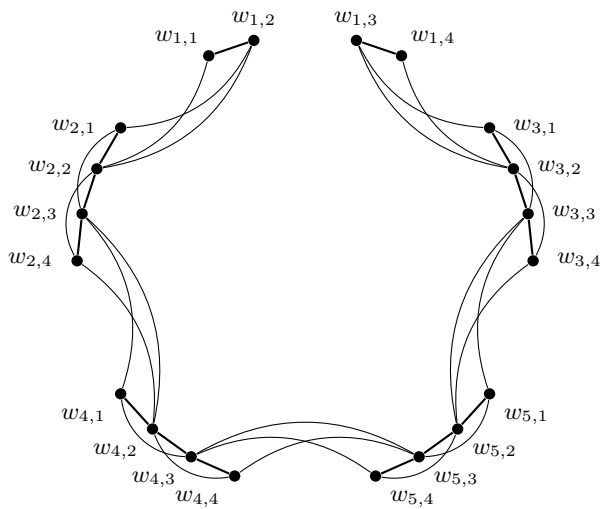
Given a graph $G$ and $\mathbf{a}, \mathbf{b} \in V^2$, we say that a
backbone-path {\em connecting $\mathbf{a}$ to $\mathbf{b}$} is an embedding $g
\colon V(B_\ell) \to V(G)$, for an appropriate $\ell \in \N$, such that
$\mathbf{a}= (g(w_{1, 2}), g(w_{1, 1}))$ and $\mathbf{b} = (g(w_{1, 4}), g(w_{1,
3}))$.

The connection between backbone-paths and pseudo-paths is given by the following
proposition.

\begin{proposition}\label{prop:embedding-backbone}
  Let $G$ be a graph, $\mathbf{a}, \mathbf{b}, \mathbf{c} \in V(G)^2$ disjoint
  pairs of vertices, and $\ell_1, \ell_2 \in \N$ are such that both $\ell_1$ and
  $\ell_1 + \ell_2 - 2$ are divisible by four. Suppose that in $G$ there exists
  a $(2, \ell_1)$-pseudo-path connecting $\mathbf{a}$ to $\mathbf{b}$ and a $(2,
  \ell_2)$-pseudo-path connecting $\mathbf{\overline c}$ to $\mathbf{b}$ such
  that these paths are internally vertex-disjoint. Then the union of these two
  paths is a backbone-path that connects $\mathbf{a}$ to $\mathbf{c}$.
\end{proposition}
\begin{proof}
  One easily verifies that Figure~\ref{fig:2-conn-path} describes an embedding
  of the two pseudo-paths whose union is a backbone-path. We omit the details.
  \begin{figure}[!htbp]
    \centering
    \begin{subfigure}[b]{.49\textwidth}
      \tikzset{>=latex}
\begin{tikzpicture}[scale=0.5]
	\tikzstyle{blob} = [fill=black, circle, inner sep=1.5pt, minimum size=0.5pt]

	\foreach \i in {1,...,4} {
		\node[blob, label={[label distance=0.1cm]left:{\footnotesize $w_{2, \i}$}}] (2\i) at ({((2 * 3 + \i - 1)*360/30 - 288}:6) {};
	}

	\draw[color=royalazure, very thick, dashed, ->]  (22) to (21);
	\draw[bend right=40, color=royalazure, very thick, dashed, ->] (21) to (23);
	\draw[very thin]  (22) to (23);
	\draw[very thin, bend right=40] (22) to (24);
	\draw[color=royalazure, very thick, dashed, ->]  (23) to (24);

	\foreach \i in {1,...,4} {
		\node[blob, label={below left:{\footnotesize $w_{4, \i}$}}] (4\i) at ({((4 * 3 + \i - 1)*360/30 - 288}:6) {};
	}

	\draw[color=royalazure, very thick, dashed, ->]  (42) to (41);
	\draw[bend right=40, very thin] (41) to (43);
	\draw[very thin]  (42) to (43);
	\draw[very thin, bend right=40] (42) to (44);
	\draw[very thin] (43) to (44);

	\foreach \i in {1,...,4} {
		\node[blob, label={below right:{\footnotesize $w_{5, \i}$}}] (5\i) at ({(4 * 6  - \i - 2)*360/30 - 288}:6) {};
	}
	\draw[very thin]  (52) to (51);
	\draw[very thin, bend left=40] (51) to (53);
	\draw[very thin]  (52) to (53);
	\draw[bend right=40, very thin] (54) to (52);
	\draw[very thin]  (53) to (54);

	\foreach \i in {1,...,4} {
		\node[blob, label={[label distance=0.1cm]right:{\footnotesize $w_{3, \i}$}}] (3\i) at ({5 * 6  - \i - 2)*360/30 - 288}:6) {};
	}
	\draw[very thin]  (32) to (31);
	\draw[very thin, bend left=40] (31) to (33);
	\draw[very thin]  (32) to (33);
	\draw[bend right=40, very thin] (34) to (32);
	\draw[very thin]  (33) to (34);

	\node[blob, label={[label distance=-0.1cm]above right:{\footnotesize $w_{1, 4}$}}] (04) at ({0 - 295}:6) {};
	\node[blob, label={[label distance=-0.05cm]above:{\footnotesize $w_{1, 3}$}}] (03) at ({1 * 360/30 - 295}:6) {};
	\node[blob, label={[label distance=-0.05cm]above:{\footnotesize $w_{1, 2}$}}] (02) at ({2 * 360/30 - 281}:6) {};
	\node[blob, label={[label distance=-0.1cm]above left:{\footnotesize $w_{1, 1}$}}] (01) at ({3 * 360/30 - 281}:6) {}; 

	\draw[color=royalazure, very thick, dashed, ->]  (02) to (01);
	\draw[very thin]  (04) to (03);

	\draw[bend left, color=royalazure, very thick, dashed, ->]  (01) to (22);
	\draw[very thin, bend left]  (02) to (22);
	\draw[very thin, bend left]  (02) to (21);

	\draw[very thin, bend left]  (23) to (41);
	\draw[very thin, bend left]  (23) to (42);
	\draw[bend left, color=royalazure, very thick, dashed, ->]  (24) to (42);

	\draw[very thin, bend left]  (43) to (54);
	\draw[very thin, bend left]  (43) to (53);
	\draw[very thin, bend left]  (44) to (53);

	\draw[very thin, bend left]  (51) to (33);
	\draw[very thin, bend left]  (52) to (33);
	\draw[very thin, bend left]  (52) to (34);

	\draw[very thin, bend left]  (31) to (03);
	\draw[very thin, bend left]  (32) to (03);
	\draw[very thin, bend left]  (32) to (04);
\end{tikzpicture}
      \caption{Blue arrows indicate the order of the vertices mapped by $g_1$.}
      \label{fig:embedding-g1}
    \end{subfigure}
    \hfill
    \begin{subfigure}[b]{.49\textwidth}
      \tikzset{>=latex}
\begin{tikzpicture}[scale=0.5]
	\tikzstyle{blob} = [fill=black, circle, inner sep=1.5pt, minimum size=0.5pt]
	
	\foreach \i in {1,...,4} {
		\node[blob, label={[label distance=0.1cm]left:{\footnotesize $w_{2, \i}$}}] (2\i) at ({((2 * 3 + \i - 1)*360/30 - 288}:6) {};
	}

	\draw[very thin]  (22) to (21);
	\draw[very thin, bend right=40] (21) to (23);
	\draw[very thin]  (22) to (23);
	\draw[bend left, very thin] (24) to (22);
	\draw[very thin]  (23) to (24);

	\foreach \i in {1,...,4} {
		\node[blob, label={below left:{\footnotesize $w_{4, \i}$}}] (4\i) at ({((4 * 3 + \i - 1)*360/30 - 288}:6) {};
	}

	\draw[color=royalred, very thick, dashed, ->]  (42) to (41);
	\draw[very thin, bend right=40] (41) to (43);
	\draw[very thin]  (42) to (43);
	\draw[bend left, color=royalred, very thick, dashed, ->] (44) to (42);
	\draw[color=royalred, very thick, dashed, ->] (43) to (44);

	\foreach \i in {1,...,4} {
		\node[blob, label={below right:{\footnotesize $w_{5, \i}$}}] (5\i) at ({(4 * 6  - \i - 2)*360/30 - 288}:6) {};
	}
	\draw[color=royalred, very thick, dashed, ->]  (52) to (51);
	\draw[bend left, color=royalred, very thick, dashed, ->] (51) to (53);
	\draw[very thin]  (52) to (53);
	\draw[very thin, bend right=40] (54) to (52);
	\draw[color=royalred, very thick, dashed, ->]  (53) to (54);

	\foreach \i in {1,...,4} {
		\node[blob, label={[label distance=0.1cm]right:{\footnotesize $w_{3, \i}$}}] (3\i) at ({5 * 6  - \i - 2)*360/30 - 288}:6) {};
	}
	\draw[color=royalred, very thick, dashed, ->]  (32) to (31);
	\draw[bend left, color=royalred, very thick, dashed, ->] (31) to (33);
	\draw[very thin]  (32) to (33);
	\draw[very thin, bend right=40] (34) to (32);
	\draw[color=royalred, very thick, dashed, ->]  (33) to (34);

	\node[blob, label={[label distance=-0.1cm]above right:{\footnotesize $w_{1, 4}$}}] (04) at ({0 - 295}:6) {};
	\node[blob, label={[label distance=-0.05cm]above:{\footnotesize $w_{1, 3}$}}] (03) at ({1 * 360/30 - 295}:6) {};
	\node[blob, label={[label distance=-0.05cm]above:{\footnotesize $w_{1, 2}$}}] (02) at ({2 * 360/30 - 281}:6) {};
	\node[blob, label={[label distance=-0.1cm]above left:{\footnotesize $w_{1, 1}$}}] (01) at ({3 * 360/30 - 281}:6) {}; 

	\draw[very thin]  (02) to (01);
	\draw[color=royalred, very thick, dashed, ->]  (03) to (04);

	\draw[very thin, bend left]  (01) to (22);
	\draw[very thin, bend left]  (02) to (22);
	\draw[very thin, bend right]  (21) to (02);

	\draw[very thin, bend right]  (41) to (23);
	\draw[very thin, bend left]  (23) to (42);
	\draw[very thin, bend left]  (24) to (42);

	\draw[bend right, color=royalred, very thick, dashed, ->]  (54) to (43);
	\draw[very thin, bend left]  (43) to (53);
	\draw[very thin, bend left]  (44) to (53);

	\draw[very thin, bend left]  (51) to (33);
	\draw[very thin, bend left]  (52) to (33);
	\draw[bend right, color=royalred, very thick, dashed, ->]  (34) to (52);

	\draw[bend right, color=royalred, very thick, dashed, ->]  (04) to (32);
	\draw[very thin, bend left]  (31) to (03);
	\draw[very thin, bend left]  (32) to (03);
\end{tikzpicture}
      \caption{Red arrows indicate the order of the vertices mapped by $g_2$.}
      \label{fig:embedding-g2}
    \end{subfigure}
    \caption{An embedding of a graph $B_5$ by combining two $2$-pseudo-paths.}
    \label{fig:2-conn-path}
  \end{figure}
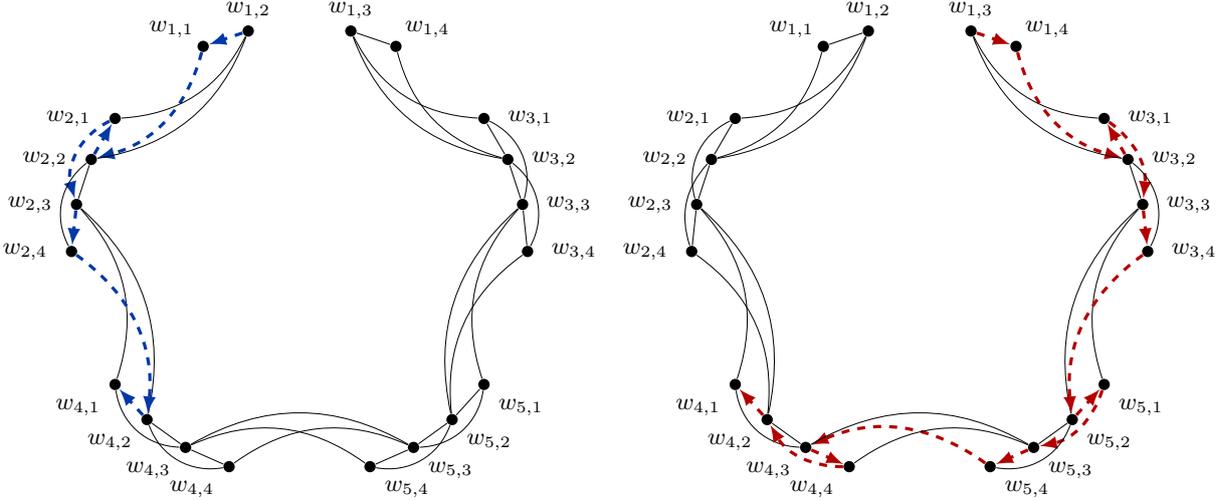
\end{proof}

The reason behind a rather complex looking definition of a backbone-path should
become more apparent once we make use of it as a building block for absorbers
later on (see Figure~\ref{fig:absorber}).

An {\bf $(b, \ell)$-connecting-path} $C^b_\ell$, for $b \in \{1, 2\}$ and $\ell$
divisible by four, is a graph on $\ell$ vertices defined as
\[
  C^{b}_{\ell} :=
    \begin{cases}
      P^{2}_{\ell}, & \text{if } b = 1, \\
      B_{\ell/4}, & \text{if } b = 2.
    \end{cases}
\]

\section{Proof of Theorem \ref{thm:square-res-codegree}}\label{sec:theproof}

Our proof strategy uses the {\em absorbing method}, in particular following a
variant used by Nenadov and the second author in~\cite{nenadov2019powers}. Let
$A$ be a graph and $\mathbf{a}, \mathbf{b} \in V(A)^2$ disjoint pairs of
vertices of $A$. Given a subset $X \subseteq V(A)$, we say that $A$ is an
\emph{$(\mathbf{a}, \mathbf{b}, X)$-absorber} if for every subset $X' \subseteq
X$ there exists a square-path $P \subseteq A$ connecting $\mathbf{a}$ to
$\mathbf{b}$ such that $V(P) = V(A) \setminus X'$.

The following lemma shows that one can find an absorber in a member of
$\calG(\Gamma, n, \alpha, p)$ for a large subset $X$ and $\Gamma \sim \Gnp$. The
proof of the Absorbing Lemma is deferred to Section~\ref{sec:absorbers}.

\begin{restatable}[\textbf{Absorbing Lemma}]{lemma}{absorbinglemma}
  \label{lemma:absorbing-lemma}
  For every $\alpha > 0$, there exists a positive constant $C := C(\alpha)$ such
  that w.h.p.\ for a random graph $\Gamma \sim \Gnp$ every $G \in \calG(\Gamma,
  n, \alpha, p)$ has the following property.

  Let $s \geq C\log^4 n/p^2$ be an integer. Then there are at least $(1 -
  n^{-3})\binom{n}{s}$ subsets $W \subseteq V(G)$ of size $s$ satisfying
  the following. For every subset $X \subseteq V(G) \setminus W$ of size $|X|
  \le |W|/(C\log^2 n)$ there exists an $(\mathbf{a}, \mathbf{b}, X)$-absorber
  $A$ in $G$ such that $V(A) \setminus X \subseteq W$, where $\mathbf{a},
  \mathbf{b} \in W^2$ are two disjoint pairs of vertices.
\end{restatable}

In order to construct absorbers one typically resorts to what is usually called
a {\em Connecting Lemma}. Intuitively, it allows us to connect certain pairs of
vertices by vertex-disjoint copies of a fixed graph $F$ through a reservoir of
vertices $W$. For ease of reference, we now define such a notion formally.

\begin{definition}
  Let $t, \ell \in \N$, let $G$ be a graph, and let $W \subseteq V(G)$ be a
  subset of vertices of $G$. Given a family $\calI = \{(\bx_i, \by_i)\}_{i \in
  [t]} \subseteq V(G)^{4}$ of pairwise disjoint $4$-tuples and $b \in \{1, 2\}$,
  we say that a collection $\{F_i \subseteq G\}_{i \in [t]}$ of subgraphs of $G$
  forms an $(\calI, b, \ell)$-{\em matching} in $W$ if the following holds:
  \begin{itemize}
    \item $F_i$ is a copy of a $(b, \ell)$-connecting-path connecting $\bx_i$ to
      $\by_i$, for every $i \in [t]$,
    \item $V(F_i) \setminus \{\bx_i, \by_i\} \subseteq W$, and
    \item $V(F_i) \cap V(F_j) = \varnothing$ for all distinct $i, j \in [t]$.
  \end{itemize}
\end{definition}

In other words, a $(b, \calI, \ell)$-matching `connects' prescribed tuples of
vertices from $\calI$ with copies of $(b, \ell)$-connecting-paths. The
Connecting Lemma shows that under certain conditions such a matchings exist.

\begin{restatable}[\bf Connecting Lemma]{lemma}{connectinglemma}
  \label{lemma:connecting-lemma}
  For every $b \in \{1, 2\}$ and every $\alpha > 0$, there exist positive
  constants $\eps := \eps(\alpha)$ and $C := C(\alpha)$ such that w.h.p.\ for a
  random graph $\Gamma \sim \Gnp$ every $G \in \calG(\Gamma, n, \alpha, p)$ has
  the following property.

  Let $s \geq C\log^4 n/p^2$ be an integer. Then there are at least $(1 -
  n^{-4})\binom{n}{s}$ subsets $W \subseteq V(G)$ of size $s$ satisfying the
  following. For every family of disjoint $4$-tuples $\{\bx_i, \by_i\}_{i \in
  [t]} \subseteq (V(G) \setminus W)^4$, such that $t\log n \leq \eps^6|W|$,
  there exists an $(\{\bx_i, \by_i\}_{i \in [t]}, b, 4\log n)$-matching in $W$.
\end{restatable}

In \cite{nenadov2019powers} the authors rely on Janson's inequality in order to
show such a statement. As we are working with a subgraph of a random graph, we
cannot apply this technique here. The proof of the Connecting Lemma thus becomes
a much more challenging task and requires a detailed analysis of `expansion of
the edges' in certain subsets. We defer it to
Section~\ref{sec:connecting-lemma}.

\subsection{Proof of the main result}

Let us first briefly give an overview of the various steps of the proof. The
first step is to partition the graph uniformly at random into sets $U$, $W$, and
$X$ such that $|U| = |W| = \Theta(n)$ and $|X| = \Theta(n/\log^2 n)$. Next, we
find an $(\mathbf{a}, \mathbf{b}, X)$-absorber $A$ for some $\mathbf{a},
\mathbf{b} \in W^2$, such that $V(A) \setminus X \subseteq W$. Let $W'$ denote
the vertices of $W$ which are not part of the absorber, and let $U' = U \cup
W'$. In the third step, we construct $t$ vertex-disjoint square-paths $P_1,
\dotsc, P_t$ in $U'$ such that
\[
  \Big| U' \setminus \bigcup_{i \in [t]} V(P_i) \Big| \ll |X|/\log n \qquad
  \text{and} \qquad t \ll |X|/\log n.
\]
Let us denote the set of vertices from $U'$ not contained in any $P_i$ by $Q$.
Using the Connecting Lemma with $X$ as the `reservoir' (set $W$ in
Lemma~\ref{lemma:connecting-lemma}) we connect $P_1, \dotsc, P_t$, vertices from
$Q$, and pairs $\mathbf{a}$ and $\mathbf{b}$ into a square-path $P$ such that
$P$ connects $\mathbf{b}$ to $\mathbf{a}$ and $V(G) \setminus X \subseteq V(P)$.
Let $X'$ be the set of vertices from $X$ contained in $P$. By the definition of
the absorber $A$ there exists a square-path $P'$ connecting $\mathbf{a}$ to
$\mathbf{b}$ such that $V(P') = V(A) \setminus X'$. By combining $P$ and $P'$ we
obtain the square of a Hamilton cycle. In the remainder of the section we
formalise this argument.

Let $C_1 = C_{\ref{lemma:absorbing-lemma}}(\alpha)$, $C_2 =
C_{\ref{lemma:connecting-lemma}}(\alpha)$, $\eps_1 =
\eps_{\ref{prop:random-set-is-good}}(\alpha)$, $\eps_2 =
\eps_{\ref{lemma:connecting-lemma}}(\alpha)$, $\eps = \min\{\eps_1, \eps_2,
1/(4C_1)\}$, and $C = \max\{ \eps^{-2}C_1C_2, 100 \}$. Let $G$ be a member of
$\calG(\Gamma, n, \alpha, p)$ and let $V(G) = X_1 \cup X_2 \cup W \cup U$ be a
partition of $V(G)$ chosen uniformly at random, where $X := X_1 \cup X_2$, such
that
\[
  |X| = \floor*{\frac{\eps n}{2C_1\log^2 n}}, \quad |X_1| =
  \floor{\eps|X|}, \quad |W| = \floor{\eps n}, \quad \text{and} \quad |U| = n -
  |X| - |W|.
\]
Note that w.h.p.\ $W$ satisfies the conclusion of the Absorbing Lemma
(Lemma~\ref{lemma:absorbing-lemma}), $X_2$ the conclusion of the Connecting
Lemma (Lemma~\ref{lemma:connecting-lemma}), and all four sets are $(\alpha/2,
\eps, p)$-good by Proposition~\ref{prop:random-set-is-good}. From now on we fix
such a choice of subsets.

We apply Lemma~\ref{lemma:absorbing-lemma} with $W$ and $X$ to obtain an
$(\mathbf{a}_0, \mathbf{b}_0, X)$-absorber $A$ for some $\mathbf{a}_0,
\mathbf{b}_0 \in W^2$, such that $V(A) \setminus X \subseteq W$. We can indeed
do this, as $|W| \ge \eps n - 1 \ge C_1\log^4 n/p^2$ and $|X| \le |W|/(C_1\log^2
n)$. Let us denote by $W'$ the subset of vertices from $W$, which are not
contained in $A$. Furthermore, let $U' := U \cup W'$. Next, we use the following
claim whose proof is given at the end of the section.

\begin{restatable}[\textbf{Covering Claim}]{claim}{coveringlemma}
  \label{claim:covering}
  For every $\eps > 0$, there exists a positive constant $K := K(\eps)$ such that
  w.h.p.\ the following holds. The induced subgraph $G[U']$ contains $t \le
  K\log\log n$ vertex-disjoint square-paths $P_1, \dotsc, P_t$ such that $\big|
  U' \setminus \bigcup_{i \in [t]} V(P_i) \big| \leq \eps|U'|/\log^3 n$.
\end{restatable}

By applying the Covering Claim with $\eps^{10}$ (as $\eps$) we get that there is
a constant $K = K_{\ref{claim:covering}}(\eps^{10})$ and that w.h.p.\ $G[U']$
contains $t \leq K\log\log n$ vertex-disjoint square-paths $P_1, \dotsc, P_t$
which contain all but at most
\begin{equation}\label{eq:main-proof-eq1}
  \frac{\eps^{10} |U'|}{\log^3 n} \le \frac{\eps^{10} n}{\log^3 n} \le
  \frac{\eps^8 |X|}{\log n}
\end{equation}
vertices from $U'$. We denote the set of uncovered vertices by $Q$ and the
end-pairs of $P_i$ by $\mathbf{a}_i$ and $\mathbf{b}_i$, for every $i \in [t]$.

We next show that there exists a matching between $Q$ and $X_1$ which saturates
$Q$ by verifying Hall's condition. Recall, $X_1$ is $(\alpha/2, \eps, p)$-good.
Let $Q_1 \subseteq Q$ be an arbitrary subset of $Q$ and let us denote $N_G(Q_1,
X_1)$ by $Z$. If $|Q_1| \leq \eps^{-3}\log n/p$, then for a vertex $v \in Q_1$
\begin{equation}\label{eq:main-proof-eq2}
  |Z| \geq \deg_{G}(v, X_1) \geq (2/3 + \alpha/2)|X_1|p \geq \frac{\eps^2
  np}{4C_1 \log^2 n} \geq \frac{\eps^{-3}\log n}{p} \geq |Q_1|,
\end{equation}
where the second to last inequality follows from the bound on $p$. On the other
hand, if $|Q_1| \geq \eps^{-3}\log n/p$ then by \ref{p:unif-density} we have
\[
  |Q_1| (2/3 + \alpha/2)|X_1|p \leq e_G(Q_1, Z) \leq (1+\eps)|Q_1|\max\{|Z|,
  \eps^{-3}\log n/p\}p,
\]
which implies $|Z| \geq |X_1|/4$ (with room to spare) as otherwise we get $|X_1|
\leq 4\eps^{-3}\log n/p$, which is not true, again by the bound on $p$. Since
$|Z| \geq |X_1|/4$ we have by \eqref{eq:main-proof-eq2} that $|Z| \geq |Q_1|$
and thus by Hall's theorem (Lemma~\ref{lem:star-hall}) there exists a matching
between $Q$ and $X_1$ which saturates $Q$. Let us denote the edges of the
matching by $\mathbf{m}_i$ for all $i \in \{1, \dotsc, |Q|\}$.

As the last step, we apply the Connecting Lemma
(Lemma~\ref{lemma:connecting-lemma}) with $X_2$ (as $W$) and with the following
family of $4$-tuples (as $\calI$):
\[
  \{ (\mathbf{b}_i, \mathbf{a}_{i + 1}) \}_{i \in \{0, \dotsc, t - 1\}} \cup
  (\mathbf{b}_t, \mathbf{m}_1) \cup \{ (\mathbf{m}_i, \mathbf{m}_{i + 1})
  \}_{i\in[|Q| -1]} \cup (\mathbf{m}_{|Q|}, \mathbf{a}_0),
\]
in order to obtain a square path $P$ which connects $\mathbf{b}_0$ to
$\mathbf{a}_0$ and contains all vertices from $V(G) \setminus V(A)$ and possibly
some vertices from $X_2$. Note that we can indeed apply the lemma as
\[
  |X_2| \geq (1 - \eps) \frac{\eps n}{2C_1 \log^2 n} - 2 \geq \frac{C_2 \log^4
  n}{p^2}
\]
and by \eqref{eq:main-proof-eq1} we have
\[
  t + |Q| + 1 \le K\log\log n + \frac{\eps^8 |X|}{\log n} + 1 \le \frac{\eps^6
  |X_2|}{\log n}.
\]
Finally, let $X' \subseteq X$ be the set of vertices from $X$ contained in $P$.
By the definition of the $(\mathbf{a}_0, \mathbf{b}_0, X)$-absorber $A$ there
exists a path $P'$ connecting $\mathbf{a}_0$ to $\mathbf{b}_0$ such that $V(P')
= V(A) \setminus X'$. By combining $P$ and $P'$ we obtain the square of a
Hamilton cycle in $G$, which concludes the proof of
Theorem~\ref{thm:square-res-codegree}.

In the next subsection we provide the missing proof of the Covering Claim.

\subsection{Proof of the Covering Claim}

The goal of this subsection is to show that $G[U']$ contains $o(|X|/\log n)$
vertex-disjoint square-paths which contain all but at most $o(|X|/\log n)$
vertices from $U'$. In an earlier paper by some of the
authors~\cite{vskoric2018local} we proved that w.h.p.\ any subgraph of $\Gnp$
contains the square of a Hamilton cycle on $(1 - o(1))n$ vertices, provided that
$p \gg (\log n/n)^{1/2}$. The proof of Claim~\ref{claim:covering} relies on this
result which we thus state precisely.

\begin{theorem}[Škorić, Steger, Trujić~\cite{vskoric2018local}]
  \label{thm:almost-spanning-square}
  For every $\eps, \alpha > 0$ there exist positive constants $C := C(\eps,
  \alpha)$ and $b := b(\eps, \alpha)$, such that if $p \geq C (\log n/n)^{1/2}$
  then the random graph $\Gamma \sim \Gnp$ has the following property with
  probability at least $1 - e^{-b n^2p/\log^2 n}$. Every spanning subgraph of
  $\Gamma$ with minimum degree at least $(2/3 + \alpha)np$, contains the square
  of a cycle on at least $(1 - \eps)n$ vertices.
\end{theorem}

As a corollary we get the following statement.

\begin{corollary}\label{cor:covering-ub}
  For every $\eps, \alpha > 0$ there exists a positive constant $C := C(\eps,
  \alpha)$, such that if $p \geq C(\log^4 n/n)^{1/2}$ then the random graph
  $\Gamma \sim \Gnp$ w.h.p.\ has the following property. Every subgraph $G
  \subseteq \Gamma$ of size $v(G) \ge \eps n/\log^3 n$ with minimum degree at
  least $(2/3 + \alpha) v(G)p$, contains the square of a cycle on at least $(1 -
  \eps)v(G)$ vertices.
\end{corollary}
\begin{proof}
  Let $C' = C_{\ref{thm:almost-spanning-square}}(\eps, \alpha)$, $b =
  b_{\ref{thm:almost-spanning-square}}(\eps, \alpha)$, and let $C =
  \max\{1000/(\eps b), C'/\sqrt{\eps}\}$. Note that for all subgraphs $G
  \subseteq \Gamma$ of size $s$, such that $s \ge \eps n / \log^3 n$, we have
  \[
    p \geq C \Big( \frac{\log^4 n}{n} \Big)^{1/2} \geq C \Big( \frac{\eps\log^4
    n}{s \log^3 n} \Big)^{1/2} \geq C'\Big( \frac{\log s}{s} \Big)^{1/2}.
  \]
  Thus, for a fixed subgraph $G$ of size $s$ with the required minimum degree we
  have by Theorem~\ref{thm:almost-spanning-square} that with probability at
  least $1 - e^{-b s^2 p / \log^2 s}$, $G$ contains the square of a cycle on at
  least $(1 - \eps) s$ vertices. Since
  \[
    \frac{b s^2 p}{\log^2 s} \ge s \cdot \frac{b\eps np}{\log^5 n} \ge s \cdot
    \frac{b \eps C \cdot \sqrt{n}}{\log^2 n} \ge 1000s\log n,
  \]
  we may additionally do the union bound over all $s \ge \eps n/\log^3 n$ and
  all subsets of size $s$.
\end{proof}

With Corollary~\ref{cor:covering-ub} at hand we are ready to give the proof of
Claim~\ref{claim:covering} by using a bootstrapping technique developed by
Nenadov and the second author~\cite{nenadov2020komlos}.
\begin{proof}[Proof of Claim~\ref{claim:covering}]
  Without loss of generality we assume that $\eps$ is sufficiently small w.r.t.\
  $\alpha$. Recall that $U$ is $(\alpha/2, \eps, p)$-good and $|U| \ge (1 -
  2\eps)|U'|$, so for each $v \in U'$ we have by \ref{p:good-deg}
  \[
    \deg_G(v, U') \geq \deg_G(v, U) \geq (2/3 + \alpha/2)|U|p \geq (1 -
    2\eps)(2/3 + \alpha/2)|U'|p \geq (2/3 + \alpha/4)|U'|p.
  \]
  Let $q$ be the largest integer such that $|U'|/2^{q - 1} \geq \ceil{n/\log^3
  n}$ and note that $q = O(\log\log n)$. Consider a uniformly at random chosen
  partition $U' = V_1 \cup \dotsb \cup V_q$ such that $V_i = \floor{|U'|/2^i}$
  for all $i \in [q - 1]$ and
  \[
    |V_q| = |U'| - (|V_1| + \dotsb + |V_{q - 1}|) \ge |U'| - |U'| \sum_{i =
    1}^{q - 1} 2^{-i} = |U'|/2^{q - 1}.
  \]
  Similarly, $|V_q| \le |U'| / 2^{q - 1} + q$. Since $p \gg n^{-1/2}\log^{3} n$
  we have
  \begin{equation}\label{eq:covering-eq1}
    |V_i| \ge |U'| / 2^{q-1} - 1 \ge \frac{n}{2\log^3 n} \ge \frac{ \log^2 n
    }{p}.
  \end{equation}
  Thus, as $V_i$ is a random subset of $U'$, by a simple application of
  Chernoff's inequality for a hypergeometric distribution we get
  \[
    \Pr[\deg_G(v, V_i) < (2/3+\alpha/8)|V_i|p] = e^{-\Omega(|V_i|p)} < 1/n^2,
  \]
  where the last inequality follows from \eqref{eq:covering-eq1}. Using the
  union bound over all $v \in U'$ and $i \in [q]$ we have that w.h.p.\ for each
  $v \in U'$ and each $i \in [q]$
  \[
    \deg_G(v, V_i) \ge (2/3 + \alpha/8)|V_i|p.
  \]
  Since $p \gg n^{-1/2}\log^{3}n$ we have that w.h.p.\
  Corollary~\ref{cor:covering-ub} holds when applied with $\eps/8$ (as $\eps$)
  and $\alpha/16$ (as $\alpha$). Having this in mind, we prove by induction on
  $i \in [q]$ that $G[V_1 \cup \dotsb \cup V_i]$ contains $i$ square-paths which
  cover all but at most $(\eps/4)|V_i|$ vertices from $V_1 \cup \dotsb \cup
  V_i$. Since $q = O(\log\log n)$ and
  \[
    |V_q| \le |U'| / 2^{q - 1} + q \le \frac{2n}{\log^3 n} + q \le \frac{4
    n}{\log^3 n},
  \]
  we have that by setting $i = q$ the induction implies the claim.

  By Corollary~\ref{cor:covering-ub} we directly get that there exists a
  square-path in $G[V_1]$ which covers all but at most $(\eps/8)|V_1|$ vertices
  from $V_1$, settling the base case. Assume now that the hypothesis holds for
  some $1 \leq i < q$ and let $P_1, \dotsc, P_{i}$ be the square-paths
  guaranteed by the hypothesis. Furthermore, let $Q \subseteq V_1 \cup \dotsb
  \cup V_i$ denote the subset of vertices not contained in any of the $P_i$'s.
  Then $|Q| \le (\eps/8)|V_i| \le \eps|V_{i + 1}|$ and for every $v \in U'$ we
  have
  \begin{equation}\label{eq:covering-eq2}
    \begin{aligned}
      \deg_{G}(v, Q \cup V_{i + 1}) &\ge \deg_{G}(v, V_{i + 1}) \ge (2/3 +
      \alpha/8)|V_{i + 1}|p \\ &\ge (2/3 + \alpha/8) \frac{|Q| + |V_{i + 1}|}{1
      + \eps}p \ge (2/3 + \alpha/16)|Q \cup V_{i + 1}|p,
    \end{aligned}
  \end{equation}
  where the last inequality follows from the assumption on $\eps$. Therefore, by
  Corollary~\ref{cor:covering-ub} we know that $G[Q \cup V_{i + 1}]$ contains a
  square-path $P_{i + 1}$ which covers all but at most $(\eps/8)|Q \cup V_{i +
  1}| \le (\eps/4)|V_{i + 1}|$ vertices. Observe that we can indeed use
  Corollary~\ref{cor:covering-ub} since $|V_{i + 1}| \ge n/(2 \log^3 n)$ and by
  \eqref{eq:covering-eq2} we have that $\delta(G[Q \cup V_{i + 1}]) \ge (2/3 +
  \alpha/16) |Q\cup V_{i + 1}|p$. As the vertices from $(V_1 \cup \dotsb \cup
  V_{i}) \setminus Q$ are already contained in $V(P_1) \cup \dotsb \cup
  V(P_{i})$, this shows that the hypothesis holds for $i + 1$.
\end{proof}

\section{Proof of the Absorbing Lemma}\label{sec:absorbers}

Our strategy for constructing an absorber for a set $X$ consists of two steps.
In the first step we find an $(\mathbf{a}_i, \mathbf{b}_i, \{x_i\})$-absorber
$A_i$ (a \emph{single-vertex} absorber) for each $x_i \in X$, such that they are
pairwise disjoint. In the second step, by using the Connecting Lemma, we find a
square-path from $\mathbf{b}_i$ to $\mathbf{a}_{i+1}$, for every $1 \le i \le m
- 1$, such that they are pairwise disjoint and also disjoint from $A_i$'s. It is
easy to see that this gives an $(\mathbf{a}_1, \mathbf{b}_m, X)$-absorber: given
$X' \subseteq X$, for every $x_i \in X$ we choose a square-path in $A_i$
depending on whether $x_i \in X'$ or $x_i \notin X'$.

We use the following construction for a single-vertex absorber.

\begin{claim}\label{cl:single-vertex-absorber}
  Let $A_x$ be a graph obtained as the union of following graphs and edges
  \[
    B_\ell \cup \{w_{1, 2}, w_{1, 3}\} \cup \bigcup_{i \in [4]} \{w_{1, i}, x\}
    \cup \bigcup_{i \in [\ell]} U_i,
  \]
  where $U_i$ is a square-path connecting $\mathbf{w}_i^b$ to $\mathbf{w}_{i +
  1}^a$ for every $1 \le i < \ell$, such that all the square-paths are pairwise
  vertex-disjoint and also disjoint from $B_\ell$ (except for the pairs of
  vertices they connect). Furthermore, vertex $x$ is not contained in either
  $B_\ell$ or any $U_i$. Then $A_x$ is a $(\bw_1^a, \bw_\ell^b,
  \{x\})$-absorber.
\end{claim}
\begin{proof}
  There are only two cases we need to consider: $X' = \varnothing$ and $X' =
  \{x\}$. We specify the desired square-path from $\bw_1^a$ to $\bw_\ell^b$ in
  each case by giving the ordering in which we traverse the vertices of such a
  path (see Figure~\ref{fig:absorber}):
  \begin{itemize}
    \item $X' = \varnothing$: $\bw_1^a, x, \bw_1^b, U_1, \bw_2^a, \bw_2^b, U_2,
      \bw_3^a, \bw_3^b, U_3, \bw_4^a, \dotsc, \bw_{\ell - 1}^b, U_{\ell-1},
      \bw_{\ell}^a, \bw_{\ell}^b$,
    \item $X' = \{x\}$: $\bw_1^a, \obw_2^a, \overline{U}_1, \obw_1^b, \obw_3^a,
      \overline{U}_2, \obw_2^b, \obw_4^a, \overline{U}_3, \obw_3^b, \dotsc
      \obw_{\ell}^a, \overline{U}_{\ell - 1}, \obw_{\ell - 1}^b, \bw_{\ell}^b$.
  \end{itemize}

  \begin{figure}[!htbp]
    \centering
    \begin{subfigure}[b]{.49\textwidth}
      \centering
      \begin{tikzpicture}[scale=0.5]
	\tikzstyle{blob} = [fill=black, circle, inner sep=1.5pt, minimum size=0.5pt]
	
	\foreach \a in {2, 4} {
		\foreach \i in {1,...,4} {
			\node[blob] (\a\i) at ({((\a * 3 + \i - 1)*360/30 - 288}:6) {};					
		}

		\draw[color=royalazure, very thick] (\a1) to (\a2);
		\draw[bend right=40] (\a1) to (\a3);
		\draw[color=royalazure, very thick]  (\a2) to (\a3);
		\draw[bend right=40] (\a2) to (\a4);
		\draw[color=royalazure, very thick]  (\a3) to (\a4);
	}

	\foreach \i in {1,...,4} {	
		\node[blob] (5\i) at ({(4 * 6  - \i - 2)*360/30 - 288}:6) {};							
	}
	\draw[color=royalazure, very thick]  (51) to (52);
	\draw[bend left=40] (51) to (53);			
	\draw[color=royalazure, very thick]  (52) to (53);
	\draw[bend left=40] (52) to (54);
	\draw[color=royalazure, very thick]  (53) to (54);		

	\foreach \i in {1,...,4} {	
		\node[blob] (3\i) at ({5 * 6  - \i - 2)*360/30 - 288}:6) {};							
	}		
	\draw[color=royalazure, very thick]  (31) to (32);
	\draw[bend left=40] (31) to (33);			
	\draw[color=royalazure, very thick]  (32) to (33);
	\draw[bend left=40] (32) to (34);
	\draw[color=royalazure, very thick]  (33) to (34);		

	\node[blob] (04) at ({0 - 295}:6) {};
	\node[blob] (03) at ({1 * 360/30 - 295}:6) {};
	\node[blob] (02) at ({2 * 360/30 - 281}:6) {};
	\node[blob] (01) at ({3 * 360/30 - 281}:6) {};

	\node[blob] (x) at ({360/20 - 288}:8) {}; \node at ({360/20 - 288}:8.5) {$x$};

	\draw[color=royalazure, very thick]  (01) to (02); \draw[color=royalazure, very thick]  (02) to (x); 
	\draw[color=royalazure, very thick]  (x) to (03); 
	\draw[color=royalazure, very thick]  (03) to (04); 
	\draw (02) to (03); \draw[bend left=50] (01) to (x); \draw[bend left=50] (x) to (04);

	\draw[color=royalazure, very thick] (04) to (21); \node at (-0.5, 4) {$U_1$};
	\draw[color=royalazure, very thick] (24) to (31); \node at (-0.5, 1) {$U_2$};
	\draw[color=royalazure, very thick] (34) to (41); \node at (1, -1) {$U_3$};
	\draw[color=royalazure, very thick] (44) to (51); \node at (1, -4) {$U_4$};

	\node at (-1.9, 6.2) {$\bw_1^a$}; \node at (1.9, 6.2) {$\bw_1^b$};
	\node at (-6.2, 3.2) {$\bw_2^a$}; \node at (-6.9, 0.8) {$\bw_2^b$};
	\node at (6.2, 3.2) {$\bw_3^a$}; \node at (6.9, 0.8) {$\bw_3^b$};
	\node at (-5.3, -4.5) {$\bw_4^a$}; \node at (-2.2, -5) {$\bw_4^b$};
	\node at (5.3, -4.5) {$\bw_5^a$}; \node at (2.2, -4.9) {$\bw_5^b$};
\end{tikzpicture}
      \caption{The square-path from $\bw_1^a$ to $\bw_{5}^b$ including $x$.}
    \end{subfigure}
    \hfill
    \begin{subfigure}[b]{.49\textwidth}
      \centering
      \begin{tikzpicture}[scale=0.5]
	\tikzstyle{blob} = [fill=black, circle, inner sep=1.5pt, minimum size=0.5pt]

	\foreach \a in {2, 4} {
		\foreach \i in {1,...,4} {				
			\node[blob] (\a\i) at ({((\a * 3 + \i - 1)*360/30 - 288}:6) {};
		}

		\draw[color=royalred, very thick] (\a1) to (\a2);
		\draw[color=royalred, very thick] (\a3) to (\a4);
	}

	\foreach \i in {1,...,4} {
		\node[blob] (5\i) at ({(4 * 6  - \i - 2)*360/30 - 288}:6) {};
	}
	\draw[color=royalred, very thick] (51) to (52);
	\draw[color=royalred, very thick] (53) to (54);

	\foreach \i in {1,...,4} {
		\node[blob] (3\i) at ({5 * 6  - \i - 2)*360/30 - 288}:6) {};
	}
	\draw[color=royalred, very thick] (31) to (32);
	\draw[color=royalred, very thick] (33) to (34);

	\node[blob] (04) at ({0 - 295}:6) {};
	\node[blob] (03) at ({1 * 360/30 - 295}:6) {};
	\node[blob] (02) at ({2 * 360/30 - 281}:6) {};
	\node[blob] (01) at ({3 * 360/30 - 281}:6) {};

	\draw[color=royalred, very thick] (01) to (02); \draw[color=royalred, very thick] (03) to (04);

	\draw[color=royalred, very thick, bend left] (04) to (21); \node at (-0.5, 4) {$U_1$};
	\draw[color=royalred, very thick] (24) to (31); \node at (-0.5, 1) {$U_2$};
	\draw[color=royalred, very thick] (34) to (41); \node at (1, -1) {$U_3$};
	\draw[color=royalred, very thick, bend left] (44) to (51); \node at (1.2, -3) {$U_4$};

	\draw[bend left, color=royalred, very thick] (02) to (22); \draw[bend left] (01) to (22); \draw[bend left] (02) to (21);
	\draw[bend right] (04) to (32); \draw[bend right, color=royalred, very thick] (03) to (32); \draw[bend right] (03) to (31);
	\draw[bend left] (24) to (42); \draw[bend left, color=royalred, very thick] (23) to (42); \draw[bend left] (23) to (41);
	\draw[bend right] (34) to (52); \draw[bend right, color=royalred, very thick] (33) to (52); \draw[bend right] (33) to (51);
	\draw[bend left] (44) to (53); \draw[bend left, color=royalred, very thick] (43) to (53); \draw[bend left] (43) to (54);

	\node at (-2, 6.2) {$\bw_1^a$}; \node at (2, 6.2) {$\bw_1^b$};
	\node at (-5.8, 3.2) {$\bw_2^a$}; \node at (-6.5, 0.8) {$\bw_2^b$};
	\node at (5.7, 3.2) {$\bw_3^a$}; \node at (6.6, 0.8) {$\bw_3^b$};
	\node at (-5, -4.5) {$\bw_4^a$}; \node at (-2.7, -6) {$\bw_4^b$};
	\node at (5, -4.3) {$\bw_5^a$}; \node at (2.7, -6) {$\bw_5^b$};
\end{tikzpicture}
      \caption{The square-path from $\bw_1^a$ to $\bw_{5}^b$ without $x$.}
    \end{subfigure}

    \caption{A single-vertex absorber $A_x$. Note that both square-paths use all
    vertices of $A_x$.}
    \label{fig:absorber}
  \end{figure}
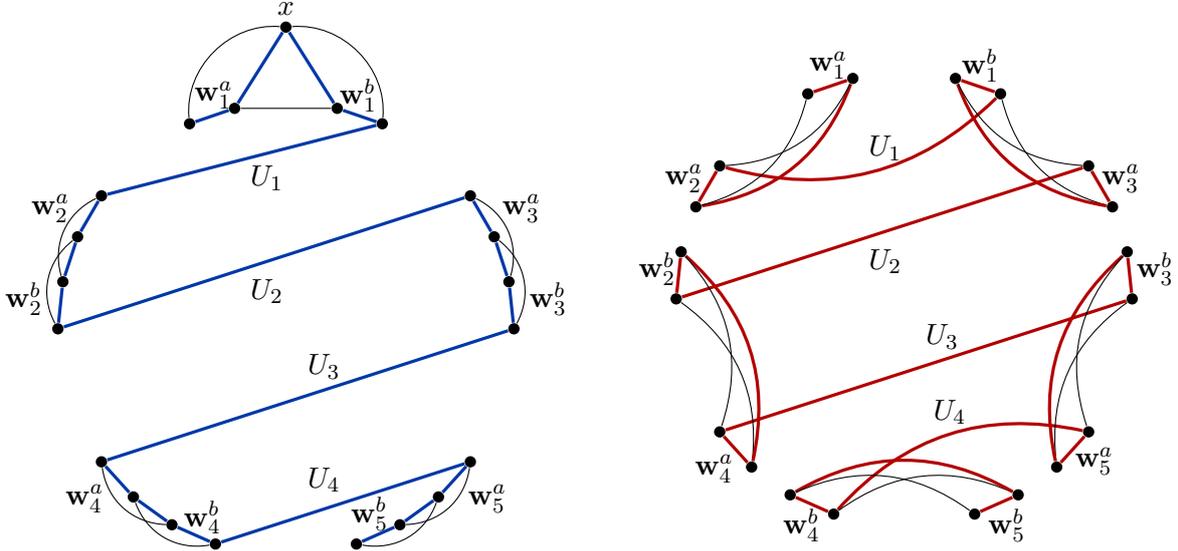
\end{proof}

Now, we are ready to present the proof of the Absorbing Lemma. For the
convenience of the reader we first restate the lemma.

\absorbinglemma*
\begin{proof}[Proof of Lemma \ref{lemma:absorbing-lemma}]
  Let $\eps' = \eps_{\ref{lemma:connecting-lemma}}(\alpha)$, $\eps'' =
  \eps_{\ref{prop:random-set-is-good}}(\alpha)$, $\eps = \min\{\eps', \eps'',
  1/100\}$, $C' = C_{\ref{lemma:connecting-lemma}}(\eps', \alpha)$, and $C =
  40\max\{C', \eps^{-6}\}$. Let $W = W_1 \cup \dotsb \cup W_7$ be such that
  $W_i$'s are pairwise vertex-disjoint, all of size $|W_i| \geq \floor{s/10}$,
  and chosen uniformly at random from $G$ among all sets of prescribed size. By
  the union bound, with probability at least $1 - o(n^{-3})$ it holds that for
  all $i \in [7]$, $W_i$ is $(\alpha/2, \eps, p)$-good by
  Proposition~\ref{prop:random-set-is-good}, and satisfies the conclusion of the
  Connecting Lemma (Lemma~\ref{lemma:connecting-lemma}). From now on fix such a
  choice of sets $W_i$ (and thus $W$) and note that it is sufficient to show the
  conclusion of the lemma for this particular $W$.

  Let $X = \{x_1, \dotsc, x_m\} \subseteq V(G) \setminus W$ be a subset of
  vertices such that $m \le s / (C \log^2 n)$ and observe that $4|X|\log^2 n \le
  \eps^6|W_i|$. Furthermore, let $S_{x_i}$ be the subgraph of $A_{x_i}$ (as
  defined in Claim~\ref{cl:single-vertex-absorber}) induced on the vertex set
  $\{w_{1, 1}, w_{1, 2}, w_{1, 3}, w_{1, 4}, x_i\}$. We aim to construct a
  vertex-disjoint collection $\{S_{x_i}\}_{i \in [m]}$ in $G$, such that each
  $S_{x_i}$ contains $x_i$ and no other vertex from $X$. We do this in four
  steps.

  {\bf Step 1.} First, we show that there exists a matching $M_1$ between $X$
  and $W_1$ saturating $X$. Let $X' \subseteq X$ be a subset of $X$ and let us
  denote $N_G(X', W_1)$ by $Z$. If $|X'| \le \eps^{-3}\log n/p$ then for a
  vertex $x_i \in X'$
  \[
    |Z| \ge |N_G(x_i, W_1)| \ge (2/3 + \alpha/2)|W_1|p \ge (2C/30) \log^4 n/p
    \ge |X'|,
  \]
  where the second inequality follows from \ref{p:good-deg}. If we assume $|X'|
  > \eps^{-3}\log n/p$ (and hence $|Z| \geq \eps^{-3}\log n/p$ by analysis from
  above) then by \ref{p:unif-density} and \ref{p:good-deg} we have
  \[
    |X'| (2/3 + \alpha/2)|W_1| p \le e_G(X', Z) \le (1 + \eps) |X'||Z|p,
  \]
  which implies $|Z| \ge |W_1|/2 \ge |X'|$. Thus, Hall's condition is satisfied
  and the desired matching $M_1$ exists. Let us denote the matched vertex of
  some $x_i \in X$ in $M_1$ by $M_1(x_i)$.

  {\bf Step 2.} In the next step, we want to find a family of $m$
  vertex-disjoint triangles $\calT_1$, such that each triangle contains exactly
  one edge from $M_1$ and intersects $W_2$ in exactly one vertex. We achieve
  this again with the help of Hall's matching theorem
  (Lemma~\ref{lem:star-hall}). Let $X' \subseteq X$ and let us denote
  \[
    \bigcup_{x_i \in X'} \big( N_G(x_i, W_2) \cap N_G(M_1(x_i), W_2) \big)
  \]
  by $Z$. By applying \ref{p:good-edge-expansion} to the edges of $M_1$ incident
  to the vertices of $X'$ or its arbitrary subset of size $\eps/p^2$ in case
  $|X'| > \eps/p^2$ (as $\calP$) we have
  \[
    |Z| \ge \alpha \cdot \min\{\eps/p^2, |X'|\} |W_2|p^2 \ge \alpha \cdot \min\{
    \eps|W_2|, |X'| \cdot (C/10) \log^4 n \} \ge |X'|,
  \]
  and by Lemma~\ref{lem:star-hall} we conclude that the desired collection
  $\calT_1$ exists. For $x_i \in X$ let us denote by $\mathbf{u}_i = (u_{i, 1},
  u_{i, 2})$ the two vertices sharing the triangle with $x_i$ from $\calT_1$,
  where $u_{i, 1} \in W_1$ and $u_{i, 2} \in W_2$.

  {\bf Step 3.} In a manner analogous to that seen in the second step, we find a
  collection $\calT_2$ of $m$ vertex-disjoint triangles such that each triangle
  in $\calT_2$ contains exactly one vertex from $W_3$ and vertices $x_i$ and
  $u_{i, 2}$, for every $i \in [m]$. Let us denote by $v_{i, 1}$ the third
  vertex in the triangle from $\calT_2$ which contains $x_i$ and $u_{i, 2}$.

  {\bf Step 4.} In the last step, we find a collection $\calT_3$ of $m$ vertex
  disjoint triangles such that each triangle in $\calT_3$ contains exactly one
  vertex from $W_4$ and vertices $x_i$ and $v_{i, 1}$, for every $i \in [m]$.
  Let us denote by $v_{i, 2}$ the third vertex in the triangle from $\calT_3$
  which contains $x_i$ and $v_{i, 1}$ and $\mathbf{v}_i = (v_{i, 1}, v_{i, 2})$.

  This completes the first part of the embedding scheme as we have constructed
  the vertex-disjoint collection $\{S_{x_i}\}_{i \in [m]}$ containing vertices
  from $X$.

  The rest of the proof consists of three consecutive applications of the
  Connecting Lemma (Lemma~\ref{lemma:connecting-lemma}). First, by applying it
  with $b = 2$, $W_5$ (as $W$), and $\{(\overline{\mathbf{u}}_i,
  \overline{\mathbf{v}}_i)\}_{i \in [m]}$ (as $\{(\bx_i, \by_i)\}_i$) we
  conclude that there exists a
  \[
    \big( \{(\overline{\mathbf{u}}_i, \overline{\mathbf{v}}_i)\}_{i \in [m]}, 2,
    4\log n \big)\text{-matching}
  \]
  in $W_5$. We can apply the Connecting Lemma as $|W_5| \ge (C/10) \log^4 n/p^2$
  and $m \log n \le \eps^6 |W_5|$. Let $\ell = \log n$ and let $g_i$ be the
  embedding of the backbone-path $B_{\ell}$ given by the above matching, for
  each $i \in [m]$. Next, consider the family of $4$-tuples $\{ (g_i(\bw_j^b),
  g_i(\bw_{j + 1}^a)) \}_{i \in [m], j \in [\ell - 1]}$. We apply the Connecting
  Lemma with $b = 1$, $W_6$ (as $W$), and $\{ (g_i(\bw_j^b), g_i(\bw_{j + 1}^a)
  )\}_{i \in [m], j \in [\ell - 1]}$ (as $\{(\bx_i, \by_i)\}_i$) to conclude
  that there exists a
  \[
    \big( \{(g_i(\bw_j^b), g_i(\bw_{j + 1}^a))\}_{i \in [m], j \in [\ell - 1]},
    1, 4\log n \big)\text{-matching}
  \]
  in $W_6$. We can do that as $|W_6| \ge (C/10) \log^4 n/p^2$ and $4 m\log^2 n
  \le \eps^6 |W_6|$. Let us denote pairs $g_i(\bw^a_1)$ and $g_i(\bw^b_{\ell}) $
  by $\mathbf{a}_i$ and $\mathbf{b}_i$, for every $i \in [m]$. By
  Claim~\ref{cl:single-vertex-absorber} we conclude that the set $X \cup W_1
  \cup \dotsb \cup W_6$ contains an $(\mathbf{a}_i, \mathbf{b}_i,
  \{x_i\})$-absorber $A_i$ for each $i \in [m]$, such that all $A_i$'s are
  pairwise vertex-disjoint.

  Lastly, using the vertices in $W_7$ we connect all $A_i$'s into a single
  absorber for the set $X$. Consider the family of $4$-tuples $\{(\mathbf{b}_i,
  \mathbf{a}_{i + 1})\}_{i \in [m - 1]}$. By applying the Connecting Lemma with
  $b = 1$, $W_7$ (as $W$), and $\{(\mathbf{b}_i, \mathbf{a}_{i + 1})\}_{i \in [m
  - 1]}$ (as $\{(\bx_i, \by_i)\}_i$), there exists an
  \[
    \big( \{(\mathbf{b}_i, \mathbf{a}_{i + 1})\}_{i \in [m - 1]}, 1, 4\log n
    \big)\text{-matching}
  \]
  in $W_7$. We can do that as $|W_7| \ge (C/10) \log^4 n/p^2$ and $m \log n \le
  \eps^6 |W_7|$. If we denote the square-paths connecting $\mathbf{b}_i$ to
  $\mathbf{a}_{i + 1}$ (given by the last application of the Connecting Lemma)
  by $Q_1, Q_2, \dotsc, Q_{m - 1}$, then Proposition~\ref{prop:embedding-square}
  implies that
  \[
    A = \bigcup_{i \in [m]} A_i \cup \bigcup_{i \in [m - 1]} Q_i
  \]
  is an $(\mathbf{a}_1, \mathbf{b}_m, X)$-absorber: consider some subset $X'
  \subseteq X$ and for each $i \in [m]$ let $P_i \subseteq A_i$ be the
  square-path from $\mathbf{a}_i$ to $\mathbf{b}_i$ which contains $x_i$ if and
  only if $x_i \notin X'$ and, moreover, contains all other vertices in $A_i$.
  Such a path exists as $A_i$ is an $(\mathbf{a}_i, \mathbf{b}_i,
  \{x_i\})$-absorber. Then
  \[
    \mathbf{a}_1 \xrsquigarrow{$P_1$} \mathbf{b}_1 \xrsquigarrow{$Q_1$}
    \mathbf{a}_2 \xrsquigarrow{$P_2$} \mathbf{b}_2 \xrsquigarrow{$Q_2$} \;
    \cdots \; \xrsquigarrow{$Q_{m - 1}$} \mathbf{a}_m \xrsquigarrow{$P_m$}
    \mathbf{b}_m
  \]
  is a square-path from $\mathbf{a}_1$ to $\mathbf{b}_m$ which contains all
  vertices in $A$ except those in $X'$. This concludes the proof of the lemma.
\end{proof}

\section{Proof of the Connecting Lemma}\label{sec:connecting-lemma}

In the remainder of the paper we give the proof of the Connecting Lemma. For the
convenience of the reader let us first restate the lemma.

\connectinglemma*

The proof relies on the following lemma whose proof we defer to the next
subsection. The idea behind it is that even after removal of a not too large
subset $X$ from the reservoir $W$, we can find a copy of a $(b, 4\log
n)$-connecting path, or in the phrasing of the lemma above---a `matching',
connecting at least one pair $\bx_i$ to the corresponding pair $\by_i$.

\begin{lemma}\label{lem:fractional-connecting}
  For every $b \in \{1, 2\}$ and every $\alpha > 0$, there exist positive
  constants $\eps := \eps(\alpha)$ and $C := C(\alpha)$ such that w.h.p.\ for a
  random graph $\Gamma \sim \Gnp$ every $G \in \calG(\Gamma, n, 2\alpha, p)$ has
  the following property.

  Let $s \geq C\log^4 n/p^2$ be an integer. Then there are at least $(1 -
  n^{-4})\binom{n}{s}$ subsets $W \subseteq V(\Gamma)$ of size $s$ satisfying
  the following. For every family of disjoint $4$-tuples $\{\bx_i, \by_i\}_{i
  \in [t]} \subseteq (V(G) \setminus W)^4$, such that $t \log n \le \eps^6 |W|$,
  and every subset $X \subseteq W$ of size $|X| \leq 8t\log n$, there exists an
  $i \in [t]$ and an $(\{ (\bx_i, \by_i) \}, b, 4\log n)$-matching in $W
  \setminus X$.
\end{lemma}
\begin{proof}[Proof of Lemma~\ref{lemma:connecting-lemma}]
  For given $\alpha$, let $\eps =
  \eps_{\ref{lem:fractional-connecting}}(\alpha/2)$ and $C =
  C_{\ref{lem:fractional-connecting}}(\eps, \alpha/2)$. Set $\ell = 4\log n$ and
  $\calI = \{(\bx_i, \by_i)\}_{i \in [t]}$. Let $W$ be one of the $(1 -
  n^{-4})\binoms{n}{s}$ subsets satisfying the conclusion of
  Lemma~\ref{lem:fractional-connecting}.

  We define an $\ell$-uniform hypergraph $\cH$ on the vertex set $\calI \cup W$
  whose edge set is defined as follows. For every $4$-tuple $(\bx_i, \by_i)$ and
  every set $Y \subseteq W$ of size $\ell - 4$, we add an edge $(\bx_i, \by_i)
  \cup Y$ if and only if $G$ contains a $(b, \ell)$-connecting-path connecting
  $\bx_i$ to $\by_i$ and its internal vertices belong to $Y$. Clearly, if there
  is an $\calI$-saturating matching in $\cH$, then there is an $(\calI, b,
  \ell)$-matching in $W$. We use Haxell's criteria
  (Theorem~\ref{thm:haxells-theorem}) in order to show this.

  Let $\calI' \subseteq \calI$ and $X \subseteq W$ be arbitrary subsets such
  that $|X| \leq 2|\calI'| \cdot \ell$. It is enough to show that for some
  $(\bx_i, \by_i) \in \calI'$ there is a $(b, \ell)$-connecting-path connecting
  $\bx_i$ to $\by_i$ whose internal vertices are completely contained in the set
  $W \setminus X$. This in turn implies that $\cH$ contains an edge intersecting
  $\calI'$ and not intersecting $X$ and the condition of
  Theorem~\ref{thm:haxells-theorem} is satisfied. Applying
  Lemma~\ref{lem:fractional-connecting} with $\calI'$ (as $\{(\bx_i, \by_i)\}_{i
  \in [t]}$) gives us exactly that. Namely, we may apply the lemma since
  $|\calI'|\log n \leq t\log n \leq \eps^6|W|$ and $|X| \leq 2|\calI'| \cdot
  \ell \leq 8t\log n$.
\end{proof}

\subsection{Proof of Lemma~\ref{lem:fractional-connecting}}

Let us set $\eps = \min\{ 1/2^{300}, \alpha^{10},
\eps_{\ref{prop:random-set-is-good}}(2\alpha) \}$, $C = 5\eps^{-27}$, and take
$m = 4 \log n - 4$. Let $G$ be an arbitrary member of $\calG(\Gamma, n, 2\alpha,
p)$. Let $s \geq C\log^4n/p^4$ and $\tilde n := s/(5\log n) \geq
\eps^{-27}\log^3n/p^2$. Let $W$ be a set of size $s$ chosen uniformly at random
among all such sets. As by Proposition~\ref{prop:random-set-is-good} w.h.p.\
$\Gamma$ is such that for every $G \in \calG(\Gamma, n, 2\alpha, p)$ at least
$(1-n^{-5})\binom{n}{\tilde n}$ sets of size $\tilde n$ are
$(\alpha,\eps,p)$-good, it follows by Chernoff's inequality for
hypergeometrically distributed random variables that with high probability at
least $(1-n^{-4})\binom{s}{\tilde n}$ subsets $W' \subseteq W$ of size $\tilde
n$ are $(\alpha,\eps,p)$-good. In particular, w.h.p.\ $\Gamma$ is such that for
every $G \in \calG(\Gamma, n, 2\alpha, p)$, there are at least
$(1-n^{-4})\binom{n}{s}$ sets $W \subseteq V(G)$ of size $s$ which in turn
contain at least $(1-n^{-4})\binom{s}{\tilde n}$ subsets of size $\tilde n$
which are $(\alpha,\eps,p)$-good. Condition on this event and let $W$ be one of
those $(1-n^{-4})\binom{n}{s}$ sets of size $s$.

Fix an arbitrary $X \subseteq W$ of size $|X| \leq 8t\log n$. We now show the
existence of disjoint subsets $W_1, \dotsc, W_m$ of $W$ which for all $i \in
[m]$ satisfy:
\begin{enumerate}[leftmargin=3em, font=\itshape, label=(W\arabic*)]
  \item\label{w:class-size-bound} $|W_i| = \tilde n$, for $\tilde n :=
    \left\lceil \frac{|W|}{5 \log n} \right\rceil \ge \frac{\eps^{-27} \log^3
    n}{p^2}$,
  \item\label{w:goodness} $W_i$ is $(\alpha, \eps, p)$-good, and
  \item\label{w:X-size-ub} $|X \cap W_i| \leq \eps^5 \tilde n$.
\end{enumerate}
Consider a uniformly at random chosen collection $W_1, \dotsc, W_m$ of disjoint
subsets of $W$ satisfying \ref{w:class-size-bound}. We claim that such a random
collection satisfies \ref{w:goodness} and \ref{w:X-size-ub} with positive
probability. First, as each $W_i$ is u.a.r.\ chosen from $W$ which contains
$(1-n^{-4})\binom{s}{\tilde n}$ subsets which are good, \ref{w:goodness} follows
the union bound over all $i \in [m]$. As for \ref{w:X-size-ub}, observe first
that $|X| \le 8t\log n \le 8\eps^{6} |W|$. If $|X| \le \eps^{-3}\log^2 n$ then
the claim holds vacuously. Otherwise another application of Chernoff's
inequality and the union bound implies that with probability at least
\[
  1 - 4\log n \cdot e^{-\eps^2 |X|/(6 \log n)} \ge 1 - 4\log n \cdot e^{- \log
  n/(6\eps)} \ge 1 - n^{-5},
\]
we have
\[
  |X \cap W_i| \le (1 + \eps) |X| \frac{\tilde n}{|W|} \le (1 + \eps) \cdot
  8\eps^{6} \tilde n \leq \eps^5 \tilde n,
\]
for every $i \in [m]$. As $\max\{\eps^{-3} \log^2 n, \eps^5 \tilde n\} = \eps^5
\tilde n$, it holds that $|X \cap W_i| \le \eps^5 \tilde n$. From now on we thus
assume that we have disjoint subsets $W_1, \dotsc, W_{m}$ that satisfy
\ref{w:class-size-bound}--\ref{w:X-size-ub}. For convenience, we also set $X_i
:= X \cap W_i$ and $\tilde W_i := W_i \setminus X_i$, for every $i \in [m]$.

The goal of the remainder of the proof is to show that there exists an embedding
$g$ of a $(b, m + 4)$-pseudo-path connecting $\bx_i$ to $\by_i$, for some $i \in
[t]$. In order to do this, we first introduce a couple of definitions. Let $f
\colon \{0, \dotsc, m\} \to \Z$ be a function defined as:
\[
  f(i) =
    \begin{cases}
      i - 1, & \text{if $i$ is even}, \\
      i - b, & \text{otherwise}. \\
    \end{cases}
\]

Note that this function can be used to describe the left neighbour other than
$u_{i - 1}$ of a vertex $u_i$, $i \ge 3$, in a $(b,\ell)$-pseudo-path. Indeed,
the two left neighbours are $u_{f(i - 1)}$ and $u_{i - 1}$. Next, we define a
graph that is the union of all $(b, \ell)$-pseudo-paths that start in a set of
given edges.

\begin{definition}[{\bf Projection graph}]
  Let $\pi \colon [m] \to [m]$ be a permutation of the set $[m]$ and let
  $\{(a_i, b_i)\}_{i \in [t]} \subseteq (V(G) \setminus W)^2$, denoted by
  $\calI$, be a set of $t$ disjoint ordered pairs. We define an \emph{$(\calI,
  \pi)$-projection graph $F$} on the vertex set
  \[
    V(F) = W_{-1} \cup W_0 \cup W_{\pi(1)} \cup \dotsb \cup W_{\pi(m)},
  \]
  where $W_{-1} = \{a_1, \dotsc, a_t\}$ and $W_0 = \{b_1, \dotsc, b_t\}$. The
  edge set of $F$ is defined inductively as follows. Let $E_0$ be the edges
  between the sets $W_{-1}$ and $W_0$, i.e.\ all edges in $\bigcup_{i \in [t]}
  \{a_i, b_i\}$. Then for all $j \in [m]$ we let
  \begin{align*}
    E_j := \big\{ \{u, v\}, \{w, v\} : v \in \tilde W_{\pi(j)} \text{ and }
    \{u, w\} \in E_{j - 1}(W_{\pi(f(j - 1))}, W_{\pi(j - 1)}), \{w, v\}, \{u,
    v\} \in E \big\}.
  \end{align*}
  Lastly, we set the edge set of $F$ to $E_0 \cup \dotsb \cup E_m$
  (Figure~\ref{fig:fgraph}).
\end{definition}

\begin{figure}[!htbp]
  \centering
  \includegraphics[scale=0.6]{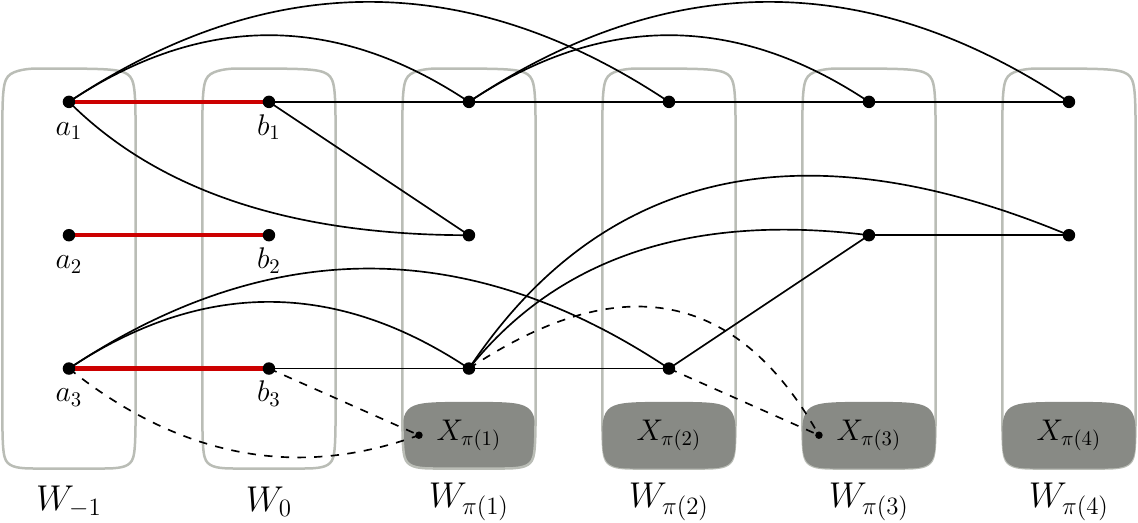}
  \caption{An example of a projection graph $F$ where $m = 4$ and $b = 2$.
  Dashed edges do not belong to the projection graph $F$.}
  \label{fig:fgraph}
\end{figure}

To understand this definition, observe that $E_{j - 1}$ are exactly those edges
for which an edge expansion into $\tilde W_{\pi(j)}$ `extends' the pseudo-path
constructed so far by a vertex from $\tilde W_{\pi(j)}$. Crucially, even though
the vertex set of $F$ is defined as the union of the sets $W_i$, the edge set
consists {\em only} of edges that run between $\tilde W_i$'s.

The next proposition thus follows immediately.

\begin{proposition}\label{prop:proj-graph-embedding}
  Let $b \in \{1, 2\}$, let $\pi \colon [m] \to [m]$ be a permutation of the set
  $[m]$, and let $\{(a_i, b_i)\}_{i \in [t]} \subseteq (V(G) \setminus W)^2$,
  denoted by $\calI$, be a family of $t$ disjoint ordered pairs. Furthermore,
  let $F$ be an $(\calI, \pi)$-projection graph. Then for each $j$, where $1 \le
  j \le m /2$, and each edge $\{v, w\} \in E_F(W_{\pi(2j - 1)}, W_{\pi(2j)})$,
  there exists an $i \in [t]$ and an embedding of a $(b, 2j + 2)$-pseudo-path in
  $F$ connecting $(a_i, b_i)$ to $(v, w)$ that contains exactly one vertex from
  each set $W_i$, $i \in [m]$, and no vertex from the set $X$. \qed
\end{proposition}

The following claim is the main tool in the proof of
Lemma~\ref{lem:fractional-connecting}. The proof of the claim is technical and
quite involved and thus it is presented in the next section. In the remainder of
this section, we show how the claim implies
Lemma~\ref{lem:fractional-connecting}.

\begin{claim}\label{cl:one-reach-all}
  Let $t' \in \N$ be such that $t/3 \le t' \le t$. Let $\pi \colon [m] \to [m]$
  be a permutation of the set $[m]$ and let $\{(a_i, b_i)\}_{i \in [t']}
  \subseteq (V(G) \setminus W)^2$, denoted by $\calI$, be a set of $t'$ disjoint
  edges in $G$. Then there exists $(a_i, b_i) \in \calI$ such that
  \[
    e_F(W_{\pi(m/2 - 1)}, W_{\pi(m/2)}) \ge \frac{2}{3} \tilde n^2 p \qquad
    \text{and} \qquad e_{F}(W_{\pi(m/2 + 1)}, W_{\pi(m/2 + 2)}) \ge \frac{2}{3}
    \tilde n^2 p,
  \]
  where $F$ is the $(\{(a_i, b_i)\}, \pi)$-projection graph.
\end{claim}

Having the previous claim at hand, we finish the proof of
Lemma~\ref{lem:fractional-connecting}. Let $\pi_1$ be the identity permutation
of the set $[m]$ and let $\pi_2$ be a permutation of $[m]$ defined as $\pi_2(i)
= m - i + 1$. Let $\calI_x$ be the largest subset of $\{\bx_i\}_{i \in [t]}$
such that for every $\bx_i \in \calI_x$ it holds that
\[
  e_{F'}(W_{m/2 - 1}, W_{m/2}) \ge \frac{2}{3} \tilde n^2 p,
\]
where $F'$ is the $(\calI_x, \pi_1)$-projection graph. If $|\calI_x| \le t/2$
then by applying Claim~\ref{cl:one-reach-all} with $\{\bx_i\}_{i \in [t]}
\setminus \calI_x$ (as $\calI$) and $\pi_1$ (as $\pi$) we get a contradiction
with the maximality of $\calI_x$. Thus $|\calI_x| > t/2$. Similarly, let
$\calI_y$ be the largest subset of $\{\oby_i\}_{i \in [t]}$ such that for every
$\oby_i \in \calI_y$ it holds that
\[
  e_{F''}(W_{\pi_2(m/2 + 1)}, W_{\pi_2(m/2 + 2)}) \ge \frac{2}{3}\tilde n^2p,
\]
where $F''$ is the $(\calI_y, \pi_2)$-projection graph. As in the case of
$\calI_x$ it must be that $|\calI_y| > t /2$. The fact that both $\calI_x$ and
$\calI_y$ are larger than $t/2$ implies that there must be a single integer
$i^\star \in [t]$ such that
\begin{equation}\label{eq:joining}
  e_{F_x}(W_{m/2 - 1}, W_{m/2}) \ge \frac{2}{3} \tilde n^2p \qquad \text{and}
  \qquad e_{F_y}(W_{\pi_2(m/2 + 1)}, W_{\pi_2(m/2 + 2)}) \ge \frac{2}{3} \tilde
  n^2p,
\end{equation}
where $F_x$ and $F_y$ are the $(\{\bx_{i^\star}\}, \pi_1)$-projection graph and
the $(\{\oby_{i^\star}\}, \pi_2)$-projection graph, respectively.

Let $\ell_1 = m/2$. Note that $\{\pi_2(m - \ell_1 + 2) , \pi_2(m - \ell_1 + 1)\}
= \{\ell_1 - 1, \ell_1\}$. Hence, from \eqref{eq:joining} we have
\[
  e_{F_x}(W_{\ell_1 - 1}, W_{\ell_1}) \ge \frac{2}{3} \tilde n^2 p \qquad
  \text{and} \qquad e_{F_y}(W_{\ell_1 - 1}, W_{\ell_1}) \ge \frac{2}{3} \tilde
  n^2 p.
\]
This implies, by \ref{p:unif-density}, that there must exist an edge $e = \{u,
v\}$ such that
\[
  e \in E_{F_x}(W_{\ell_1 - 1}, W_{\ell_1}) \cap E_{F_y}(W_{\ell_1 -1},
  W_{\ell_1}).
\]

By Proposition~\ref{prop:proj-graph-embedding} and the definitions of $\pi_1$
and $\pi_2$ we get that there exist two embeddings $g_1$ and $g_2$ of a $(b,
\ell_1 + 2)$-pseudo-path and a $(b, \ell_1 + 4)$-pseudo-path such that:
\begin{itemize}
  \item $(g_1(u_1), g_1(u_2)) = \bx_{i^\star}$ and $(g_2(u_1), g_2(u_2)) =
    \oby_{i^\star}$,
  \item $g_1(u_i) \in W_{i - 2}$, for every $i \in \{3, \dotsc, \ell_1\}$,
  \item $g_2(u_i) \in W_{m - i + 3}$, for every $i \in \{3, \dotsc, \ell_1 +
    2\}$,
  \item $g_1(u_{\ell_1 + 1}) = u$, $g_1(u_{\ell_1 + 2}) = v$, and
  \item $(g_2(u_{\ell_1 + 3}), g_2(u_{\ell_1 + 4})) =
    \begin{cases}
      (v, u), & \text{if } b = 1, \\
      (u, v), & \text{if } b = 2.
    \end{cases}$
\end{itemize}
Using Proposition~\ref{prop:embedding-square} and
Propositions~\ref{prop:embedding-backbone} we conclude that there exists an
$(\{(\bx_i, \by_i)\}, b, m + 4)$-matching in $W \setminus X$, as desired. This
concludes the proof of Lemma~\ref{lem:fractional-connecting}. It remains to
prove Claim~\ref{cl:one-reach-all}.

\subsection{Proof of Claim~\ref{cl:one-reach-all}}\label{sec:main-claim}

In this subsection we give the proof of Claim~\ref{cl:one-reach-all}. As the
choice of the permutation $\pi$ does not play a role in the proof, we assume
$\pi$ is the identity permutation and we completely omit it from the definition
of the projection graph. Thus, throughout the section when we write
$\calI$-projection graph we mean $(\calI, \pi)$-projection graph, where $\pi$ is
the identity permutation. For a projection graph $F$, we refer to a pair of
bipartite graphs $F[W_{f(i)}, W_i]$ and $F[W_{f(i + 1)}, W_{i + 1}]$ as the {\em
$i$-th step of $F$}. Next, we introduce some terminology and define when a step
is {\em expanding} or {\em non-expanding}.

\begin{definition}
  Let $F$ be a projection graph. Let $C \ge 1$ be a real number and let $i \ge
  1$. We say that the $i$-th step of $F$ is $C$-{\em expanding} if
  \[
    e_F(W_{f(i + 1)}, W_{i + 1}) \ge C \cdot e_F(W_{f(i)}, W_{i}).
  \]
  Otherwise, it is $C$-{\em non-expanding}.
\end{definition}

Due to the asymmetry in the definition of pseudo-paths (for $b = 2$), it is
easier to not consider every step, but to group steps into {\em blocks of two}.

\begin{definition}
  Let $F$ be a projection graph. Let $C \ge 1$ be a real number and let $i \ge
  1$. We say that the $i$-th block is $C$-{\em expanding} in $F$, if at least
  one of the $2i$-th and the $(2i + 1)$-st step of $F$ are $C$-expanding.
  Otherwise, it is $C$-{\em non-expanding}.
\end{definition}

To understand how the two definitions relate, it is helpful to observe that for
the edges between blocks, i.e.\ the red edges in Figure~\ref{fig:block}, we have
$(W_{2i - 1}, W_{2i}) = (W_{f(2i)}, W_{2i})$, by the definition of~$f$.

\begin{figure}[!htbp]
  \centering
  \includegraphics[scale=0.475]{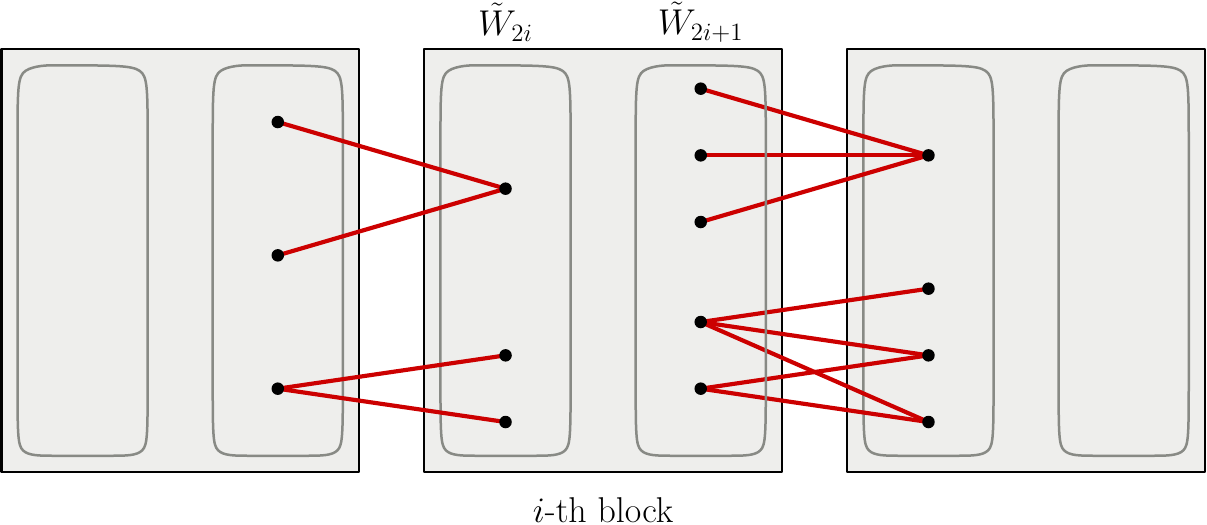}
  \caption{If $i$-th block is expanding, then the right set of red edges is
  larger than the left set of red edges. (Note: this expansion does not follow
  from the definition, but requires some proofs.)}
  \label{fig:block}
\end{figure}

The intuition behind expanding blocks is that the number of edges in $F$ `after'
an expanding block should be larger than the number of edges in $F$ `before' it
(see Figure~\ref{fig:block}). The following claim makes this precise.

\begin{claim}\label{claim:how-blocks-can-grow}
  Let $F$ be a projection graph and $1 \leq i \leq m/2 - 2$.
  \begin{enumerate}[(i)]
    \item\label{claim:sparse-blocks} Suppose $e_F(W_{2i - 1}, W_{2i}) \ge
      \eps^{-21} \tilde n \log^2 n$. If the $i$-th block is
      $(1/\eps)$-expanding, then
      \[
        e_F(W_{2i + 1}, W_{2i + 2}) \ge (1/\sqrt\eps) \cdot e_F(W_{2i - 1},
        W_{2i}),
      \]
      and otherwise
      \[
        e_F(W_{2i + 1}, W_{2i + 2}) \ge \frac{\alpha^2}{256} \tilde n^2 p.
      \]
    \item\label{claim:dense-blocks} Suppose $e_F(W_{2i - 1}, W_{2i}) \ge 2\eps
      \tilde n^2p$. If the $i$-th block is $(1 + 4\sqrt\eps)$-expanding, then
      \[
        e_F(W_{2i + 1}, W_{2i + 2}) \ge (1 + 2\sqrt\eps) \cdot e_F(W_{2i - 1},
        W_{2i}),
      \]
      and otherwise
      \[
        e_F(W_{2i + 3}, W_{2i + 4}) \ge (2/3 + \alpha/2 )\tilde n^2p.
      \]
  \end{enumerate}
\end{claim}

The proof of the claim is quite technical and relies mostly on properties given
by Lemma~\ref{lem:everything} about expansion of edges and triangles; we defer
it to the next section. With this claim at hand we are ready to give the proof
of Claim~\ref{cl:one-reach-all}.

\begin{proof}[Proof of Claim~\ref{cl:one-reach-all}]
  As mentioned earlier, we assume w.l.o.g.\ that $\pi$ is the identity
  permutation as the actual choice of $\pi$ does not play a role in the proof.
  The proof comprises of four natural steps:
  \begin{enumerate}[(1)]
    \item\label{step-1} starting from the edges in $\calI$ show that $e_F(W_{2i
      - 1}, W_{2i}) \geq \eps^{-21} \tilde n \log^2 n$, for all $m/80 \leq i
      \leq m/2$;
    \item\label{step-2} starting from the edges obtained in the previous step
      show that $e_F(W_{2i - 1}, W_{2i}) \geq 2\eps \tilde n^2p$, for all $m/40
      \leq i \leq m/2$;
    \item\label{step-3} knowing that starting from all edges in $\calI$ we can
      reach $2\eps \tilde n^2 p$ edges in some number of steps, show that there
      is at least one edge $e \in \calI$ such that $e_{F'}(W_{2i - 1}, W_{2i})
      \geq 2\eps \tilde n^2p$, for all $m/10 \leq i \leq m/2$, where $F'$ is the
      $e$-projection graph;
    \item\label{step-4} starting from the edges obtained in the previous step
      show that $e_{F'}(W_{2i - 1}, W_{2i}) \geq (2/3)\tilde n^2p$, for all $m/5
      \leq i \leq m/2$.
  \end{enumerate}

  \textbf{Step~\ref{step-1}.} Let us first deal with the trivial case in which
  $p > \eps^2$. For $j \geq 0$ and a vertex $v \in W_j$ with $\deg_F(v,
  W_{f(j)}) \geq 1$, by \ref{p:good-codeg} and as $|X_{j + 1}| \leq \eps^5
  \tilde n$ (recall,~\ref{w:X-size-ub}), it holds that
  \[
    \deg_F(v, W_{j + 1}) \geq \alpha \tilde np^2 - |X_{j + 1}| \geq \alpha
    \eps^4 \tilde n - \eps^5\tilde n \geq \eps^5 \tilde n,
  \]
  and thus also $e_F(W_{2i - 1}, W_{2i}) \geq \eps^5 \tilde n^2 \gg \tilde n
  \log^2 n$, for all $i \geq 1$.

  In the remainder of the proof we assume $p \leq \eps^2$. We first show by
  induction that the following invariant is true for every $j \in \{0, 1,
  \dotsc, m\}$: there exists $E_j \subseteq E_F(W_{f(j)}, W_j)$ such that
  \begin{enumerate}[(i), font=\itshape]
    \item\label{beg-matching-size} $|E_j \cap W_j| \geq \min\{ \eps^2\tilde n,
      |\calI|(1/\eps)^j \}$ and $|E_j \cap W_{f(j)}| \geq (1 - j\eps + f(j)\eps)
      \min\{ \eps^2\tilde n, |\calI|(1/\eps)^{f(j)} \}$, where $E_j \cap W_j$
      and $E_j \cap W_{f(j)}$ denote the set of vertices from $W_j$ and
      $W_{f(j)}$ incident to the edges in $E_j$, respectively,
    \item\label{beg-back-deg} for each $v \in W_{f(j)} \cup W_j$ we have
      $\deg_{E_j}(v) \le \eps/p$.
  \end{enumerate}
  The base of the induction $j = 0$ holds trivially as, by assumption, the
  starting set of edges is a matching of size $|\calI|$. Thus, let $j \ge 1$ and
  let us assume the hypothesis holds for all values smaller than $j$. Let $E_{j
  - 1} \subseteq E_F(W_{f(j - 1)}, W_{j - 1})$ be as given by the induction
  hypothesis for $j - 1$. Consider first the case $f(j) = j - 2$. Note that then
  $f(j - 1) = j - 2$ as well and thus $|E_{j - 1} \cap W_{f(j - 1)}| \geq (1 -
  \eps) \min\{ \eps^2\tilde n, |\calI|(1/\eps)^{j - 2} \}$. We apply
  Lemma~\ref{lem:everything}~$\ref{lemma:star-matching}$ with $W_{j - 1}, W_{j -
  2}, W_{j}$ (as $W_1, W_2, W_3$) and $E_{j - 1}$ (as $F_{12}$) to conclude that
  there exists a subset $U' \subseteq E_{j - 1} \cap W_{j - 2}$ of size
  \[
    |U'| \geq (1 - \eps) \min\{ |E_{j - 1} \cap W_{j - 2}|, \eps/p^2 \} \geq (1
    - 2\eps) \min\{ \eps^2\tilde n, |\calI|(1/\eps)^{j - 2}, \eps/p^2 \}
  \]
  and an $(\alpha \tilde np^2/2)$-star-matching $M$ saturating $U'$. We may
  indeed apply the lemma by $\ref{beg-back-deg}$ and since
  \[
    |X_j| \leq |X| \leq 4t \log n \leq 12|\calI| \log n.
  \]
  As $\alpha\tilde np^2/2 \geq 1/\eps^4$, $E_j := M$ satisfies all the required
  properties. In case $f(j) = j - 1$ we apply
  Lemma~\ref{lem:everything}~$\ref{lemma:star-matching}$ with $W_{j - 2}, W_{j -
  1}, W_j$ (as $W_1, W_2, W_3$). Observing that $|E_{j - 1} \cap W_{j - 1}| \geq
  \min\{ \eps^2 \tilde n, |\calI|(1/\eps)^{j - 1}\}$, and doing the same
  analysis as in the previous case shows that the invariant holds also in this
  case.

  With these preparations at hand we can now finish the proof. For $j \ge
  \frac{m}{200} \ge \frac{\log n}{200}$, $\ref{beg-matching-size}$ implies
  \[
    |E_j \cap W_j| \ge (1 - 2\eps) \min\{ \eps^2 \tilde n, |\calI|
    (1/\eps)^{\log n/200} \} \geq \eps^3 \tilde n,
  \]
  as $\eps$ is chosen in order for $\log (1/\eps) > 200$ to hold. Next, for
  every $\frac{m}{100} \le i \leq \frac{m}{2}$, we apply
  Lemma~\ref{lem:everything}~$\ref{lemma:star-matching}$ with $W_{f(2i - 1)},
  W_{2i - 1}, W_{2i}$ (as $W_1, W_2, W_3$) and $E_{2i - 1}$ (as $F_{12}$) to
  conclude that for every subset $U \subseteq E_{2i - 1} \cap W_{2i - 1}$ of
  size $|U| = \min\{\eps/p^2, |E_{2i - 1} \cap W_{2i - 1}|\}$ there exists a
  subset $U' \subseteq U$ of size $(1 - \eps)|U|$ and an $(\alpha \tilde
  np^2/2)$-star-matching $M$ saturating $U'$ (note that degree assumption needed
  for Lemma~\ref{lem:everything}~$\ref{lemma:star-matching}$ holds by
  $\ref{beg-back-deg}$). This further implies
  \[
    e_F(W_{2i - 1}, W_{2i}) \geq (1 - \eps) |E_{2i - 1} \cap W_{2i - 1}| \geq (1
    - \eps) \eps^3 \tilde n \cdot (\alpha/2) \tilde np^2 \ge \eps^4 \tilde n^2
    p^2 \ge \eps^{-21} \tilde n \log^2 n.
  \]

  \textbf{Step~\ref{step-2}.} Fix $i = m/80$. From Step~$\ref{step-1}$ we know
  that $e_F(W_{2i - 1}, W_{2i}) \geq \eps^{-21}\tilde n \log^2 n$. From
  Claim~\ref{claim:how-blocks-can-grow}~$\ref{claim:sparse-blocks}$, and the
  fact that $(\alpha^2/256) \tilde n^2 p \ge \eps^{-21} \tilde n \log^2 n$, we
  further have that for all $j \ge i + 1$
  \[
    e_F(W_{2j - 1}, W_{2j}) \ge \min \Big\{ \eps^{-(j - i)/2} \cdot e_F(W_{2i -
    1}, W_{2i}), \frac{\alpha^2}{256} \tilde n^2 p \Big\}.
  \]
  Recall that $m \ge \log n$ and that we have chosen $\eps$ sufficiently small
  so that $\eps^{-(j - i)/2} \ge n^2$ for all $j \ge i + m/80$. The assertion in
  Step~\ref{step-2} then follows as $\alpha^2/256 \ge 2\eps$.

  \textbf{Step~\ref{step-3}.} Fix $i = m/40$. From Step~$\ref{step-2}$ we know
  that $e_F(W_{2i - 1}, W_{2i}) \geq 2\eps\tilde n^2 p$. Let $\calI = \calI_1
  \cup \dotsb \cup \calI_{\ceil{\tilde n^{1/3}}}$ be an arbitrary partition of
  $\calI$ such that $|\calI_j| \leq |\calI|/\ceil{\tilde n^{1/3}}$ and let $F_j$
  be the $\calI_j$-projection graph, for every $j \in \{1, \dotsc, \ceil{\tilde
  n^{1/3}}\}$. Since
  \[
    \sum_{j = 1}^{\ceil{\tilde n^{1/3}}} e_{F_j}(W_{2i - 1}, W_{2i}) \ge
    e_{F}(W_{2i - 1}, W_{2i}),
  \]
  there must be a $j^\star \in \{1, \dotsc, \ceil{\tilde n^{1/3}} \}$ such that
  \[
    e_{F_{j^\star}}(W_{2i - 1}, W_{2i}) \ge \frac{2\eps \tilde
    n^2p}{\ceil{\tilde n^{1/3}}} \ge \eps^{-21} \tilde n\log^2 n,
  \]
  where the second inequality holds as $\tilde n^{2/3} p \geq \eps^{-21} \log^2
  n$. By the same argument as in the Step~$\ref{step-2}$, this time for
  $F_{j^\star}$ and $i$, we get that for $k = i + \frac{m}{80}$
  \[
    e_{F_{j^\star}}(W_{2k - 1}, W_{2k}) \ge (\alpha^2/256) \tilde n^2p,
  \]
  implying that we can repeat the argument from above and partition ${\cal
  I}_{j^\star}$ into $\ceil{\tilde n^{1/3}}$ parts. The claim follows by
  applying this argument successively at most two more times.

  \textbf{Step~\ref{step-4}.} Fix $i = m/10$ and let $e \in \calI$ be the edge
  obtained in Step~$\ref{step-3}$ and $F'$ the $e$-projection graph. From
  Step~$\ref{step-3}$ we know that $e_{F'}(W_{2i - 1}, W_{2i}) \geq
  (\alpha^2/256)\tilde n^2p$. Let us choose a constant $L := L(\alpha, \eps)$
  such that $(1 + 2\sqrt\eps)^{L - 2} (\alpha^2/256) > 2$. Observe that not all
  blocks $i, \dotsc, i + L - 2$, can be $(1 + 4\sqrt\eps)$-expanding as then
  Claim~\ref{claim:how-blocks-can-grow}~$\ref{claim:dense-blocks}$ would imply
  \[
    e_{F'}(W_{2(i + L - 2) - 1}, W_{2(i + L - 2)}) \ge (1 + 2\sqrt\eps)^{L - 2}
    \cdot \frac{\alpha^2}{256} \tilde n^2 p \ge 2 \tilde n^2p,
  \]
  which is a contradiction with \ref{p:unif-density}. Let $i^\star \in [i, i + L
  - 2]$ be the smallest index such that the $i^\star$-th block is $(1 +
  4\sqrt\eps)$-non-expanding. Then $e_{F'}(W_{2i^\star + 3}, W_{2i^\star + 4})
  \ge (2/3 + \alpha/2) \tilde n^2 p$, by
  Claim~\ref{claim:how-blocks-can-grow}~$\ref{claim:dense-blocks}$.

  If the $(i^\star + 2)$-nd block is $(1 + 4\sqrt\eps)$-expanding then
  \[
    e_{F'}(W_{2i^\star + 5}, W_{2i^\star + 6}) \geq (1 + 2\sqrt\eps) \cdot
    e_{F'}(W_{2i^\star + 3}, W_{2i^\star + 4}) \geq (2/3 + \alpha/2)\tilde n^2p.
  \]
  If, on the other hand, the $(i^\star + 2)$-nd block is $(1 +
  4\sqrt\eps)$-non-expanding, then $e_{F'}(W_{2i^\star + 7}, W_{2i^\star + 8})
  \geq (2/3 + \alpha/2)\tilde n^2p$, by
  Claim~\ref{claim:how-blocks-can-grow}~$\ref{claim:dense-blocks}$. In addition,
  by applying Lemma~\ref{lem:everything}~$\ref{lemma:dense-doesnt-drop}$ to
  $W_{2i^\star + 3}$, $W_{2i^\star + 4}$, $W_{2i^\star + 5}$ (as $W_1, W_2,
  W_3$) and $E_{F'}(W_{2i^\star + 3}, W_{2i^\star + 4})$ (as $F_{12}$) we obtain
  by symmetry that
  \[
    e_{F'}(W_{2i^\star + 3}, W_{2i^\star + 5}), e_{F'}(W_{2i^\star + 4},
    W_{2i^\star + 5}) \geq (1 - \sqrt\eps)e_{F'}(W_{2i^\star + 3}, W_{2i^\star +
    4}).
  \]
  Applying Lemma~\ref{lem:everything}~$\ref{lemma:dense-doesnt-drop}$ again,
  this time to $W_{f(2i^\star + 5)}, W_{2i^\star + 5}, W_{2i^\star + 6}$ (as
  $W_1, W_2, W_3$) and $E_{F'}(W_{f(2i^\star + 5)}, W_{2i^\star + 5})$ (as
  $F_{12}$), we obtain
  \[
    e_{F'}(W_{2i^\star + 5}, W_{2i^\star + 6}) \geq (1 - \sqrt\eps)^2
    e_{F'}(W_{2i^\star + 3}, W_{2i^\star + 4}) \geq (2/3)\tilde n^2 p,
  \]
  by our choice of $\eps$.

  Repeating this argument, that is starting from $E_{F'}(W_{2i^\star + 5},
  W_{2i^\star + 6})$ or $E_{F'}(W_{2i^\star + 7}, W_{2i^\star + 8})$ depending
  on whether the $(i^\star + 2)$-nd block was $(1 + 4\sqrt\eps)$-expanding or
  not, shows that $e_{F'}(W_{2j - 1}, W_{2j}) \geq (2/3)\tilde n^2p$ for all $j
  \in [i^\star + 2, m/2]$, and the claim follows since $i^\star \leq m/5 - 2$
  (recall that we had set $m = 4\log n - 4$ and thus $m/2$ is even).

  This completes the proof of Claim~\ref{cl:one-reach-all}.
\end{proof}

\subsection{Proof of Claim~\ref{claim:how-blocks-can-grow}}
\label{sec:block-expansion}

In this section we provide the proof of the assertions in
Claim~\ref{claim:how-blocks-can-grow}. We start with some general remarks.
Recall that the $i$-th block consists of steps $2i$ and~$2i + 1$. Recall also
that step~$2i$ extends edges between sets $W_{2i - 1}$ and $W_{2i}$ into set
$W_{2i + 1} \setminus X_{2i + 1}$ via triangles. Similarly, step $2i + 1$
extends edges between sets $W_{f(2i + 1)}$ and $W_{2i + 1}$ into set $W_{2i + 2}
\setminus X_{2i + 2}$. These extensions exactly mimic the setting covered by
Lemma~\ref{lem:everything}. However, the actual set $W_{f(2i + 1)}$ depends on
the value of $b$ of the $b$-pseudo-path that we want to construct. In order to
hide this difference we often use variables $t_1, \dotsc, t_4$ as follows: $t_4
= 2i + 2$, $t_3 = 2i + 1$, $t_2 = f(2i + 1)$, and $t_1$ is the unique element
from $\{2i, 2i - 1\} \setminus \{t_2\}$. One easily checks that this implies
that we can always apply Lemma~\ref{lem:everything} with $W_{t_1}, W_{t_2},
W_{t_3}$ to address the $2i$-th step and to $W_{t_2}, W_{t_3}, W_{t_4}$ (as
$W_1, W_2, W_3$) to address the $(2i + 1)$-st step.

\begin{proof}[Proof of~$\ref{claim:sparse-blocks}$ in
Claim~\ref{claim:how-blocks-can-grow}: $i$-th block is $(1/\eps)$-expanding]
  If both steps $2i$ and $2i + 1$ are $(1/\eps)$-expanding then the claim
  follows directly from the definition of an expanding step together with the
  observation that $f(j) = j - 1$ whenever $j$ is even. Thus, let us assume that
  one of the two steps is $(1/\eps)$-non-expanding. We aim to prove the
  following for every $j \geq 1$:
  \begin{equation}\label{eq:sparse-blocks-eq1}
    \text{\em If} \quad e_F(W_{f(j)}, W_{j}) \ge \eps^{-21} \tilde n \log^2 n
    \quad \text{\em then} \quad e_F(W_{f(j + 1)}, W_{j + 1}) \ge
    \frac{\alpha}{16} e_F(W_{f(j)}, W_{j}).
  \end{equation}
  Note that this implies the claim regardless of whether the non-expanding step
  is the first or the second step within the block, as then
  \[
    e_F(W_{2i + 1}, W_{2i + 2}) \ge \frac{1}{\eps} \cdot \frac{\alpha}{16}
    e_F(W_{2i - 1}, W_{2i}) \ge \frac{1}{\sqrt\eps} \cdot e_F(W_{2i -1},
    W_{2i}),
  \]
  which is what we wanted to prove.

  We now prove \eqref{eq:sparse-blocks-eq1}. So assume $e_F(W_{f(j)}, W_j) \geq
  \eps^{-21} \tilde n \log^2 n$, for some $j \geq 1$. Observe that the
  definition of the function $f$ implies that \eqref{eq:sparse-blocks-eq1}
  involves exactly three sets $W_j$. Indeed, by the definition of $f$ we have
  $f(j + 1) \in \{j, f(j)\}$. Let $t_2 = f(j + 1)$ and $t_1$ be the unique
  element from $\{j, f(j)\} \setminus \{t_2\}$. Let $S \subseteq W_{t_2}$ be the
  set of vertices with the degree at most $\eps^{-4}\log n/p$ in $E_F(W_{t_1},
  W_{t_2})$ and set $M := W_{t_2} \setminus S$. Furthermore, let us denote
  $E_F(W_{t_1}, S)$ and $E_F(W_{t_1}, M)$ by $I_S$ and $I_M$, respectively. If
  $|I_S| \ge e_F(W_{t_1}, W_{t_2})/2 $ then by applying
  Lemma~\ref{lem:everything}~\ref{lemma:small-expand} with $W_{t_1}, W_{t_2},
  W_{j + 1}$ (as $W_1, W_2, W_3$), $S$ (as $U$), and $I_S$ (as $F_{12}$) we get
  \[
    e_F(S, W_{j + 1}) \ge \frac{1}{\eps^4} e_F(W_{t_1}, S) =
    \frac{1}{\eps^4}|I_S| \ge \frac{1}{2\eps^4} e_F(W_{f(j)}, W_{j}) \ge
    \frac{\alpha}{16} e_F(W_{f(j)}, W_{j}).
  \]
  On the other hand, if $|I_M| \ge e_F(W_{t_1}, W_{t_2})/2$, by applying
  Lemma~\ref{lem:everything}~\ref{lemma:medium-dont-shrink} with $W_{t_1},
  W_{t_2}, W_{j + 1}$ (as $W_1, W_2, W_3$), $M$ (as $U$), and $I_M$ (as
  $F_{12}$) we get
  \[
    e_F(M, W_{j+ 1}) \ge \frac{\alpha \tilde np}{4} |M| \ge \sum_{v \in M}
    \frac{\alpha}{8} \deg_{I_M}(v, W_{t_1}) = \frac{\alpha}{8} |I_M| \ge
    \frac{\alpha}{16} e_F(W_{f(j)}, W_{j}),
  \]
  where the second inequality follows from \ref{p:unif-deg}.
\end{proof}

\begin{proof}[Proof of~$\ref{claim:sparse-blocks}$ in
Claim~\ref{claim:how-blocks-can-grow}: $i$-th block is $(1/\eps)$-non-expanding]
  Set $t_4 = 2i + 2$, $t_3 = 2i + 1$, $t_2 = f(2i + 1)$, and let $t_1$ be the
  unique element from $\{2i, 2i - 1\} \setminus \{t_2\}$.

  By assumption, the steps $2i$ and $2i + 1$ are both $(1/\eps)$-non-expanding.
  Thus, in particular $e_F(W_{t_2}, W_{t_3}) < (1/\eps) e_F(W_{t_1}, W_{t_2})$.
  Additionally, since $e_F(W_{t_1}, W_{t_2}) \geq \eps^{-21} \tilde n \log^2 n$,
  we can apply Lemma~\ref{lem:everything}~\ref{lemma:sparse-expand-or} with
  $W_{t_1}, W_{t_2}, W_{t_3}$ (as $W_1, W_2, W_3$), and $E_F(W_{t_1}, W_{t_2})$
  (as $F_{12}$) to get a set $L_2 \subseteq W_{t_2}$ such that for every $v \in
  L_2$ we have $\deg_F(v, W_{t_1}) \geq \eps^{-4}\log n/p$ and
  \begin{equation}\label{eq:bad-block-to-dense-eq1}
    e_F(W_{t_1}, L_2) \geq (1 - \eps) e_F(W_{t_1}, W_{t_2}).
  \end{equation}
  The previous inequality implies $e_F(W_{t_1}, L_2) \geq \eps^{-20} \tilde n
  \log n$ and thus by applying
  Lemma~\ref{lem:everything}~\ref{lemma:medium-dont-shrink} with $W_{t_1},
  W_{t_2}, W_{t_3}$ (as $W_1, W_2, W_3$), $L_2$ (as $U$), and $E_F(W_{t_1},
  L_2)$ (as $F_{12}$) we conclude
  \begin{equation}\label{eq:bad-block-to-dense-eq2}
    e_F(L_2, W_{t_3}) \geq \frac{\alpha \tilde np}{4} |L_2|
    \osref{\ref{p:unif-deg}}{\geq} \frac{\alpha}{8} e_F(W_{t_1}, L_2)
    \osref{\eqref{eq:bad-block-to-dense-eq1}}{\geq} \frac{\alpha}{10}
    e_F(W_{t_1}, W_{t_2})
  \end{equation}
  Since the $i$-th block is non-expanding we have $e_F(W_{t_3}, W_{t_4}) \le
  (1/\eps^2) e_F(W_{t_1}, W_{t_2})$ and thus
  \[
    e_F(W_{t_3}, W_{t_4}) < \frac{1}{\eps^2} e_F(W_{t_1}, W_{t_2})
    \osref{\eqref{eq:bad-block-to-dense-eq2}}\leq \frac{1}{\eps^2}
    \frac{10}{\alpha} e_F(L_2, W_{t_3}) < \frac{1}{\eps^3} e_F(L_2, W_{t_3}).
  \]
  Hence, we can apply Lemma~\ref{lem:everything}~\ref{lemma:sparse-expand-or}
  with $W_{t_2}, W_{t_3}, W_{t_4}$ (as $W_1, W_2, W_3$), and $E_F(W_{t_2},
  W_{t_3})$ (as $F_{12}$) to get a set $L_3 \subseteq W_{t_3}$ such that for
  every $v \in L_3$ we have $\deg_F(v, W_{t_2}) \geq \eps^{-4}\log n/p$ and
  \begin{equation}\label{eq:bad-block-to-dense-eq3}
    e_F(W_{t_2}, L_3) \ge (1 - \eps) e_F(W_{t_2}, W_{t_3}) \geq (1 - \eps)
    e_F(L_2, W_{t_3}).
  \end{equation}
  Note that w.l.o.g.\ we may assume that $L_3$ contains all vertices $v \in
  W_{t_3}$ with $\deg_F(v, W_{t_2}) \geq \eps^{-4}\log n/p$. As
  \eqref{eq:bad-block-to-dense-eq2} and \eqref{eq:bad-block-to-dense-eq3} imply
  $e_F(W_{t_2}, L_3) \geq \eps^{-19} \tilde n \log n$, we can apply
  Lemma~\ref{lem:everything}~\ref{lemma:medium-dont-shrink} with $L_3$ (as $U$)
  to obtain
  \[
    e_F(W_{t_3}, W_{t_4}) \geq e_F(L_3, W_{t_4}) \geq \frac{\alpha \tilde np}{4}
    |L_3|.
  \]
  If $|L_3| \geq (\alpha/32)\tilde n$ then we are done. In the remainder of the
  proof we show that such an assumption is actually true. Towards a
  contradiction assume that $|L_3| < (\alpha/32)\tilde n$. Observe that
  (recall,~\ref{w:class-size-bound} for $W_{t_1}$ in order to apply
  \ref{p:unif-deg})
  \begin{equation}\label{eq:bad-block-to-dense-eq4}
    |L_2| \osref{\ref{p:unif-deg}}{\geq} \frac{e_F(W_{t_1}, L_2)}{(1 +
    \eps)\tilde np} \osref{\eqref{eq:bad-block-to-dense-eq1}}{\geq} \frac{(1 -
    \eps)e_F(W_{t_1}, W_{t_2})}{(1 + \eps)\tilde np} \ge \frac{(1 - \eps)
    \eps^{-21} \log^2 n}{(1 + \eps)p} \geq \frac{\eps^{-19} \log^2 n}{p}.
  \end{equation}
  Let now $S \subseteq Y \subseteq W_{t_3}$ be sets defined as
  \[
    Y := \{v \in W_{t_3} : \deg_F(v, L_2) > \eps|L_2|p \} \qquad \text{and}
    \qquad S := Y \setminus L_3.
  \]
  If $|Y| \geq (\alpha/16)\tilde n$, then $|S| \geq (\alpha/32)\tilde n$, as we
  assumed $|L_3| < (\alpha/32)\tilde n$, and thus
  \[
   e_F(W_2, S)\ge e_F(L_2, S) \geq |S| \cdot \eps|L_2|p
   \osref{\eqref{eq:bad-block-to-dense-eq4}}{\geq} \frac{\eps\alpha}{32}\tilde n
   \cdot \eps^{-19} \log^2 n \geq \eps^{-17} \tilde n \log^2 n.
  \]
  Recall that all vertices in $W_{t_3} \setminus L_3$ have degree in $F$ at most
  $\eps^{-4}\log n/p$ into $W_{t_2}$. Therefore,
  Lemma~\ref{lem:everything}~\ref{lemma:small-expand} applied with $W_{t_2},
  W_{t_3}, W_{t_4}$ (as $W_1, W_2, W_3$), $S$ (as $U$), and $E_F(W_2, S)$ (as
  $F_{12}$) shows
  \[
    e_F(S, W_{t_4}) \ge \eps^{-4} e_F(W_{t_2}, S) \ge \eps^{-3}|S|
    |L_2|p\osref{\eqref{eq:bad-block-to-dense-eq4}}{\geq} \eps^{-3} \cdot
    \frac{\alpha}{16} (1 - 2\eps) e_F(W_{t_1}, W_{t_2}) > \eps^{-2} e_F(W_{t_1},
    W_{t_2}),
  \]
  which is a contradiction with our assumption that the $i$-th block is
  $(1/\eps)$-non-expanding. Therefore, $|Y| < (\alpha/16)\tilde n$. However, as
  by \ref{p:unif-density} and \eqref{eq:bad-block-to-dense-eq4} we know that
  there are at most $\eps^{-3} \log n/p$ vertices $v \in W_{t_3}$ with
  $\deg_G(v, L_2) \geq 2|L_2|p$, we then get
  \[
    e_F(L_2, W_{t_3}) \leq |Y| \cdot 2|L_2|p + \frac{\eps^{-3} \log n}{p} \cdot
    2 \tilde np + \tilde n \cdot \eps|L_2|p \leq \Big( \frac{\alpha}{8} + \eps
    + \eps \Big) |L_2| \tilde np < \frac{\alpha \tilde np}{4} |L_2|,
  \]
  which is a contradiction with the first inequality in
  \eqref{eq:bad-block-to-dense-eq2}. We conclude $|L_3| \geq (\alpha/32)\tilde
  n$ and the claim follows.
\end{proof}

\begin{proof}[Proof of~$\ref{claim:dense-blocks}$ in
Claim~\ref{claim:how-blocks-can-grow}: $i$-th block is $(1 +
4\sqrt\eps)$-expanding]
  Note that if both steps $2i$ and $2i + 1$ are $(1 + 4\sqrt\eps)$-expanding,
  then the statement follows directly from the definition of an expanding step.
  Thus, let us assume one of the two steps is $(1 + 4\sqrt\eps)$-non-expanding
  and let us denote that step with $t \in \{2i, 2i + 1\}$. Furthermore, let $t_2
  = f(t + 1)$ and let $t_1$ be the unique element from $\{f(t), t\} \setminus
  \{t_2\}$. By applying Lemma~\ref{lem:everything}~\ref{lemma:dense-doesnt-drop}
  with $W_{t_1}, W_{t_2}, W_{t + 1}$ (as $W_1, W_2, W_3$) and $E_F(W_{t_1},
  W_{t_2})$ (as $F_{12}$) we get $e_F(W_{f(t + 1)}, W_{t + 1}) \ge (1 -
  \sqrt\eps) e_F(W_{f(t)}, W_{t})$. From here we conclude
  \[
    e_F(W_{2i + 1}, W_{2i + 2}) \ge (1 + 4\sqrt\eps) (1 - \sqrt\eps) e_F(W_{2i
    -1}, W_{2i}) \ge (1 + 2\sqrt\eps) e_F(W_{2i -1}, W_{2i}),
  \]
  which is what we wanted to prove.
\end{proof}

\begin{proof}[Proof of~$\ref{claim:dense-blocks}$ in
Claim~\ref{claim:how-blocks-can-grow}: $i$-th block is $(1 +
4\sqrt\eps)$-non-expanding]
  Set $t_4 = 2i + 2$, $t_3 = 2i + 1$, $t_2 = f(2i + 1)$, and let $t_1$ be the
  unique element from $\{2i, 2i - 1\} \setminus \{t_2\}$. Let us define $L_2$
  and $L_3$ as
  \[
    L_2 = \{ v \in W_{t_2} : \deg_F(v, W_{t_1}) > \tilde np/3 \} \qquad
    \text{and} \qquad L_3 = \{ v \in W_{t_3} : \deg_F(v, W_{t_2}) > \tilde
    np/3 \}.
  \]

  Our first goal is to show $|L_3| \geq (2/3 + \alpha/4)\tilde n$. As the $i$-th
  block is $(1 + 4\sqrt\eps)$-non-expanding, we know that $e_F(W_{t_2}, W_{t_3})
  < (1 + 4\sqrt\eps) e_F(W_{t_1}, W_{t_2})$ and thus by
  Lemma~\ref{lem:everything}~\ref{lemma:dense-expand-or} applied with $W_{t_1},
  W_{t_2}, W_{t_3}$ (as $W_1, W_2, W_3$), $E_F(W_{t_1}, W_{t_2})$ (as $F_{12}$),
  and $32\sqrt\eps/\alpha$ (as $\mu$) we conclude
  \begin{equation}\label{eq:dense-bad-to-two-thirds-eq1}
    e_F(W_{t_1}, L_2) \ge (1 - 32\sqrt\eps/\alpha) e_F(W_{t_1}, W_{t_2}) \geq (1
    - \eps^{1/3}) e_F(W_{t_1}, W_{t_2}).
  \end{equation}
  Next, by Lemma~\ref{lem:everything}~\ref{lemma:large-to-all} with $W_{t_1},
  W_{t_2}, W_{t_3}$ (as $W_1, W_2, W_3$), $L_2$ (as $U$), and $E_F(W_{t_1},
  L_2)$ (as $F_{12}$) we get
  \begin{equation}\label{eq:dense-bad-to-two-thirds-eq2}
    e_F(L_2, W_{t_3}) \ge (1 - \eps^2) e_G(L_2, W_{t_3})
    \osref{\ref{p:good-deg-con}}\ge (1 - 3\eps)e_G(W_{t_1}, L_2) \geq (1 -
    3\eps) e_F(W_{t_1}, L_2).
  \end{equation}
  This together with \eqref{eq:dense-bad-to-two-thirds-eq1} implies
  \begin{align}\label{eq:dense-bad-to-two-thirds-eq3}
    e_F(W_{t_2}, W_{t_3}) \ge e_F(L_2, W_{t_3}) \ge (1 - 3\eps)e_F(W_{t_1}, L_2)
    \geq (1 - 2\eps^{1/3}) e_F(W_{t_1}, W_{t_2}).
  \end{align}
  Once again using the fact that the $i$-th block is $(1 +
  4\sqrt\eps)$-non-expanding, we get
  \begin{equation}\label{eq:dense-bad-to-two-thirds-eq4}
    \begin{aligned}
      e_F(W_{t_3}, W_{t_4}) & < (1 + 4\sqrt\eps)^2 e_F(W_{t_1}, W_{t_2}) \leq
      (1 + 4\sqrt\eps)^2 (1 - 2\eps^{1/3})^{-1} e_F(L_2, W_{t_3}) \\
      & \leq (1 + 4\eps^{1/3}) e_F(L_2, W_{t_3}).
    \end{aligned}
  \end{equation}
  Next, we apply Lemma~\ref{lem:everything}~\ref{lemma:dense-expand-or} with
  $W_{t_2}, W_{t_3}, W_{t_4}$ (as $W_1, W_2, W_3$), $E_F(L_2, W_{t_3})$ (as
  $F_{12}$), and $32\eps^{1/3}/\alpha$ (as $\mu$). We can do that since by
  \eqref{eq:dense-bad-to-two-thirds-eq3} and
  \eqref{eq:dense-bad-to-two-thirds-eq4} we know that $e_F(L_2, W_{t_3}) \geq
  \eps \tilde n^2 p$ and $e_F(W_{t_3}, W_{t_4}) < (1 + \mu\alpha/8) e_F(L_2,
  W_{t_3})$. Therefore, Lemma~\ref{lem:everything}~\ref{lemma:dense-expand-or}
  and \eqref{eq:dense-bad-to-two-thirds-eq2} imply
  \begin{equation}\label{eq:dense-bad-to-two-thirds-eq5}
    e_F(L_2, L_3) \ge \big( 1 - \tfrac{32\eps^{1/3}}{\alpha} \big) e_F(L_2,
    W_{t_3}) \geq (1 - \eps^{1/4}) e_F(L_2, W_{t_3})
    \osref{\eqref{eq:dense-bad-to-two-thirds-eq2}}\geq (1 - 2\eps^{1/4})
    e_G(L_2, W_{t_3}).
  \end{equation}
  Furthermore, \eqref{eq:dense-bad-to-two-thirds-eq5} and \ref{p:good-deg} show
  \[
    e_F(L_2, L_3) \geq (1 - 2\eps^{1/4}) e_G(L_2, W_{t_3}) \geq (1 -
    2\eps^{1/4}) \cdot |L_2| (2/3 + \alpha)\tilde np \geq (2/3 + \alpha/2) |L_2|
    \tilde np.
  \]
  From \ref{p:unif-density}, the fact that $|L_2| \geq (\eps/2)\tilde n$
  (follows from \eqref{eq:dense-bad-to-two-thirds-eq1} and \ref{p:unif-deg}),
  and $e_F(L_2, L_3) \ge (2/3 + \alpha/2) |L_2| \tilde np$, we obtain $|L_3| \ge
  (2/3 + \alpha/4) \tilde n$.

  Next, we define
  \[
    \calL = \{ v \in W_{2i + 1} : \deg_F(v, W_{2i + 2}) > \tilde np/3 \}
    \quad \text{and} \quad \calL' = \{ v \in W_{2i + 2} : \deg_F(v, W_{2i +
    1}) > \tilde np/3 \}.
  \]
  We aim to show that $|\calL|,|\calL'| \ge (2/3) \tilde n$. Since $|L_3| \ge
  (2/3 + \alpha/4)\tilde n$, we can apply
  Lemma~\ref{lem:everything}~\ref{lemma:large-stay-large} with $W_{f(2i + 1)},
  W_{2i + 1}, W_{2i + 2}$ (as $W_1, W_2, W_3$), and $L_3$ (as $U$) to conclude
  that $|\calL'| \ge (1 - \eps)\tilde n$, as desired. Similarly, applying
  Lemma~\ref{lem:everything}~\ref{lemma:large-to-all} with $W_{f(2i + 1)}, W_{2i
  + 1}, W_{2i + 2}$ (as $W_1, W_2, W_3$), $L_3$ (as $U$), and $E_{F}(W_{t_2},
  L_3)$ (as $F_{12}$) we get that there exists a $L_3' \subseteq L_3$ of size
  $|L_3'| \geq (1 - 3\eps^3)|L_3|$ such that for all $v \in L_3'$ we have
  $\deg_{F}(v, W_{2i + 2}) \geq (2/3 + \alpha/2) \tilde np$. Clearly $L_3'
  \subseteq \calL$ and $|L_3'|\ge (2/3) \tilde n$.

  Set $r_5 = 2i + 3$, $r_4 = f(2i + 3)$, and let $r_3$ be the unique element
  from $\{2i + 2, 2i + 1\} \setminus \{r_4\}$. Moreover, let $L_4$ and $L_5$ be
  defined as
  \[
    L_4 = \{ v \in W_{r_4} : \deg_F(v, W_{r_3}) > \tilde np/3 \} \qquad
    \text{and} \qquad L_5 = \{ v \in W_{r_5} : \deg_F(v, W_{r_4}) > \tilde
    np/3 \}.
  \]
  Note that, depending on $f(2i + 3)$, the set $L_4$ lies either in $W_{2i + 1}$
  or $W_{2i + 2}$. Thus, since $|\calL|, |\calL'| \geq (2/3)\tilde n$ and $L_4
  \in \{\calL, \calL'\}$, we have $|L_4| \geq (2/3)\tilde n$ as well.

  Having this at hand we can finally show $e_F(W_{2i + 3}, W_{2i + 4}) \ge (2/3
  + \alpha/2)\tilde n^2p$. By applying
  Lemma~\ref{lem:everything}~\ref{lemma:large-stay-large} with $W_{r_3},
  W_{r_4}, W_{r_5}$ (as $W_1, W_2, W_3$), and $L_4$ (as $U$) we get that $|L_5|
  \ge (1 - \eps)\tilde n$. Finally, we apply
  Lemma~\ref{lem:everything}~\ref{lemma:large-to-all} with $W_{r_4}, W_{r_5},
  W_{2i + 4}$ (as $W_1, W_2, W_3$), $L_5$ (as $U$), and $E_F(W_{r_4}, L_5)$ (as
  $F_{12}$) to obtain
  \[
    e_F(W_{2i + 3}, W_{2i + 4}) \ge (1 - \eps^2) e_G(L_5, W_{2i + 4})
    \osref{\ref{p:good-deg-con}}\ge (1 - \eps^2) \cdot |L_5| (2/3 +
    \alpha)\tilde n p \ge (2/3 + \alpha/2) \tilde n^2 p.
  \]
  This concludes the proof of Claim~\ref{claim:how-blocks-can-grow}.
\end{proof}

\section{Concluding remarks}\label{sec:conclusion}

In this paper we introduce the notion of $H$-resilience which measures the
fraction of $H$-copies touching a given vertex that an adversary may delete
without destroying a certain given property. We demonstrate the usefulness of
the definition by showing that the $K_3$-resilience of $\Gnp$ w.r.t.\ the
containment of the square of a Hamilton cycle is w.h.p.\ $5/9 \pm o(1)$. In
other words, the adversary needs to delete more than a $(5/9)$-fraction of the
triangles lying on a vertex in order to destroy all copies of $C_n^2$ in
$\Gnp$. Our result is optimal with respect to the constant $5/9$ and the density
$p$ up to logarithmic factors.

Having the notion of $H$-resilience at hand, one can ask for similar statements
for other (spanning) graph properties. Of particular interest is the question of
the $K_3$-resilience of $\Gnp$ with respect to the containment of a triangle
factor. Theorem~\ref{thm:K3-res-ub} shows that also here the resilience is at
most $5/9 + o(1)$. Moreover, as $C_n^2$ contains a triangle factor, provided $3
\mid n$, it follows that this is the correct one whenever $p \gg n^{-1/2} \log^3
n$. However, the threshold for the appearance of a triangle factor is
significantly lower than the threshold for the appearance of a $C_n^2$, cf.\ the
seminal result of Johansson, Kahn, and Vu~\cite{johansson2008factors}. In light
of this, we conjecture that the resilience variant of this result holds when $p$
is close to the threshold for having a $K_3$-factor.

An analogous construction as in Theorem~\ref{thm:K3-res-ub} shows that the
$K_r$-resilience for a $K_r$-factor is at most $1 - (1 - 1/r)^{r - 1}$. It is
thus tempting to conjecture that this value is also the $K_r$-resilience of
$\Gnp$ w.r.t.\ containment of a $K_r$-factor, provided that $p \gg n^{-2/r}
(\log n)^{1/e(K_r)}$, as well as $C_n^{r - 1}$, provided that $p \gg n^{-1/r}$.
The conjecture is true in the case when $p = 1$, as every graph with $(1 -
1/r)^{r - 1} \binom{n}{r - 1}$ copies of $K_r$ at each vertex must have a
minimum degree of at least $(r - 1)n/r$ and the statement thus follows from the
theorem of Hajnal and Szemer\'{e}di~\cite{hajnal1970proof} and
Theorem~\ref{thm:KSS}.

{\small \bibliographystyle{abbrv} \bibliography{references}}

\end{document}